\documentclass{article}

\usepackage[utf8]{inputenc}
\usepackage{amsmath,amsfonts,amssymb}
\usepackage{amsthm}
\usepackage[a4paper]{geometry}
\usepackage{color}
\usepackage{transparent}
\usepackage[dvipsnames]{xcolor}
\usepackage{graphicx}
\usepackage{subfigure}
\usepackage{tabularx,float}
\usepackage[hidelinks]{hyperref}
\usepackage{tabularx}
\usepackage{enumerate,mathrsfs}
\usepackage{mathtools}
\usepackage{enumitem}
\usepackage[titletoc]{appendix}
\usepackage{changes}
\usepackage{orcidlink}
\usepackage{titling}
\usepackage[misc]{ifsym}
\thanksmarkseries{arabic}
\newcommand{\stkout}[1]{\ifmmode\text{\sout{\ensuremath{#1}}}\else\sout{#1}\fi}
\setdeletedmarkup{\stkout{#1}}
\definechangesauthor[name={Sara-Viola Kuntz}, color=blue]{SVK}

\newcommand*{\dd}{\mathrm{d}}

\newtheorem{thm}{Theorem}[section]

\newtheorem{proposition}[thm]{Proposition}
\newtheorem{corollary}[thm]{Corollary}
\newtheorem{lemma}[thm]{Lemma}
\newtheorem{remark}[thm]{Remark}
\newtheorem{assumption}{Assumption}

\newcommand{\eps}{\varepsilon}

\newcommand{\R}{\mathbb{R}}

\newcommand{\me}{\mathrm{e}}

\newcommand\SJJ[1]{{\color{black}{#1}}} 
\newcommand\SV[1]{{\color{black}{#1}}} 
\newcommand\SVV[1]{{\color{black}{#1}}} 

\newcommand\SJJJ[1]{{\color{black}{#1}}} 
\newcommand\SVVV[1]{{\color{black}{#1}}} 

\usepackage{fancyhdr}
\newcommand\blfootnote[1]{%
	\begingroup
	\renewcommand\thefootnote{}\footnote{#1}%
	\addtocounter{footnote}{-1}%
	\endgroup
}

\title{\textbf{\hbox{Geometric Blow-Up for Folded Limit Cycle Manifolds} in Three Time-Scale Systems}}
\author{Samuel Jelbart \orcidlink{0000-0001-8539-320X}$^1$, Christian Kuehn \orcidlink{0000-0002-7063-6173}$^{1,2}$ \& Sara-Viola Kuntz \orcidlink{0009-0000-4611-9742}$^{1,2}$
	}

\date{
	\small{$^1$\textit{Technical University of Munich, School of Computation, Information and Technology, \\ Department of Mathematics, Boltzmannstraße 3, 85748 Garching, Germany} \\
	$^2$\textit{Munich Data Science Institute, Walther-von-Dyck-Straße 10, 85748 Garching, Germany } }\\[7mm]
	\large{November 17, 2023}
}

\begin{document}
	\maketitle
	
	\vspace{-5mm}

	\begin{abstract}
		\textit{Geometric singular perturbation theory} provides a powerful mathematical framework for the analysis of `stationary' multiple time-scale systems which possess a \textit{critical manifold}, i.e.~a smooth manifold of steady states for the limiting fast subsystem, particularly when combined with a method of desingularization known as \textit{blow-up}. The theory for `oscillatory' multiple time-scale systems which possess a limit cycle manifold instead of (or in addition to) a critical manifold is less developed, particularly in the non-normally hyperbolic regime. We \SJJ{use}
		the blow-up method 
		to analyse the global oscillatory transition near a regular folded limit cycle manifold in a class of three time-scale \SJJ{`semi-oscillatory'} systems with two small parameters. The systems considered behave like oscillatory systems as the smallest perturbation parameter tends to zero, and stationary systems as both perturbation parameters tend to zero. The additional time-scale structure is crucial for the applicability of the blow-up method, which cannot be applied directly to the two time-scale \SJJ{oscillatory} counterpart of the problem. Our methods allow us to describe the asymptotics and strong contractivity of all solutions which traverse a neighbourhood of the global singularity. Our \SJJ{main} results cover a range of different cases with respect to the relative time-scale of the angular dynamics and the parameter drift. \SJJ{We demonstrate the applicability of our results for systems with periodic forcing in the slow equation, in particular for a class of Li\'enard equations. Finally, we consider a toy model used to study tipping phenomena in climate systems with periodic forcing in the fast equation, which violates the conditions of our main results, in order to demonstrate the applicability of classical (two time-scale) theory for problems of this kind}.
	\end{abstract}
	
	\noindent {\small \textbf{Keywords:} Geometric singular perturbation theory, global fast-slow systems, three time-scale systems, geometric blow-up, limit cycle bifurcations \\[-2mm]}
	
	\noindent {\small \textbf{MSC2020:} 34D15, 34E10, 34E13, 34E15, 34E20, 34C45, 34C27 \\[-7mm]}
	
	\blfootnote{\hspace{-5mm}\Letter \; saraviola.kuntz@ma.tum.de (Sara-Viola Kuntz) \\[4pt]
	This version of the article has been accepted for publication, after peer review but is not the Version of Record and does not reflect post-acceptance improvements, or any
	corrections. The Version of Record is published in the \textit{Journal of Nonlinear Science}, and is available online at  \url{https://doi.org/10.1007/s00332-023-09987-x}.}

	\section{Introduction}
	\label{sec:Introduction}
	
    	Many physical and applied systems featuring multiple time-scale dynamics can be mathematically modelled by singularly perturbed systems of ordinary differential equations in the standard form
	\begin{equation}
		\label{eq:stnd_form}
		\begin{split}
			x' &= f(x,y,\eps) , \\
			y' &= \eps g(x,y,\eps) ,
		\end{split}
	\end{equation}
	where $x \in \R^m$, $y \in \R^n$, $(\cdot)' = \dd /\dd t$, $0 < \eps \ll 1$ is a small perturbation parameter and the functions $f,g : \R^m \times \R^n \times [0,\eps_0] \to \R^m \times \R^n$ are at least $C^1$-smooth. If the limiting system obtained from \eqref{eq:stnd_form} when $\eps \to 0$ possesses a \textit{critical manifold}, i.e.~if the set of equilibria $S = \{(x,y) : f(x,y,0) = 0 \}$ forms an $n$-dimensional submanifold of the phase space $\R^m \times \R^n$, then system \eqref{eq:stnd_form} can be analysed using a mathematical framework known as \textit{geometric singular perturbation theory (GSPT)} \cite{Fenichel1979,Jones1995,Kuehn2015,Wechselberger2019}. Typical GSPT analyses consist of two parts, depending on whether $S$ is \textit{normally hyperbolic}, i.e.~depending on whether the linearisation with respect to the fast variables $D_xf|_S$ when $\eps = 0$ is hyperbolic.
	
	The local dynamics near normally hyperbolic submanifolds of $S$ can be accurately described using \textit{Fenichel-Tikhonov theory} \cite{Fenichel1979,Tikhonov1952} (see also \cite{Jones1995,Kuehn2015,Wechselberger2019,Wiggins2013}), which ensures that solutions are well approximated by regular perturbations of singular trajectories constructed as concatenations of trajectory segments from the distinct limiting problems which arise when $\eps \to 0$ in system \eqref{eq:stnd_form} on the fast and slow time-scales $t$ and $\tau = \eps t$, respectively. However, Fenichel-Tikhonov theory breaks down in the non-normally hyperbolic regime. The dynamics near non-normally hyperbolic points or submanifolds of $S$, which generically correspond to bifurcations in the layer problem \eqref{eq:stnd_form}$|_{\eps = 0}$, can be studied using various techniques. A particularly powerful approach is the so-called \textit{blow-up method}, which was pioneered for fast-slow systems in \cite{Dumortier1996} and later in \cite{Krupa2001a,Krupa2001c,Krupa2001b}. In these works and many others (see \cite{Jardon2019b} for a recent survey), the authors showed that the loss of hyperbolicity at a non-normally hyperbolic point or submanifold $Q$ can often be `resolved' after lifting the problem to a higher dimensional space in which $Q$ is blown-up to a higher dimensional manifold $\mathcal Q$. Following a suitable \textit{desingularization}, which amounts to a singular transformation of time, a non-trivial vector field with improved hyperbolicity properties can be identified on $\mathcal Q$. This allows one to study the dynamics in the non-normally hyperbolic regime using classical dynamical systems methods like regular perturbation and center manifold theory. We refer to \cite{Gucwa2009,Kosiuk2011,Kosiuk2016,Kristiansen2021,Krupa2001a,Krupa2001c,Krupa2001b,KuehnSzmolyan2015,Szmolyan2001,Szmolyan2004} for seminal works in both applied and theoretical contexts, based on this combination of Fenichel-Tikhonov theory and the blow-up method.
	
	\
	
	The mathematical theory and methods described above are local in the sense that they apply specifically to fast-slow systems possessing an $n$-dimensional critical manifold $S$. A corresponding global theory in which the layer problem \eqref{eq:stnd_form}$|_{\eps = 0}$ possesses a limit cycle manifold in place of (or in addition to) a critical manifold is less developed. Following \cite{Ilyashenko1997}, we shall refer to the former class as \textit{stationary} fast-slow systems, and the latter class as \textit{oscillatory} fast-slow systems. A program for the development of a global GSPT which is general enough to encompass oscillatory fast-slow systems was initiated by J.~Guckenheimer in 1996 \cite{Guckenheimer1996}, however a number of key analytical and methodological obstacles to its development remain.
	
	One such obstacle concerns the development of Fenichel-Tikhonov theory for oscillatory fast-slow systems. A number of authors have made important contributions in this direction. Anosova showed that normally hyperbolic limit cycle manifolds in oscillatory fast-slow systems of the form \eqref{eq:stnd_form} persist as $O(\eps)$-close locally invariant manifolds for the perturbed system \cite{Anosova1999,Anosova2002}, similarly to the persistence of a normally hyperbolic critical manifold in Fenichel-Tikhonov theory. For oscillatory fast-slow systems with one slow variable, the persistence and contractivity properties of the center-stable/unstable manifolds have been described in detail in \cite{Jelbart2022a} using properties of the Poincar\'e map and a discrete GSPT framework developed therein. This study also yielded an asymptotic formula for the slow drift along the manifold, which agrees with the formula predicted by classical averaging theory \cite{Pontryagin1960}.
	
	Global theory for oscillatory fast-slow systems in the non-normally hyperbolic regime is less developed, despite the ubiquity of fast-slow extensions of global bifurcations in applications, for example in biochemical models exhibiting \textit{bursting} \cite{Bertram1995,Ermentrout2010,Guckenheimer2005,Rinzel1987}. \SJJ{Notable} exception\SJJ{s include the study of a dynamic saddle-node of cycles bifurcation in a variant of the FitzHugh-Nagumo equations in \cite{Kirillov2015}, and} the rigorous topological classification of so-called \textit{torus canards} in \cite{Vo2017}, which occur near folded cycle bifurcations in the layer problem in oscillatory fast-slow systems \cite{Vo2017}. See e.g.~\cite{Baspinar2021,Benes2011,Desroches2012,Kramer2008,Roberts2015} for more (predominantly numerical) important work on torus canards. The results in \cite{Vo2017} were obtained using averaging theory, Floquet theory and (stationary) GSPT. Indirect contributions to the non-normally hyperbolic theory for oscillatory fast-slow systems have also been made via the study of non-normally hyperbolic singularities in fast-slow maps, since these can be used to infer dynamical properties of corresponding limit cycle bifurcations (or fast-slow extensions thereof) in one greater dimension. The slow passage through a flip/period-doubling bifurcation (and even through an entire period-doubling cascade) was treated in \cite{Baesens1991,Baesens1995}, and further results on the slow passage through discrete transcritical, pitchfork and Neimark-Sacker/torus type bifurcations have been derived using non-standard analysis; we refer to the review paper \cite{Fruchard2009} and the references therein.
	
	\
	
	In general, the development of mathematical methods for handling non-hyperbolic dynamics in oscillatory fast-slow systems is complicated by the fact that local techniques like the blow-up method rely upon near-equilibrium properties possessed by stationary but not oscillatory fast-slow systems. The desingularization step in blow-up analyses, for example, relies upon a formal division of the blown-up vector field by zero. This step is crucial for obtaining a desingularized vector field with improved hyperbolicity properties, but it is only well-defined if the original (blown-up) vector field is in equilibrium wherever it is formally divided by zero. For oscillatory fast-slow systems of the form \eqref{eq:stnd_form}, a `typical' non-normally hyperbolic cycle has non-equilibrium dynamics, so it cannot be desingularized and the blow-up method does not apply. 
	
	The aim of this work is to show that stationary methods (in particular the blow-up method) which may not apply for oscillatory fast-slow systems, may be applicable in the study of oscillatory multiple time-scale systems with at least three distinct time-scales. Specifically, we consider systems of the form
	\begin{equation}
		\label{eq:original_system_1}
		\begin{split}
			r' &= f(r,\theta,y,\eps_1,\eps_2) , \\
			\theta' &= \eps_1 g(r,\theta,y,\eps_1,\eps_2) , \\
			y' &= \eps_2 h(r,\theta,y,\eps_1,\eps_2) ,
		\end{split}
	\end{equation}
	where $(r,\theta,y) \in \R_{\geq 0} \times \R / \mathbb Z \times \R$ are cylindrical coordinates, $0 < \eps_1, \eps_2 \ll 1$ are singular perturbation parameters, and $f,g,h : \R_{\geq 0} \times \R / \mathbb Z \times \R \times [0,\eps_{1,0}] \times [0,\eps_{2,0}] \to \SJJ{\R}$ 
	are sufficiently smooth for our purposes.
	Depending upon the relative magnitude of $\eps_1$ and $\eps_2$, system \eqref{eq:original_system_1} has either two or three distinct time-scales. Although we present results on a range of different cases, we are primarily interested in the case
	\begin{equation}
		\label{eq:eps_scalings}
		0 < \eps_2 \ll \eps_1 \ll 1 ,
	\end{equation}
	which defines a class of three time-scale \SJJ{\textit{semi-oscillatory}} systems that are in a certain sense `in between' the classes of stationary and oscillatory fast-slow systems described above. Heuristically, this is because under suitable assumptions (to be outlined in detail in Section \ref{sec:setting}), system \eqref{eq:original_system_1} is an oscillatory fast-slow system with respect to the limit $\eps_1 > 0, \ \eps_2 \to 0$, and a stationary fast-slow system with respect to the double singular limit $(\eps_1, \eps_2) \to (0,0)$~\cite{Kuehn2022}. It is worthy to note that multiple time-scale systems with more than two time-scales appear frequently in applications. The long term dynamics of a forced \SVVV{van der Pol} oscillator with three time-scales was studied as early as 1947 in \cite{Cartwright1947}. A theoretical basis for normally hyperbolic theory for stationary multiple time-scale systems with three or more time-scales has appeared more recently in e.g.~\cite{Cardin2014,Cardin2017,Kruff2019,Lizarraga2020b}. Progress has also been made in the non-normally hyperbolic setting, particularly via the study of three time-scale applications and `prototypical systems' inspired by applications; we refer to \cite{Maesschalck2014,Desroches2018,Jalics2010,Kaklamanos2022c,Kaklamanos2022b,Kaklamanos2022a,Krupa2008,Letson2017,Nan2015}.
	
	We present a detailed analysis of the `jump-type' transition near a non-normally hyperbolic cycle of regular fold type using geometric blow-up. More precisely, we assume that the limiting system~\eqref{eq:original_system_1}$_{\eps_1>0, \eps_2 = 0}$ undergoes a
	\SJJ{type of} folded cycle bifurcation under variation in $y$ \SJJ{which is common in applications with periodic forcing in the slow equation}. \SJJ{This global bifurcation is closely related to (and can in many ways be seen as an oscillatory analogue to)} the regular fold or jump point in stationary fast-slow systems, which has been studied \SJJ{in $\R^2$ and $\R^3$ in particular} using blow-up techniques in \SJJ{\cite{Krupa2001a,Szmolyan2004}}; see also \cite{Mishchenko1975} for a detailed treatment \SJJ{of the planar case} using classical asymptotic methods. After deriving a \SJJ{prototypical system by imposing a number of defining assumptions on system \eqref{eq:original_system_1},}
	we show that the blow-up method can be applied, even though a rigorous reduction to the stationary setting is not \SV{possible} \SJJ{due to angular coupling}.
	The formal division by zero which is necessary to obtain a desingularized vector field with improved hyperbolicity properties is possible if the time-scale associated to the rotation is sufficiently slow relative to the fast radial dynamics, i.e.~if $\eps_1$ is sufficiently small. 
	
	The blow-up analysis allows for the detailed characterisation of the transition map induced by the flow, including the asymptotic and contractivity properties of the transition undergone by solutions traversing the neighbourhood of the global singularity. \SJJ{If $\eps_2 / \eps_1 \sim 1$ or $\eps_2 / \eps_1 \gg 1$ (so that \eqref{eq:eps_scalings} is not satisfied), then system \eqref{eq:original_system_1} is a stationary fast-slow system. In this case, 
	the local dynamics can (for the most part) be described in detail using the results established for two time-scale systems in \cite{Krupa2001a,Mishchenko1975,Szmolyan2004}, after Taylor expansion about a given jump point. However, these results do not apply directly to the scaling regimes we consider which satisfy \eqref{eq:eps_scalings}, i.e.~they do not apply in the semi-oscillatory case. One important reason for this is that the singularity is `global' in the angular coordinate $\theta$. Consequently, it does not suffice to blow-up at a single point on the non-hyperbolic cycle. Rather, it is necessary to blow-up the entire non-hyperbolic cycle to a `torus of spheres' $S^1 \times S^2$. A similar approach is adopted in the geometric analysis of the periodically forced van der Pol equation in \cite{Burke2016}, however in our case, the leading order equations derived on the blown-up sphere may depend upon the angular variable $\theta$, which remains non-local. As a consequence, the local dynamics cannot be analysed with a straightforward adaptation of the arguments used to study the dynamics near a regular fold point/curve in \cite{Krupa2001a,Szmolyan2004}. Rather, new arguments are needed.
	We derive results for two different scaling regimes defined by $(\eps_1, \eps_2) = (\eps^\alpha, \eps^3)$, where $\alpha \in \{1,2\}$ and $0 < \eps \ll 1$. In each case, the size of the leading order term in the asymptotics} for the parameter drift in $y$ is shown to agree with the known results for the stationary regular fold point \SJJ{\cite{Krupa2001a,Szmolyan2004}. In contrast to the classical fold, however, the leading order coefficient is shown to depend on $\theta$, and we provide an explicit formula for this dependence in the case $\alpha = 2$}. We also provide asymptotics for the angular coordinate $\theta$ as a function of the initial conditions and small parameters, and an asymptotic estimate for the number of rotations about the $y$-axis over the course of the transition. The results obtained are shown to depend on the relative magnitude of $\eps_1$ and $\eps_2$ \SJJ{(i.e.~on $\alpha$)}, with the main qualitative difference pertaining to the asymptotic estimates for $\theta$ and the corresponding number of rotations.
	
	\SJJ{Finally, we apply our results in order to derive detailed asymptotic information near folded limit cycle manifolds of periodically forced Li\'enard equations, and we consider a simple model proposed in~\cite{Zhu2015} as a toy model for the study of tipping phenomena in climate systems. Our main results do not apply directly to the latter problem, due to periodic forcing in the fast equation. Rather, we aim to demonstrate with a partial but illustrative geometric analysis that problems of this kind can be treated using classical approaches based on established results for two time-scale systems.}
	
	\
	
    The manuscript is structured as follows. In Section \ref{sec:setting} we introduce defining assumptions and present the \SJJ{prototypical system} for which our main results are stated. The singular dynamics and geometry, which differ in \SJJ{different scalings},
    are presented in Section \ref{sub:geometry_and_dynamics_in_the_singular_limit}. The main results are presented and described in Section \ref{sec:main_results}, and the blow-up analysis and proof of the main results are presented in Section~\ref{sec:blowupfold}. \SJJ{The applications are treated in Section \ref{sec:applications}. Specifically, in Section \ref{sub:applications_1} we apply our main results to periodically forced Li\'enard equations, and in Section \ref{sub:applications_2} we consider the toy model for the study of tipping phenomena from \cite{Zhu2015} which cannot be treated directly with the results from Section \ref{sec:main_results}.} We conclude with a summary and outlook in Section \ref{sec:summary}.

	\section{Assumptions and Setting}
	\label{sec:setting}
	
    	We consider $C^{k}$-smooth multiple time-scale systems in the general form \eqref{eq:original_system_1}, restated here for convenience:
	\begin{equation}
		\label{eq:original_system}
		\begin{split}
			r' &= f(r,\theta,y,\eps_1,\eps_2) , \\
			\theta' &= \eps_1 g(r,\theta,y,\eps_1,\eps_2) , \\
			y' &= \eps_2 h(r,\theta,y,\eps_1,\eps_2) ,
		\end{split}
	\end{equation}
	where \SJJ{$k \in \mathbb N$ will be assumed to be `sufficiently large' throughout}, $(\cdot)' = \dd/\dd t$, the variables are given in cylindrical coordinates $(r,\theta,y) \in \R_{\geq 0} \times \R / \mathbb Z \times \R$, and $\eps_1, \eps_2$ are singular perturbation parameters satisfying $0 < \eps_1, \eps_2 \ll 1$. \SJJ{Note also that smoothness implies that $f,g,h$ are $1$-periodic in $\theta$.}
	System \eqref{eq:original_system} evolves on either two or three time-scales, depending on whether the ratio $\eps_1 / \eps_2$ is asymptotically large, constant or small. The setup and defining assumptions presented below are primarily motivated by the case $0 < \eps_2 \ll \eps_1 \ll 1$, for which system \eqref{eq:original_system} in a certain sense `intermediate' between stationary and oscillatory fast-slow systems. \SJJ{We shall refer to this as the \textit{semi-oscillatory} case.} Since we also consider other possibilities, however, we leave the exact relation between $\eps_1$ and $\eps_2$ unspecified for now. It suffices to observe that the forward evolution of a generic initial condition is characterised by radial motion on the fast time-scale $t$, angular motion on a time-scale $\tau_{\eps_1} = \eps_1 t$, and vertical `parameter drift' on a time-scale $\tau_{\eps_2} = \eps_2 t$.
	
	In the following we impose a number of defining conditions in terms of the limiting oscillatory fast-slow system obtained in the singular limit $\eps_1 > 0, \ \eps_2 \to 0$, i.e.~on
	\begin{equation}
		\begin{split}
			\label{eq:layer_delta}
			r' &= f(r,\theta,y,\eps_1,0) , \\
			\theta' &= \eps_1 g(r,\theta,y,\eps_1,0) , \\
			y' &= 0 .
		\end{split}
	\end{equation}
	We remark that the singular limit $\eps_1 > 0, \ \eps_2 \to 0$ is only `natural' \SJJ{in the semi-oscillatory case, i.e.~if $\eps_1 / \eps_2 \gg 1$ so that} the rotation is fast relative to the parameter drift.
	
	\SJJ{
	\begin{assumption}
		\label{ass:1_new}
		\textup{(Existence of a limit cycle for $\eps_1 > 0, \ \eps_2 \to 0$).} There exist a constant $\eps_{1,0} > 0$ and a constant $v > 0$ such that system \eqref{eq:layer_delta} has a circular limit cycle
		\[
		S_0^c = \{(v, \theta, 0) : \theta \in \R / \mathbb Z\} .
		\]
		More precisely, we assume that
		\begin{equation}
		    \label{eq:ass_1_conds}
	    	f(v, \theta, 0, \eps_1, 0) = 0 , \qquad 
	    	g(v, \theta, 0, \eps_1, 0) \neq 0 ,
		\end{equation}
		for all $\theta \in \R / \mathbb Z$ and $\eps_1 \in (0, \eps_{1,0})$.
	\end{assumption}
	
	\begin{remark}
	\label{rem:ass_1_restrictions}
	    The assumption that the limit cycle $S_0^{\textup{c}}$ is circular and in particular the zero condition on~$f$ in \eqref{eq:ass_1_conds}, is natural in applications with \SJJJ{small-amplitude} external periodic forcing \SJJJ{(amplitudes of $\mathcal O(\varepsilon_2)$ or smaller).} 
	    However, it rules out applications with \SJJJ{`large'} periodic forcing in the fast equation for~$r$. One reason for imposing such a restriction is that problems of the latter kind can often be treated using classical theory for two time-scale systems. This is demonstrated for a particular application in Section~\ref{sub:applications_2}.
	\end{remark}
		
	We shall be interested in the dynamics in a neighbourhood of $S_0^c$. This motivates the introduction of the signed radius variable
	\[
	\tilde r \coloneqq r - v ,
	\]
	in which case $S_0^c = \{ (\tilde r, \theta, y) : \tilde r = 0, \theta \in \R / \mathbb Z, y = 0 \}$. We assume without loss of generality that
	\[
	g(v,\theta,0,\eps_1,0) > 0 ,
	\]
	divide the right-hand side of system \eqref{eq:original_system} by $g(\tilde r + v, \theta, y, \eps_1, \eps_2)$, i.e.~we apply a time-dependent transformation satisfying $\dd \tilde t = g(r,\theta,y,\eps_1,\eps_2) \dd t$ (which is positive in a sufficiently small toroidal or tubular neighbourhood \SV{$\mathcal{V}$} of $S_0^{\textup{c}}$),
	and rewrite the system in $(\tilde r, \theta, y)$-coordinates in order to obtain
	\begin{equation}
		\label{eq:transformed_system}
		\begin{split}
			\tilde r' &= F(\tilde r, \theta, y, \eps_1, \eps_2) , \\
			\theta' &= \eps_1 , \\
			y' &= \eps_2 H(\tilde r, \theta, y, \eps_1, \eps_2) ,
		\end{split}
	\end{equation}
	where by a slight abuse of notation the dash now denotes differentiation with respect to $\tilde t$, and}
	\[
		\SJJ{F(\tilde r, \theta, y, \eps_1, \eps_2) \coloneqq \frac{f(\tilde r + v, \theta, y, \eps_1, \eps_2) }{g(\tilde r + v, \theta, y, \eps_1, \eps_2)} , 
		\qquad
		H(\tilde r, \theta, y, \eps_1, \eps_2) \coloneqq \frac{h(\tilde r + v, \theta, y, \eps_1, \eps_2)}{g(\tilde r + v, \theta, y, \eps_1, \eps_2)} .}
	\]
	
	\begin{remark}
	    \label{rem:desingularisation}
		In \SJJ{the derivation of system \eqref{eq:transformed_system} and} in many of the proofs below we make use of transformations of time which are formulated in terms of differentials, e.g.~\SJJ{$\dd \tilde t = g(r,\theta,y,\eps_1,\eps_2) \dd t$}. Strictly speaking, such an `transformation' only defines $\tilde t$ as a unique function of $t$ up to an additive constant. Since we are interested in the behaviour of solutions to autonomous ODEs, which are invariant under time translation, this additive constant can be set to zero without loss of generality.
	\end{remark}
	
	\SJJ{
	We are interested in three time-scale systems \SV{\eqref{eq:transformed_system}} which feature a regular fast-slow fold of limit cycles with respect to the partial singular limit $\eps_1 > 0, \eps_2 \to 0$. Necessary and sufficient conditions for this to occur can be given in terms of the Poincar\'e map induced on the transversal section $\Delta$ obtained by intersecting the \SV{toroidal/}tubular neighbourhood $\mathcal V \supset S_0^{\textup{c}}$ with the half-plane defined by $\theta = 0$ (decreasing the size of $\mathcal V$ if necessary). 
	
	\begin{proposition}
		\label{prop:Poincare_map}
		System \eqref{eq:transformed_system} with $\eps_1 \in (0, \eps_{1,0})$ fixed and $0 < \eps_2 \ll 1$ induces a Poincar\'e map \mbox{$P : \Delta \to \Delta$} given by
		\begin{equation}
			\label{eq:Poincare_map}
			P(\tilde r, y, \eps_1, \eps_2) = 
			\begin{pmatrix}
				P_{\tilde r}(\tilde r, y, \eps_1, \eps_2) \\
				P_y(\tilde r, y, \eps_1, \eps_2)
			\end{pmatrix}
			=
			\begin{pmatrix}
				\tilde r \\
				y
			\end{pmatrix}
			+ \eps_1^{-1}
			\begin{pmatrix}
				\int_0^1 F(\tilde r(\theta), \theta, y(\theta), \eps_1, 0) \, \dd\theta + O(\eps_2)  \\
				\eps_2 \int_0^1 H(\tilde r(\theta), \theta, y(\theta), \eps_1, 0) \, \dd\theta + O(\eps_2^2) 
			\end{pmatrix} .
		\end{equation}
	\end{proposition}
	
	\begin{proof}
		Since
	    \[
	    \frac{\dd \tilde r}{\dd \theta} = \eps_1^{-1} F(\tilde r, \theta, y, \eps_1, \eps_2) , \qquad 
	    \frac{\dd y}{\dd \theta} = \eps_1^{-1} \eps_2 H(\tilde r, \theta, y, \eps_1, \eps_2) ,
	    \]
	    we have
	    \[
	    P(\tilde r, y, \eps_1, \eps_2) = 
	    \begin{pmatrix}
	        P_{\tilde r}(\tilde r, y, \eps_1, \eps_2) \\
	        P_{y}(\tilde r, y, \eps_1, \eps_2)
	    \end{pmatrix}
	    = 
	    \begin{pmatrix}
	        \tilde r \\
	        y
	    \end{pmatrix}
	    + \eps_1^{-1}
	    \begin{pmatrix}
	        \int_0^1 F(\tilde r(\theta), \theta, y(\theta), \eps_1, \eps_2) \, \dd \theta  \\
	        \eps_2 \int_0^1 H(\tilde r(\theta), \theta, y(\theta), \eps_1, \eps_2) \, \dd \theta
	    \end{pmatrix} .
	    \]
	    The result follows after Taylor expanding about $\eps_2 = 0$.
	\end{proof}
	
	The defining conditions for $S_0^c$ to be a regular fast-slow fold of cycles are as follows:
	\begin{equation}
		\label{eq:map_sn_conds}
		P_{\tilde r}(0, 0, \eps_1, 0) = 0, \qquad 
		\frac{\partial P_{\tilde r}}{\partial \tilde r}(0, 0, \eps_1, 0) = 1,
	\end{equation}
	together with
	\begin{equation}
		\label{eq:map_sn_conds_2}
		\frac{\partial^2 P_{\tilde r}}{\partial \tilde r^2}(0, 0, \eps_1, 0) \neq 0 , \qquad
		\frac{\partial P_{\tilde r}}{\partial y}(0, 0, \eps_1, 0) \neq 0, \qquad 
		\int_0^1 H(\tilde r(\theta), \theta, y(\theta), \eps_1, 0) \, \dd\theta \neq 0 ,
	\end{equation}
	for all $\eps_1 \in (0,\eps_{1,0})$. The conditions in \eqref{eq:map_sn_conds}-\eqref{eq:map_sn_conds_2} are in 1-1 correspondence with the defining conditions for a fold bifurcation in the 1D `layer map' $\tilde r \mapsto \tilde r + P_{\tilde r}(\tilde r, y, \eps_1, 0)$ (with $y$ as a bifurcation parameter), see e.g.~\cite[Ch.~4.3]{Kuznetsov2013}, except for the integral condition, which can be viewed as the analogue of the slow regularity condition on the fast-slow regular fold point in planar continuous time systems in \cite{Krupa2001a,Kuehn2015}.
	
	\SJJJ{\begin{remark}
	    The defining conditions for a regular fast-slow fold of cycles in \eqref{eq:map_sn_conds}-\eqref{eq:map_sn_conds_2} do not depend on the specific form of the Poincar\'e map in \eqref{eq:Poincare_map}. Nevertheless, we have chosen to state Proposition \ref{prop:Poincare_map} prior to the conditions in \eqref{eq:map_sn_conds}-\eqref{eq:map_sn_conds_2} in order to clarify the interpretation of the slow regularity condition, i.e.~the integral expression in \eqref{eq:map_sn_conds_2}, which is not common in the literature. This condition can be viewed as a condition on the `reduced map'; we refer to \cite{Jelbart2022a,Jelbart2023} for details.
	\end{remark}}
	
	In the following, we shall actually assume stronger conditions that are sufficient but not necessary for a regular fast-slow fold of cycles, instead of those in \eqref{eq:map_sn_conds}-\eqref{eq:map_sn_conds_2}. Similarly to Assumption \ref{ass:1_new}, these conditions are expected to be satisfied in applications \SJJJ{with small-amplitude} external periodic forcing; 
	see again Remark \ref{rem:ass_1_restrictions}.
	
	\begin{assumption}
		\label{ass:2_new}
		\textup{(Sufficient conditions for $S_0^c$ to be a regular folded cycle)} The following sufficient (but not necessary) conditions for a fast-slow fold of cycles are satisfied by system \eqref{eq:transformed_system}:
			\begin{equation}
				\label{eq:SN_conds1}
				F(0,\theta,0,\eps_1,0) = 0 , \qquad
				\frac{\partial F}{\partial \tilde r}(0,\theta,0,\eps_1,0) = 0 ,
			\end{equation}
			together with
			\begin{equation}
				\label{eq:SN_conds2}
				\frac{\partial^2 F}{\partial \tilde r^2}(0,\theta,0,\eps_1,0) \neq 0 , \qquad
				\frac{\partial F}{\partial y}(0,\theta,0,\eps_1,0) \neq 0 , \qquad 
				H(0, \theta, 0, \eps_1, 0) \neq 0 ,
			\end{equation}
			for all $\theta \in \R / \mathbb Z$ and $\eps_1 \in \SV{(}0,\eps_{1,0})$.
	\end{assumption}
	
	It is straightforward to verify that the conditions in \eqref{eq:SN_conds1}-\eqref{eq:SN_conds2} are sufficient to ensure that the Poincar\'e map \eqref{eq:Poincare_map} satisfies the fold conditions \eqref{eq:map_sn_conds}-\eqref{eq:map_sn_conds_2}. In particular, \eqref{eq:SN_conds1}-\eqref{eq:SN_conds2}} are directly analogous to the defining conditions for the (stationary) regular fold point in \cite{Krupa2001a}\SJJ{, except that we require them to hold for $\theta$-dependent functions.

	Assumptions \ref{ass:1_new}\SJJJ{-}\ref{ass:2_new} and the implicit function theorem imply that system \eqref{eq:layer_delta} has a two-dimensional manifold of regular limit cycles
	\[
	S_0 = \left\{ (r,\theta,y) \in \mathcal V : F(\tilde r, \theta, y, \eps_1,0) = 0 \right\} ,
	\]
	where \SVVV{$\mathcal V \subseteq \R_{\geq 0} \times \R / \mathbb Z \times \R$} is the \SV{toroidal/}tubular neighbourhood about $S_0^{\textup{c}}$ introduced above. In other words, system \eqref{eq:transformed_system} (and by a simple computation also system \eqref{eq:original_system})} is an oscillatory fast-slow system with respect to the (partial) singular limit $\eps_1 > 0, \ \eps_2 \to 0$. 

	The combination of signs taken by the various non-zero terms in \SJJ{\eqref{eq:SN_conds2}} determines the orientation of the bifurcation. In the following we assume without loss of generality that \SJJ{
	\begin{equation}
		\label{eq:signs}
			\frac{\partial^2 F}{\partial \tilde r^2}(0,\theta,0,\eps_1,0) > 0 , \qquad
			\frac{\partial F}{\partial y}(0,\theta,0,\eps_1,0) < 0 , \qquad
			H(0,\theta,0,\eps_1,0) < 0 ,
	\end{equation}
	}which are consistent with a `jump-type' orientation in forward time; see Figures \ref{fig:guglhupf_v1} and \ref{fig:guglhupf_v2}. Based on Assumptions \SJJ{\ref{ass:1_new}-\ref{ass:2_new}} and these sign conventions, it suffices to work with the \SJJ{simplified system} provided in the following result.
	
	\begin{proposition}
		\label{prop:normal_form}
		Let Assumptions \SJJ{\ref{ass:1_new}-\ref{ass:2_new}} be satisfied \SJJ{and assume that $0 < \eps_1, \eps_2 \ll 1$. In a sufficiently small \SV{toroidal/}tubular neighbourhood about $S_0^c$, which we continue to denote by $\mathcal V$, system \eqref{eq:transformed_system} can be written as
		\begin{equation} \label{eq:thetacoupled}
			\begin{aligned}
				\tilde r' &= - a(\theta) y + b(\theta) \tilde r^2 + \mathcal R_r(\tilde r,\theta,y,\eps_1,\eps_2), \\
				\theta' &= \eps_1 , \\
				y' &= \eps_2 ( - c(\theta) + \mathcal R_y(\tilde r,\theta,y,\eps_1,\eps_2)) ,
			\end{aligned}  
		\end{equation}
		where the functions $a(\theta), b(\theta), c(\theta)$ are positive, $1$-periodic and smooth, and the higher order terms satisfy}
		\[\SV{\mathcal R_r(\tilde{r},\theta,y,\eps_1,\eps_2) = \mathcal{O}(\tilde{r}^3,y^2,\tilde{r}y,\eps_1 \tilde{r}^2, \eps_1 y, \eps_2) , \qquad
		\mathcal R_y(\tilde{r},\theta,y,\eps_1,\eps_2) = \mathcal{O}(\tilde{r},y,\eps_1,\eps_2) .}
		\]
	\end{proposition}
	
	\begin{proof}
		Consider system \SJJ{\eqref{eq:transformed_system}} under Assumptions \SJJ{\ref{ass:1_new}-\ref{ass:2_new}}. 
		Taylor expanding about $\tilde r = y = \eps_2 = 0$, we obtain
		\begin{equation*}
			\begin{split}
				\tilde r' &= - f_1(\theta,\eps_1) y + f_2(\theta,\eps_1) \tilde r^2 + \mathcal O (\tilde r^3, y^2, \tilde r y, \eps_2) , \\
				\theta' &= \eps_1 , \\ 
				y' &= \eps_2 ( - h_0(\theta, \eps_1) + \mathcal O(\tilde r, y, \eps_2) ) ,
			\end{split}
		\end{equation*}
		where $f_1(\theta,\eps_1)$, $f_2(\theta,\eps_1)$ \SV{and $h_0(\theta,\eps_1)$} are smooth\SJJ{, $1$-periodic in $\theta$} and positive (this follows \SV{from} 
		the sign conventions \SJJ{in} \eqref{eq:signs}). 
		\SJJ{Expanding $f_1(\theta,\eps_1)$, $f_2(\theta,\eps_1)$ and $h_0(\theta,\eps_1)$ about $\eps_1 = 0$ and setting 
			\[
			a(\theta) \coloneqq f_1(\theta,0), \qquad 
			b(\theta) \coloneqq f_2(\theta,0), \qquad 
			c(\theta) \coloneqq h_0(\theta,0),
			\]
		yields system \eqref{eq:thetacoupled}.}
	\end{proof}
	
	System \eqref{eq:thetacoupled} is \SJJ{related to} the local normal form for the regular fold point in \cite{Krupa2001a} \SJJ{by a simple variable rescaling} if $\eps_1 = 0$, i.e.~if $\theta$ is fixed. For $\eps_1 > 0$, however, the dependence on the angular variable appears via the functions \SJJ{$a(\theta)$, $b(\theta)$, $c(\theta)$} and the higher order terms $\SV{\mathcal R_r(\tilde{r},\theta,y,\eps_1,\eps_2)}$ and $\SV{\mathcal R_y(\tilde{r},\theta,y,\eps_1,\eps_2)}$. \SJJ{In what follows, we drop the tilde notation on $\tilde r$ and} work with system \eqref{eq:thetacoupled} 
	for the remainder of this manuscript.
	
	\SVVV{
	\begin{remark}
	    In Proposition \ref{prop:normal_form} \SJJJ{we assert that the functions $a(\theta)$, $b(\theta)$ and $c(\theta)$ are `smooth'. A more precise statement would be to say that they are $C^k$-smooth, since the system obtained after Taylor expansion is precisely as smooth as the original system \eqref{eq:transformed_system}. Since we will not be interested in smoothness per se, we shall adopt a similar terminology throughout for simplicity, i.e.~by `smooth' we shall mean sufficiently smooth for the validity of our methods (e.g.~Taylor expansions).} 
	\end{remark}
	}
	
	\SJJJ{
	\begin{remark}
	    The equation for $\tilde r$ in system \eqref{eq:thetacoupled} can be further simplified after setting \SVVV{$\widehat r = \sqrt{b(\theta) / a(\theta)} \, \tilde r$}. This leads to
	    \[
	    \widehat r' = \iota(\theta) \left( -y + \widehat r^2 \right) + \SVVV{\frac{1}{2}}\frac{a(\theta) b'(\theta) - a'(\theta) b(\theta)}{\SVVV{\iota(\theta)^2}} \eps_1 \widehat r + h.o.t. ,
	    \]
	    where $a' := \partial a / \partial \theta$, $b' := \partial b / \partial \theta$ \SVVV{and $\iota(\theta) := \sqrt{a(\theta)b(\theta)}$}. The `price' of this simplification, however, is that the $\mathcal O(\eps_1\SVVV{\widehat r} )$ term also appears in the leading order terms in the blow-up analysis in later sections. For this reason, we continue to work with the formulation in \eqref{eq:thetacoupled}.
	\end{remark}}
	
	\SJJJ{
	\begin{remark}
	    If $\eps_1 \sim \eps_2$ or $\eps_1 \ll \eps_2$, then the $\theta$-variable is `slow enough' to validate the Taylor expansion of system \eqref{eq:thetacoupled} about a fixed point $(0,\theta^\ast,0) \in S_0^{\textup{c}}$, i.e.~in this case, one can also Taylor expand in the angular coordinate $\theta$. This allows for a subsequent transformation into the simpler local normal form near a fold curve in \cite[Lem.~3]{Szmolyan2004}, thereby showing that the dynamics in these cases are governed by the well-known result for two time-scale systems in \cite[Thm.~1]{Szmolyan2004}. In the semi-oscillatory case of interest with $\eps_2 \ll \eps_1$, however, $\theta$ is fast relative to $y$, and varies over the entire domain $\R / \mathbb Z$ as solutions approach $S_0^{\textup{c}}$. As a consequence, one cannot Taylor expand the $\theta$ coordinate, and transformation to the local normal form in \cite{Szmolyan2004} is not possible.
	\end{remark}
	}
	
	\begin{remark} \label{rem:figures}
	    \SV{In order to sketch geometric objects like $S_0$ in the upcoming figures, we choose the positive, 1-periodic and smooth functions \mbox{$a(\theta) = 2 + \sin(4\pi\theta)$} and $b(\theta) = 5 + \cos(2\pi\theta-1)$. For constant functions $a$ and $b$, the figures including the $\theta$-coordinate would be rotationally symmetric.}
	\end{remark}

	\subsection{Geometry and Dynamics in the Singular Limit}
	\label{sub:geometry_and_dynamics_in_the_singular_limit}
	
	We turn now to the singular geometry and dynamics of \SJJ{system \eqref{eq:thetacoupled}}. Taking the double singular limit $(\eps_1, \eps_2) \to (0,0)$ yields the \textit{layer problem}\SJJ{
	\begin{equation} \label{eq:thetacoupled_layer}
		\begin{aligned}
			r' &= -a(\theta) y + b(\theta) r^2 + \mathcal R_r(r,\theta,y,0,0) , \\
			\theta' &= 0, \\
			y' &= 0,
		\end{aligned}  
	\end{equation}
	}which has a two-dimensional \textit{critical manifold}
	\begin{equation*}
		S_0 \coloneqq \left\{(r,\theta,\varphi_0(r,\theta)) : r \in I_r, \theta \in \R / \mathbb Z \right\} ,
	\end{equation*}
	where \SV{$I_r \coloneqq (-r_0,r_0)$} for a small but fixed \SV{$r_0 > 0$} and
	\[
	\SJJ{y = \varphi_0(r,\theta) = \frac{b(\theta)} {a(\theta)} r^2 + \mathcal O(r^3)}
	\]
	solves the equation \SV{$F(r,\theta,y,0,0) = -a(\theta)y + b(\theta)r^2 + \mathcal R_r(r,\theta,y,0,0) = 0$} locally via the implicit function theorem.
	
	The stability of $S_0$ with respect to the fast radial dynamics is determined by the unique non-trivial (i.e.~not identically zero) eigenvalue of the linearisation, namely\SJJ{
	\begin{equation*}
		\lambda(r,\theta) = \frac{\partial }{\partial r} \left( - a(\theta) y + b(\theta) r^2 + \mathcal R_r(r,\theta,y,0,0) \right) \bigg|_{S_0} = 2 b(\theta) r + \mathcal{O}(r^2) .
	\end{equation*}
	}It follows that the critical manifold has the structure $S_0 = S_0^{\textup{a}} \cup S_0^\textup{c} \cup S_0^{\textup{r}}$, where
	\[
	S_0^\textup{a} = \{(r,\theta,\varphi_0(r,\theta)) \in S_0 : r < 0 \}, \qquad 
	S_0^\textup{r} = \{(r,\theta,\varphi_0(r,\theta)) \in S_0 : r > 0 \},
	\]
	are normally hyperbolic and attracting/repelling, respectively \SJJ{(assuming $r_0 > 0$ is sufficiently small)}. The circular set $S_0^\textup{c} = \{(0,\theta,0) \in S_0\}$, which corresponds to the regular folded cycle in Assumption \SJJ{\ref{ass:2_new}}, is non-normally hyperbolic. The situation is sketched in Figures \ref{fig:guglhupf_v1} and \ref{fig:guglhupf_v2}.

	\begin{figure}[t!]
		\centering
		\begin{minipage}[t]{0.47\textwidth}
			\centering
			\vspace{0pt}
			\includegraphics[width=\textwidth]{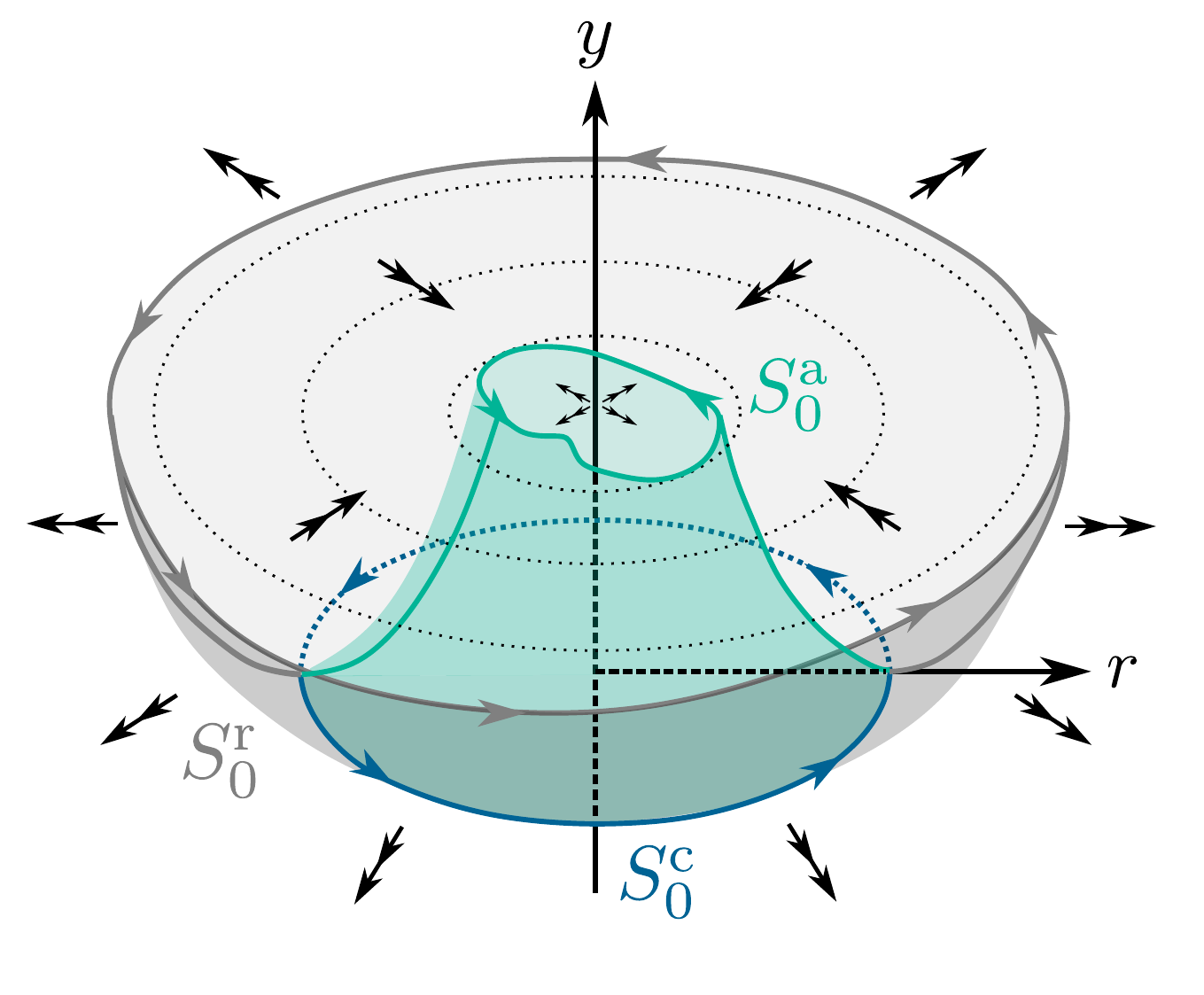}
		    \caption{Singular geometry and dynamics in case~(C1). Fast and slow dynamics are depicted (here and throughout) by double and single arrows, respectively. The attracting and repelling normally hyperbolic submanifolds of the critical manifold $S_0$, denoted by $S_0^\textup{a}$ and $S_0^{\textup{r}}$, are shown in shaded turquoise and gray, respectively, \SV{and sketched for the particular choice of $a(\theta)$ and $b(\theta)$ defined in Remark \ref{rem:figures}}. The non-normally hyperbolic folded cycle $S_0^\textup{c}$ is shown in blue. The reduced flow in case (C1) is periodic, i.e.~$y$ is a parameter and $S_0$ is foliated by limit cycles of period $\tau_{\eps_1} = 1$.}
			\label{fig:guglhupf_v1}
		\end{minipage}
		\begin{minipage}[c]{0.04\textwidth}
			\textcolor{white}{.}
		\end{minipage}
		\begin{minipage}[t]{0.47\textwidth}
			\centering
			\vspace{0pt}
			\includegraphics[width=\textwidth]{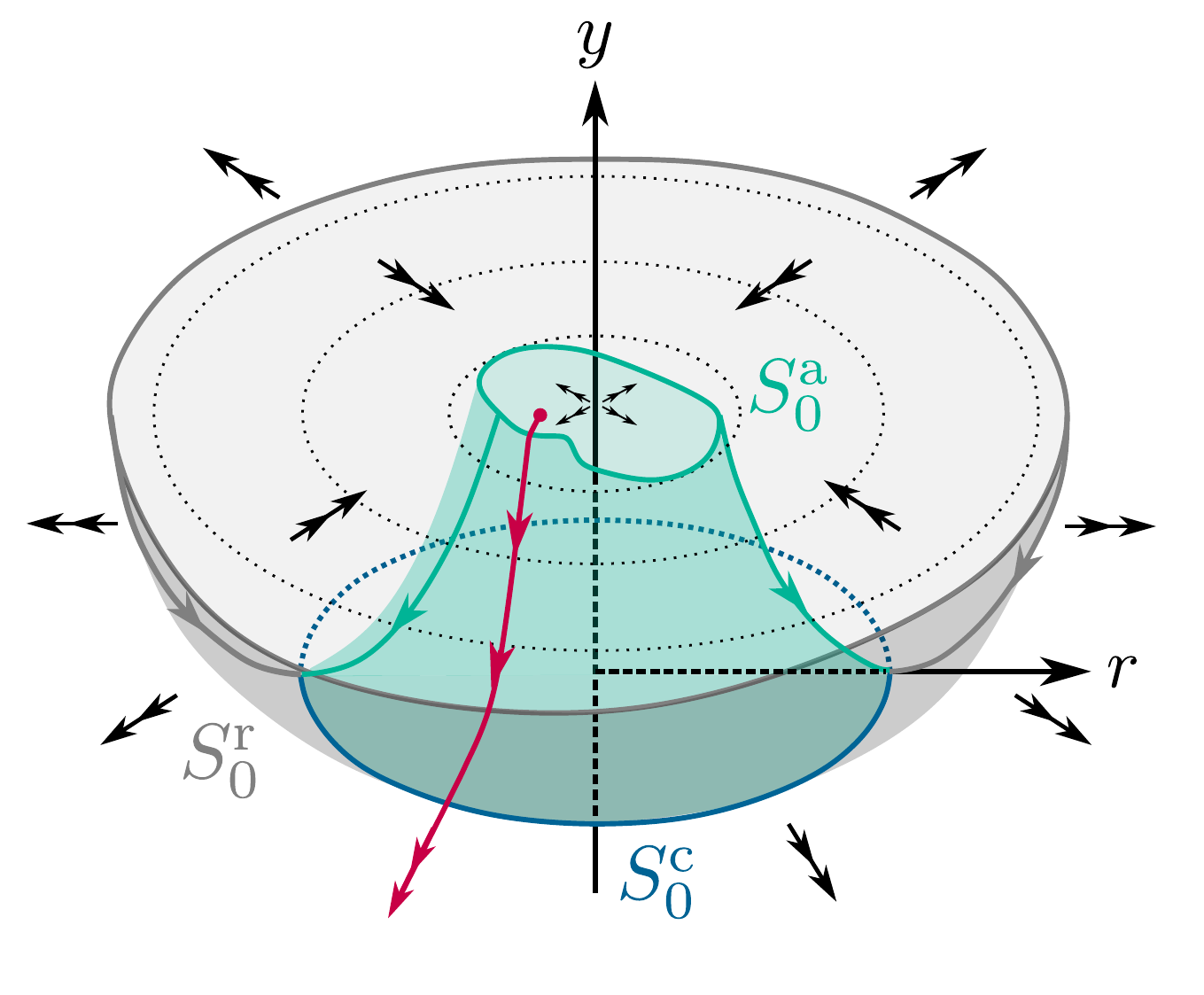}
			\caption{Singular geometry and dynamics in case (C3) \SV{sketched for the particular choice of $a(\theta)$ and $b(\theta)$ in Remark \ref{rem:figures}}. The dynamics is distinguished from cases (C1) and (C2) by the reduced flow on $S_0$. In this case, $\theta$ is a parameter and singular orbits (concatenations of solution segments of layer and reduced problem) are contained within constant angle planes $\{\theta = const.\}$. An example of such an orbit is sketched in red.}
			\label{fig:guglhupf_v2}
		\end{minipage}
	\end{figure}
	
	\
	
	The dynamics and geometry for the layer problem \eqref{eq:thetacoupled_layer} do not depend upon the relative magnitude of $\eps_1$ and $\eps_2$. The reduced dynamics on $S_0$, however, are expected to differ significantly depending on the size of $\eps_1 / \eps_2$. As noted already in Section \ref{sec:Introduction}, we consider three distinct possibilities:
	\begin{enumerate}
		\item[(C1)] Angular dynamics are fast relative to the parameter drift, i.e.~
		\[
		\frac{\eps_1}{\eps_2} \gg 1 .
		\]
		\item[(C2)] Angular dynamics occur on the same time-scale as the parameter drift, i.e.~there is a constant $\sigma > 0$ such that
		\[
		\frac{\eps_1}{\eps_2} = \sigma .
		\]
		\item[(C3)] Angular dynamics are slow relative to the parameter drift, i.e.~
		\[
		\frac{\eps_1}{\eps_2} \ll 1 .
		\]
	\end{enumerate}
	A different reduced problem is obtained on $S_0$ in each case. We briefly consider each case in turn.

	\subsubsection*{The Reduced Problem in Case (C1)}
	
	In this case
	we rewrite system \SJJ{\eqref{eq:thetacoupled}} on the slow time-scale $\tau_{\eps_1} = \eps_1 t$. This leads to\SJJ{
	\begin{equation} \label{eq:rescale_lambda}
		\begin{aligned}
			\eps_1 \dot{r} &= - a(\theta) y + b(\theta) r^2 + \mathcal R_r(r,\theta,y,\eps_1,\eps_2) , \\ 
			\dot{\theta} &= 1, \\
			\dot{y} &= \frac{\eps_2}{\eps_1}(- c(\theta) + \mathcal R_y(r,\theta,y,\eps_1,\eps_2)) ,  
		\end{aligned}  
	\end{equation}
	}where the dot denotes differentiation with respect to the slow time $\tau_{\eps_1}$. Since $\eps_1 / \eps_2 \gg 1$, \SJJJ{we first take $\eps_2 \to 0$, and then $\eps_1 \to 0$ (in that order). This leads to} the reduced problem\SJJ{
	\begin{equation} \label{eq:reduced_C1}
		\begin{aligned}
			0 &= - a(\theta) y + b(\theta) r^2 + \mathcal R_r(r,\theta,y,0,0), \\
			\dot{\theta} &= 1, \\
			\dot{y} &= 0 .
		\end{aligned}  
	\end{equation}
	In this case,
	$S_0$ is foliated by limit cycles of period $\tau_{\eps_1} = 1$, i.e.~$t=1/\eps_1$. \SJJJ{An expression for the vector field on $S_0$, expressed in the $(r,\theta)$-coordinate chart, can be obtained by differentiating the constraint $y = \varphi_0(r,\theta)$ with respect to $\tau_{\eps_1}$ and rearranging terms. We obtain}
	\[
	\begin{split}
		\dot r &= \frac{a'(\theta) b(\theta) - a(\theta) b'(\theta)}{2 a(\theta) b(\theta)} r + \mathcal O(r^2) , \\
		\dot \theta &= 1 ,
	\end{split}
	\]
	where $a' \coloneqq \partial a / \partial \theta$ and $b' \coloneqq \partial b / \partial \theta$.} This case is sketched in Figure \ref{fig:guglhupf_v1}.

	\subsubsection*{The Reduced Problem in Case (C2)}
	
	To obtain the reduced problem in case (C2) we may write system \SJJ{\eqref{eq:thetacoupled}} on either time-scale $\tau_{\eps_1}$ or $\tau_{\eps_2}$, which are related via $\tau_{\eps_1} = \sigma \tau_{\eps_2}$. Writing the system on the $\tau_{\eps_2}$ time-scale leads to\SJJ{
	\begin{equation} \label{eq:rescale_delta}
		\begin{aligned}
			\eps_2 \dot{r} &= - a(\theta) y + b(\theta) r^2 + \mathcal R_r(r,\theta,y,\eps_1,\eps_2) , \\ 
			\dot{\theta} &= \frac{\eps_1}{\eps_2}, \\
			\dot{y} &= - c(\theta) + \mathcal R_y(r,\theta,y,\eps_1,\eps_2), 
		\end{aligned}  
	\end{equation}
	}where this time, the dot denotes differentiation with respect to the slow time $\tau_{\eps_2}$. Taking the double singular limit and using the fact that $\eps_1 / \eps_2 = \sigma$ leads to the reduced problem\SJJ{
	\begin{equation} \label{eq:reduced_C2}
		\begin{aligned}
			0 &= - a(\theta) y + b(\theta) r^2 + \mathcal R_r(r,\theta,y,0,0) , \\ 
			\dot{\theta} &= \sigma, \\
			\dot{y} &= - c(\theta) + \mathcal R_y(r,\theta,y,0,0) .
		\end{aligned}  
	\end{equation}
	}\SJJJ{In contrast to the limiting equations \eqref{eq:reduced_C1} in case (C1), none of the variables become parameters in system~\eqref{eq:reduced_C2}}. 
	This is natural because in case (C2), system \SJJ{\eqref{eq:thetacoupled}} only has two (as opposed to three) time-scales. \SJJ{In the $(r,\theta)$-coordinate chart, the dynamics are determined by
	\begin{equation}
	    \label{eq:reduced_VF_C2}
    	\begin{split}
    	\dot r &= - \frac{a(\theta) c(\theta)}{2 b(\theta) r} + \mathcal O (1) , \\
    	\dot \theta &= \sigma .
    	\end{split}
	\end{equation}
	In particular, solutions reach $r = 0$ (and therefore $S_0^{\textup{c}}$) in finite time.}

	\subsubsection*{The Reduced Problem in Case (C3)}
	
	In this case we rewrite system \SJJ{\eqref{eq:thetacoupled}} on the slow time-scale $\tau_{\eps_2} = \eps_2 t$, thereby obtaining system \eqref{eq:rescale_delta}. Since $\eps_1 / \eps_2 \ll 1$,  \SJJJ{we first take $\eps_1 \to 0$, and then $\eps_2 \to 0$ (in that order). This leads to} 
	the reduced problem\SJJ{
	\begin{equation} \label{eq:reduced_C3}
		\begin{aligned}
			0 &= - a(\theta) y + b(\theta) r^2 + \mathcal R_r(r,\theta,y,0,0) , \\ 
			\dot{\theta} &= 0 , \\
			\dot{y} &= - c(\theta) + \mathcal R_y(r,\theta,y,0,0) .
		\end{aligned}  
	\end{equation}
	}This time, the angular variable $\theta$ is the slow variable to be considered as a parameter in system \eqref{eq:reduced_C2}. \SJJ{In particular, the dynamics on $S_0$ can be represented by the 1-parameter family of ODEs
	\[
	\dot r = - \frac{a c}{2 b r} + \mathcal O (1) ,
	\]
	where $a = a(\theta)$, $b = b(\theta)$ and $c = c(\theta)$ are constants paramaterised by $\theta \in \R / \mathbb Z$. Similarly to case~(C2), solutions reach $S_0^{\textup{c}}$ in finite time. Moreover,} singular orbits obtained as concatenations of layer and reduced orbit segments are contained within constant $\theta$ planes. As a result, the \SJJ{singular geometry and dynamics in case (C3) is equivalent to the singular geometry and dynamics of}
	the normal form \SJJ{for} the \SJJ{planar} regular fold point in \cite{Krupa2001a}. This case is sketched in Figure \ref{fig:guglhupf_v2}.
	
	\
	
	For $0 < \eps_1, \eps_2 \ll 1$, Fenichel-Tikhonov theory implies that compact submanifolds of the normally hyperbolic critical manifolds $S_0^\textup{a}$ and $S_0^\textup{r}$ persist as $\mathcal{O}(l(\eps_1,\eps_2))$-close locally invariant \textit{slow manifolds} $S_{l(\eps_1,\eps_2)}^\textup{a}$ and $S_{l(\eps_1,\eps_2)}^\textup{r}$, respectively \cite{Fenichel1979,Jones1995,Kuehn2015,Wechselberger2019,Wiggins2013}, where we write $l(\eps_1,\eps_2) \coloneqq \max \{ \eps_1, \eps_2 \}$ in order to keep the discussion general, i.e.~so that we need not distinguish between cases (C1), (C2) and (C3). Our goal is to describe the extension of the attracting slow manifold $S_{l(\eps_1,\eps_2)}^\textup{a}$ through a neighbourhood of the non-hyperbolic cycle $S_0^\textup{c}$ corresponding to the regular folded cycle in system \SJJ{\eqref{eq:thetacoupled}}.
	
	\begin{remark}
		In \cite{Cardin2017}, the authors extend GSPT for a class of multiple time-scale systems with $n \geq 3$ time-scales which feature `nested critical manifolds'. A requirement for the application of this theory to system \eqref{eq:thetacoupled} is that the reduced problem on $S_0$ has a one-dimensional critical manifold. This condition is not satisfied in any case (C1), (C2) or (C3) because none of the reduced problems \eqref{eq:reduced_C1}, \eqref{eq:reduced_C2} and \eqref{eq:reduced_C3} have equilibria in the neighbourhood of interest (i.e.~close to $r = y = 0$).
	\end{remark}

	\section{Main Results}
	\label{sec:main_results}
	
	We now state and describe our main results. In order to distinguish the different cases (C1), (C2) and (C3), we scale $\eps_1$ and $\eps_2$ by a single small parameter \SVVV{$0<\eps \ll 1$}. Main results are stated and proved for \SJJ{system \eqref{eq:thetacoupled}}
	with $\eps_1 = \eps^\alpha$ and $\eps_2 = \eps^3$, i.e.~for the system\SJJ{
	\begin{equation} \label{eq:systemcontinuous}
		\begin{aligned}
			r' &= - a(\theta) y + b(\theta) r^2 + \widetilde{\mathcal R}_r(r,\theta,y,\varepsilon), \\
			\theta' &= \varepsilon^\alpha, \\
			y' &= \varepsilon^3(-c(\theta)+\widetilde{\mathcal R}_y(r,\theta,y,\varepsilon)),
		\end{aligned}  
	\end{equation}
	}where \SVVV{the functions $a(\theta), b(\theta), c(\theta)$ are positive, $1$-periodic and smooth and the higher order terms satisfy}
	\begin{equation*}
	    \widetilde{\mathcal R}_r(r,\theta,y,\varepsilon) = \mathcal{O}(r^3,y^2,ry, \eps^\alpha r^2, \eps^\alpha y, \varepsilon^3), \qquad \widetilde{\mathcal R}_y(r,\theta,y,\varepsilon) = \mathcal{O}(r,y,\varepsilon^\alpha, \eps^3).
	\end{equation*} 
	The different cases (C1), (C2) and (C3) are obtained for different values of the scaling parameter \SJJ{$\alpha \in \mathbb N_+$} as follows:
	\begin{itemize}
		\item Case (C1)$^\ast$: $\alpha = 1$;
		\item Case (C1): \SJJ{$\alpha = 2$};
		\item Case (C2): $\alpha = 3$;
		\item Case (C3): \SJJ{$\alpha \geq 4$}.
	\end{itemize}
	Note that we have introduced an additional case (C1)$^\ast$. This case is dynamically distinct from the others, but it is not distinguished in Section \ref{sec:setting} because the geometry and dynamics in the double singular limit are the same as for case (C1).
	
	\begin{remark}
		\label{rem:KS_eps}
		The choice to write $\eps^3$ instead of $\eps$ in the equation for $y$ is made a-posteriori in order to avoid fractional exponents in the proofs. Comparisons with known results for the stationary regular fold point in \SJJ{\cite{Krupa2001a,Szmolyan2004}} are possible via the simple relation \SJJ{$\eps_{\textup{KSW}} = \eps^3$}, where we denote by \SJJ{$\eps_{\textup{KSW}}$} the small parameter in \cite{Krupa2001a} \SJJ{and/or \cite{Szmolyan2004}}. Similar observations motivated the use of a cubic exponent for the small parameter in other works involving folded singularities, e.g.~in \cite{Nipp2013,Nipp2009}.
	\end{remark}
	
	Our aim is to describe the forward evolution of initial conditions in an annular entry section\SJJ{
	\begin{equation}
		\label{eq:Delta_in}
		\Delta^{\textup{in}} \coloneqq \left\{(r,\theta,R^2) : r \in [\beta_-, \beta_+] , \theta \in \R / \mathbb Z \right\} ,
	\end{equation}
	}where $R$ \SV{is a small positive constant} and \SJJ{$\beta_- < \beta_+ < 0$} \SV{are two negative constants} \SJJ{chosen such that $S_0^\textup{a} \cap \Delta^{\textup{in}} \subset \Delta^{\textup{in}} \cap \{ r \in (\beta_-, \beta_+) \}$}. We track solutions of system \eqref{eq:systemcontinuous} up to their intersection with the cylindrical exit section
	\begin{equation}
		\label{eq:Delta_out}
		\Delta^{\textup{out}} \coloneqq \left\{(R,\theta,y) : \theta \in \R / \mathbb Z , y \in [-y_0,y_0] \right\} ,
	\end{equation}
	for a small positive constant $y_0 > 0$. The critical manifold $S_0$, the Fenichel slow manifolds \SJJ{(denoted now by $S_\eps^\textup{a}$ and $S_\eps^\textup{r}$)} and the entry/exit sections are visualized in the $(r,y)$-plane in Figure \ref{fig:setting_2d} and in the three-dimensional space in Figure \ref{fig:setting_3d}.
	
	\

	We now state the main result, which characterises the dynamics of the map \SJJ{$\pi^{(\alpha)}: \Delta^{\textup{in}} \rightarrow \Delta^{\textup{out}}$} induced by the flow of system \eqref{eq:systemcontinuous} \SJJ{for each $\alpha \in \mathbb N_+$, i.e.~in all four} cases (C1)$^\ast$, (C1), (C2) and (C3). 
	
	\begin{thm}
		\label{thm:main} 
		Consider system \eqref{eq:systemcontinuous} with \SJJ{fixed $\alpha \in \mathbb N_+$}. There exists an $\eps_0 > 0$ such that for all $\varepsilon \in (0, \varepsilon_0]$, the map \SJJ{$\pi^{(\alpha)}: \Delta^{\textup{in}} \rightarrow \Delta^{\textup{out}}$} is \SJJJ{well-defined} with the following properties\textup{:}
		\begin{enumerate}[label=(\alph*)]
		    \item \SJJ{\textup{(Extension of $S_\eps^{\textup{a}}$).} There exists a function $h_y^{(\alpha)} : \R / \mathbb Z \times (0,\eps_0\SJJJ{]} \to \R$ which is smooth and $1$-periodic in $\theta$ such that  $\pi^{(\alpha)}(S_\eps^{\textup{a}} \cap \Delta^{\textup{in}}) = \{ (R, \theta, h_{y}^{(\alpha)}(\theta,\varepsilon)) : \theta \in \R / \mathbb Z \}$ is a smooth, closed curve.}
			\item \textup{(Asymptotics).} \SJJ{$\pi^{(\alpha)}$ has the form
			\[
			\pi^{(\alpha)} : (r, \theta, R^2) \mapsto 
			\left(R, h_\theta^{(\alpha)}(r,\theta,\eps) , h_y^{(\alpha)}(h_\theta^{(\alpha)}(r,\theta,\eps),\eps) + h_\textup{rem}^{(\alpha)}(r,\theta,\eps) \right) ,
			\]
			where
			\[
			h_y^{(\alpha)}(h_\theta^{(\alpha)}(r,\theta,\eps),\eps) = \mathcal O(\eps^2) , \qquad 
			h_\textup{rem}^{(\alpha)}(r,\theta,\eps) = \mathcal O\left(\me^{-\kappa / \eps^3} \right) ,
			\]
			and $\kappa > 0$ is a constant. In particular we have that $h_{\theta}(r,\theta,\varepsilon) = \tilde h_{\theta}(r,\theta,\varepsilon) \mod 1$, where
			\begin{equation}
			\label{eq:theta_asymptotics}
	    		\tilde h_{\theta}^{(\alpha)}(r,\theta,\varepsilon) = 
	    		\begin{cases}
	    		    \SV{\theta + \frac{R^2}{c_0} \eps^{-2}  + \mathcal O( \ln \eps), } & \alpha = 1, \\
		    	    \SV{\theta + \frac{R^2}{c_0} \eps^{-1}  + \mathcal O(\eps \ln \eps), } & \alpha = 2, \\
		    	    \psi(\theta) + \mathcal O(\eps^3 \ln \eps), & \alpha = 3, \\
		    	    \theta + \mathcal O(\eps^3 \ln \eps), & \alpha \geq 4 ,
			    \end{cases}
			\end{equation}
			and
			
			\begin{equation}
			    \label{eq:y_asymptotics}
    			h_y^{(\alpha)}(\tilde \theta,\eps) =
    			\begin{cases}
    			    \mathcal O(\eps^2), & \alpha = 1, \\
    			    - \left( \frac{c(\tilde \theta)^2}{a(\tilde \theta) b(\tilde \theta)} \right)^{1/3} \Omega_0 \eps^2 + \mathcal O(\eps^3 \ln \eps) , & \alpha \geq 2 .
    			\end{cases}
			\end{equation}
	
			Here $c_0 \coloneqq \int_0^1 c(\theta) \, \dd \theta > 0$ is the mean value of $c$ over a single period, \mbox{$\psi(\theta) = \theta + \frac{b(\theta)}{a(\theta) c(\theta)} R^2 + \mathcal O(R^3)$} is induced by the reduced flow of system \eqref{eq:reduced_VF_C2} with $\sigma = 1$ up to $r = 0$,}
			and the constant $\Omega_0 > 0$ is the smallest positive zero of $J_{-1/3} ( 2z^{3/2} /3 ) + J_{1/3}( 2z^{3/2} / 3)$ where $J_{\pm 1/3}$ are Bessel functions of the first kind.
			
			\item \textup{(Strong contraction).} The $y$-component of \SJJ{$\pi^{(\alpha)}(r,\theta,R^2)$} is a strong contraction with respect to $r$. More precisely,
			\begin{equation*}
				\SJJ{\frac{\partial}{\partial r} \left(h_y^{(\alpha)}(h_\theta^{(\alpha)}(r,\theta,\eps),\eps) + h_\textup{rem}^{(\alpha)}(r,\theta,\eps) \right) = 
				\mathcal{O} \left( \me^{-\kappa / \eps^3} \right) .}
			\end{equation*}
		\end{enumerate}
	\end{thm}

	\begin{figure}[t!]
		\centering
		\begin{minipage}[t]{0.47\textwidth}
			\centering
			\vspace{8.3mm}
			\includegraphics[width=\textwidth]{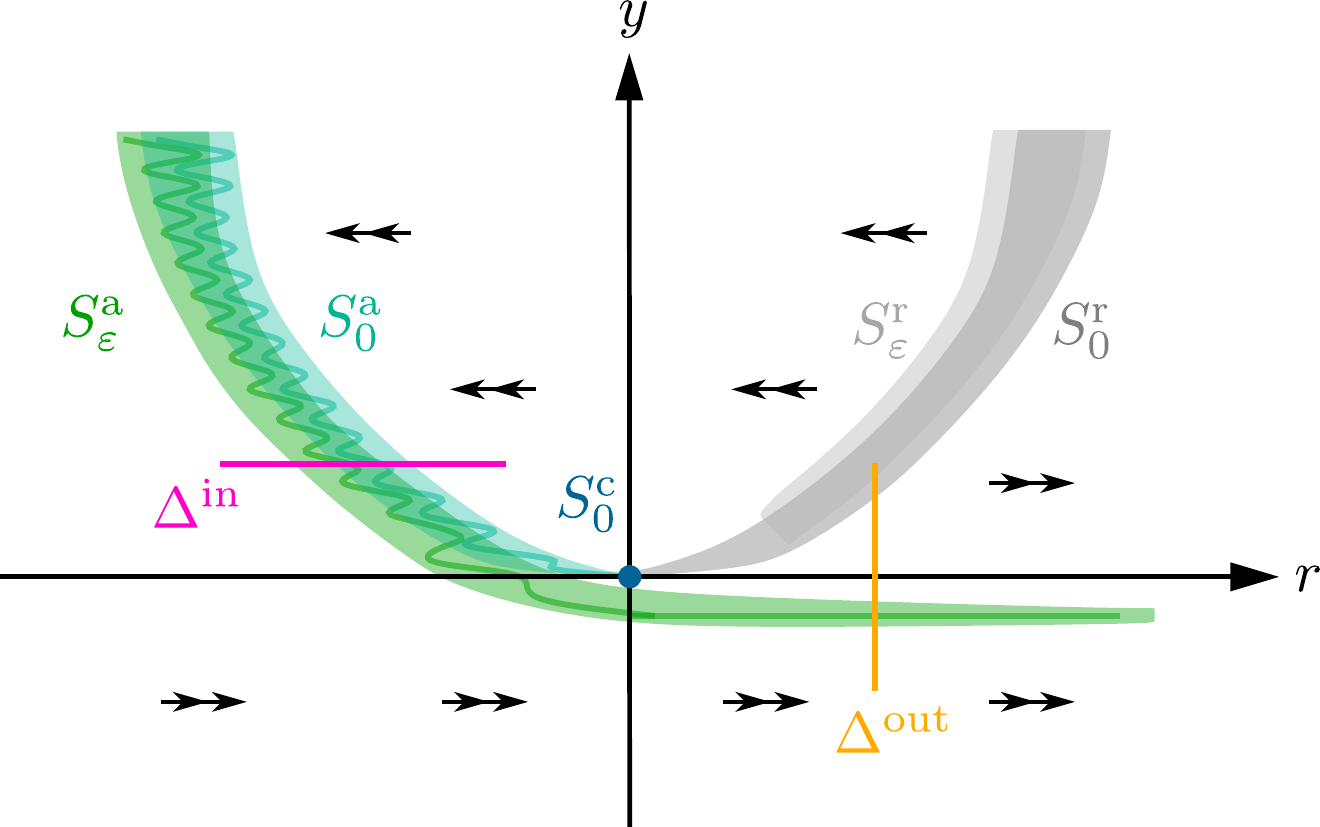}
			\vspace{8mm}
			\caption{Projected geometry and dynamics in the $(r,y)$-plane, as described by Theorem \ref{thm:main}. 
			\SV{The critical manifold and its submanifolds are sketched in colours consistent with earlier figures for the particular choice of $a(\theta)$ and $b(\theta)$ defined in Remark \ref{rem:figures}}. The entry, exit sections $\Delta^{\textup{in}}$, $\Delta^{\textup{out}}$ (magenta, orange) and the (extended) Fenichel slow manifolds $\SV{S^{\textup{a}}_{\eps}}$, $\SV{S^{\textup{r}}_{\eps}}$ are also shown \SV{as shaded regions} in green and light gray, respectively. \SV{For $S_\eps^\text{a}$ and $S_0^\text{a}$ additionally a sample trajectory for a fixed $\theta$ initial condition is shown.}
			}
			\label{fig:setting_2d}
		\end{minipage}
		\begin{minipage}[t]{0.04\textwidth}
			\textcolor{white}{.}
		\end{minipage}
		\begin{minipage}[t]{0.47\textwidth}
			\centering
			\vspace{0mm}
			\includegraphics[width=\textwidth]{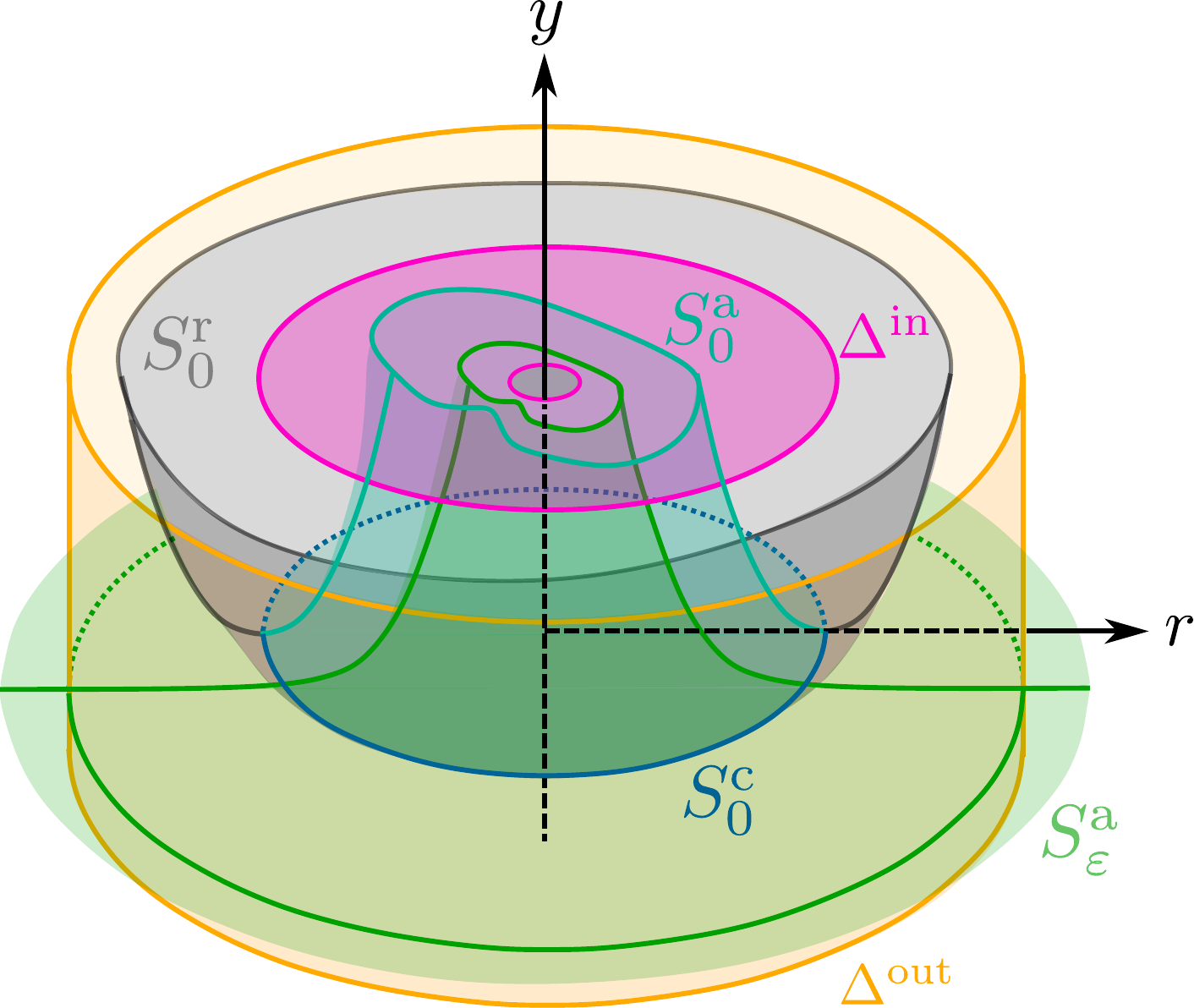}
			\vspace{0mm}
			\caption{The extension of the attracting slow manifold \SV{$S^{\textup{a}}_{\eps}$} (again in shaded green) as described by Theorem \ref{thm:main}, in the full $(r,\theta,y)$-space.
			The entry, exit sections $\Delta^\textup{in}$, $\Delta^\textup{out}$ as well as the critical manifold and its submanifolds are also shown, with the same colouring as in Figure \ref{fig:setting_2d} \SV{for the particular choice of $a(\theta)$ and $b(\theta)$ in Remark \ref{rem:figures}}. The intersection $\SV{\pi^{(\alpha)}(S_{\eps}^{\textup{a}}} \cap \Delta^{\textup{in}}) \subset \Delta^{\textup{out}}$ is topologically equivalent to a circle (shown in dark green), and $O(\eps^2)$-close to the plane $\{y = 0\}$ in the Hausdorff distance. The specific behavior of solutions, which depends on $\alpha$, is not shown (see however Figures \ref{fig:flow_0}, \ref{fig:flow_1}, \ref{fig:flow_2} and \ref{fig:flow_3}).}
			\label{fig:setting_3d}
		\end{minipage}
	\end{figure}

	Theorem \ref{thm:main} characterises the asymptotic behaviour of solutions and the extension of the attracting Fenichel slow manifold $S_\eps^{\textup{a}}$ through a neighbourhood of the regular folded cycle in \SJJ{system} \eqref{eq:systemcontinuous}. The geometry and dynamics for all four cases (C1)$^\ast$, (C1), (C2) and (C3) are sketched in Figures \ref{fig:flow_0}, \ref{fig:flow_1}, \ref{fig:flow_2} and~\ref{fig:flow_3}, respectively. The proof \SJJ{is broken into two parts, depending on whether or not $\alpha \in \{1,2\}$. A detailed proof based on the blow-up method will be given for the cases $\alpha \in \{1,2\}$ in Section \ref{sec:blowupfold}. Recall that these are the primary cases of interest, since they correspond to semi-oscillatory dynamics. If $\alpha \geq 3$, minor adaptations of the proof for case $\alpha = 2$ can be applied. Moreover, in this case, one can show directly that $S_0^{\textup{c}}$ is a closed regular fold curve. From here, Assertions (a)-(c) can be proven directly using established results (or a straightforward adaptation thereof) on the passage past a fold curve in 2-fast 1-slow systems, specifically \cite[Thm.~1]{Szmolyan2004}. We shall therefore omit the details of these cases. Note that for $\alpha \geq 4$ in particular, $h_\theta^{(\alpha)}(r, \theta, \eps) \sim \theta$ as $\eps \to 0$, which implies that the leading order coefficient in the expansion for $h_y^{(\alpha)}(\theta,\eps)$ is also constant in $\theta$. This is expected, since the dynamics for $\alpha \geq 4$ should resemble the dynamics of the stationary fold point considered in \cite{Krupa2001a} up to a slight perturbation.}

	\begin{figure}[t!]
		\centering
		\begin{minipage}[c]{0.47\textwidth}
			\centering
			\vspace{0pt}
			\includegraphics[width=\textwidth]{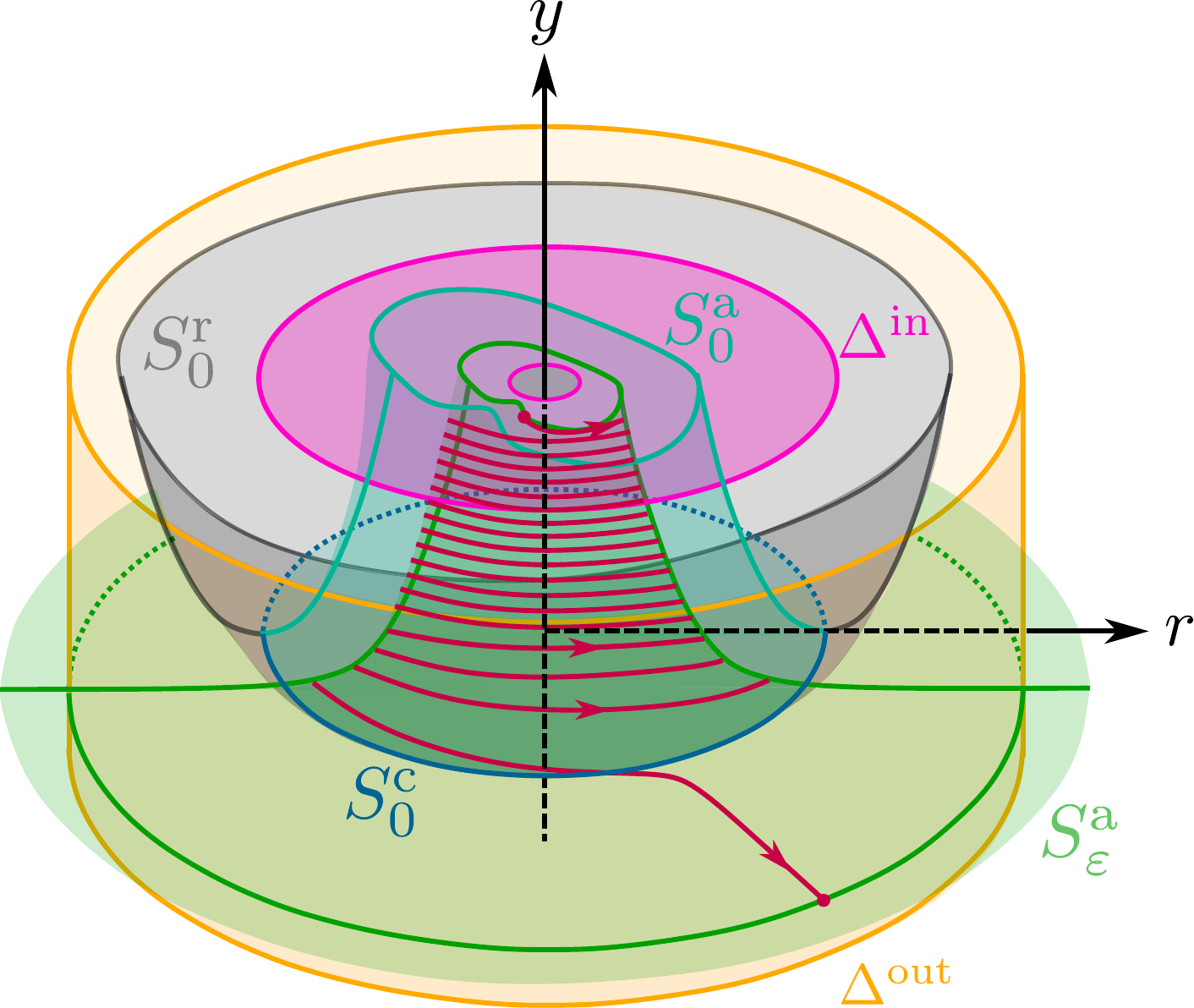}
			\vspace{2mm}
			\caption{Case (C1)$^\ast$. Sketch of the flow of system \eqref{eq:systemcontinuous} for $\alpha =1$.
			The solution sketched in red makes $\SV{N^{(1)}_\textup{rot}(r,\theta,\eps)} = \lfloor \mathcal O(\eps^{-2}) \rfloor$ rotations about the $y$-axis before leaving the neighbourhood close to $S_\eps^\textup{a} \cap \Delta^{\textup{out}}$ at an angle approximated by the expression for $\SV{h_\theta^{(1)}(r,\theta,\eps)}$ in Theorem \ref{thm:main} Assertion~\SV{(b)}.}
			\vspace{3mm}
			\label{fig:flow_0}
		\end{minipage}
		\begin{minipage}[c]{0.04\textwidth}
			\textcolor{white}{.}
		\end{minipage}
		\begin{minipage}[c]{0.47\textwidth}
			\centering
			\vspace{0pt}
			\includegraphics[width=\textwidth]{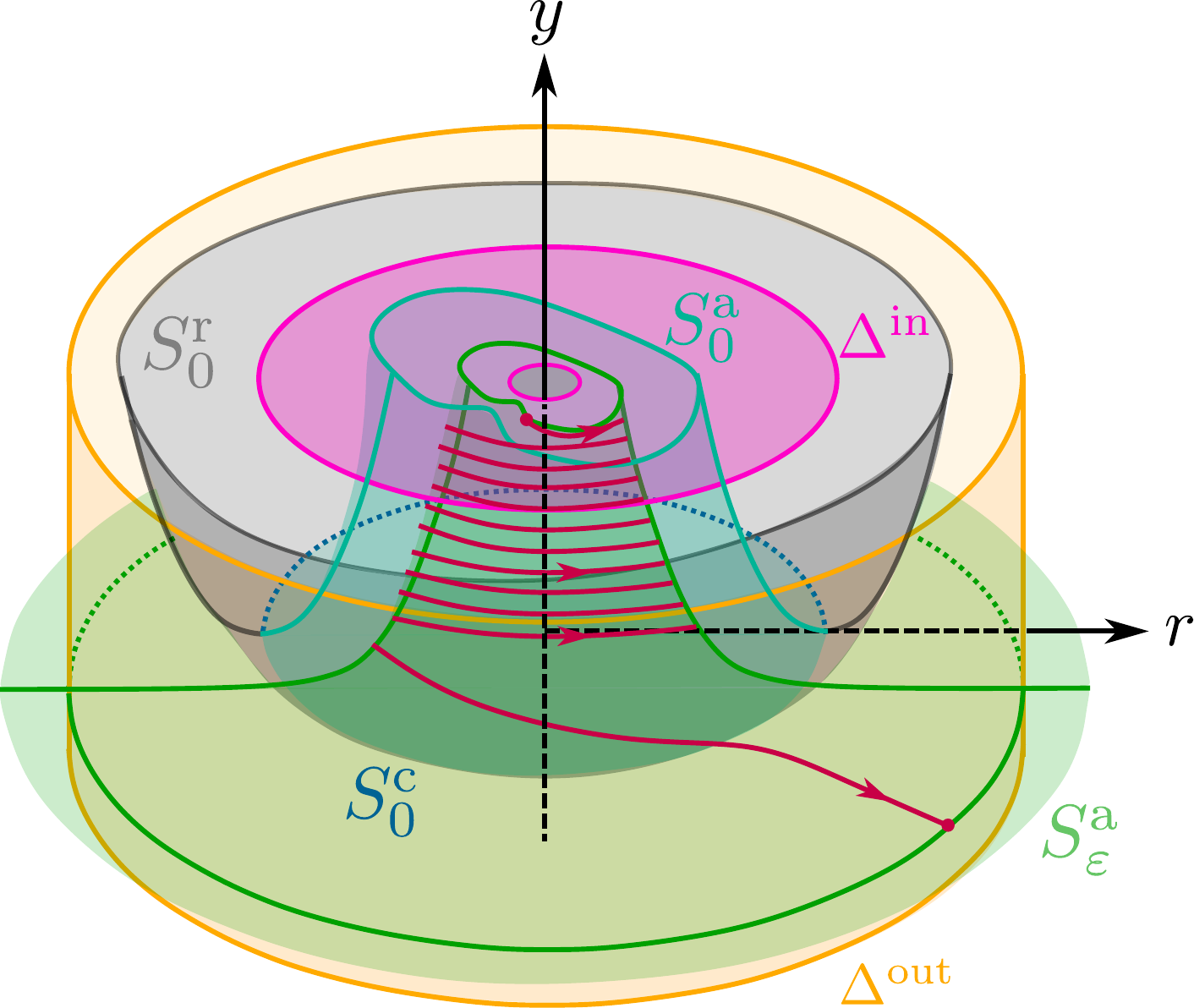}
			\vspace{2mm}
			\caption{Case (C1). Sketch of the flow of system \eqref{eq:systemcontinuous} for \SV{$\alpha = 2$}.
			The solution sketched in red makes $\SV{N_\textup{rot}^{(2)}(r,\theta,\eps)} = \lfloor \mathcal O(\eps^{\SV{-1}}) \rfloor$ rotations about the $y$-axis before leaving the neighbourhood close to $S_\eps^\textup{a} \cap \Delta^{\textup{out}}$ at an angle approximated by the expression for $\SV{h_\theta^{(2)}}(r,\theta,\eps)$ in Theorem \ref{thm:main} Assertion~\SV{(b)}.}
			\label{fig:flow_1}
			\vspace{3mm}
		\end{minipage}
		\begin{minipage}[c]{0.47\textwidth}
			\centering
			\vspace{10pt}
			\includegraphics[width=\textwidth]{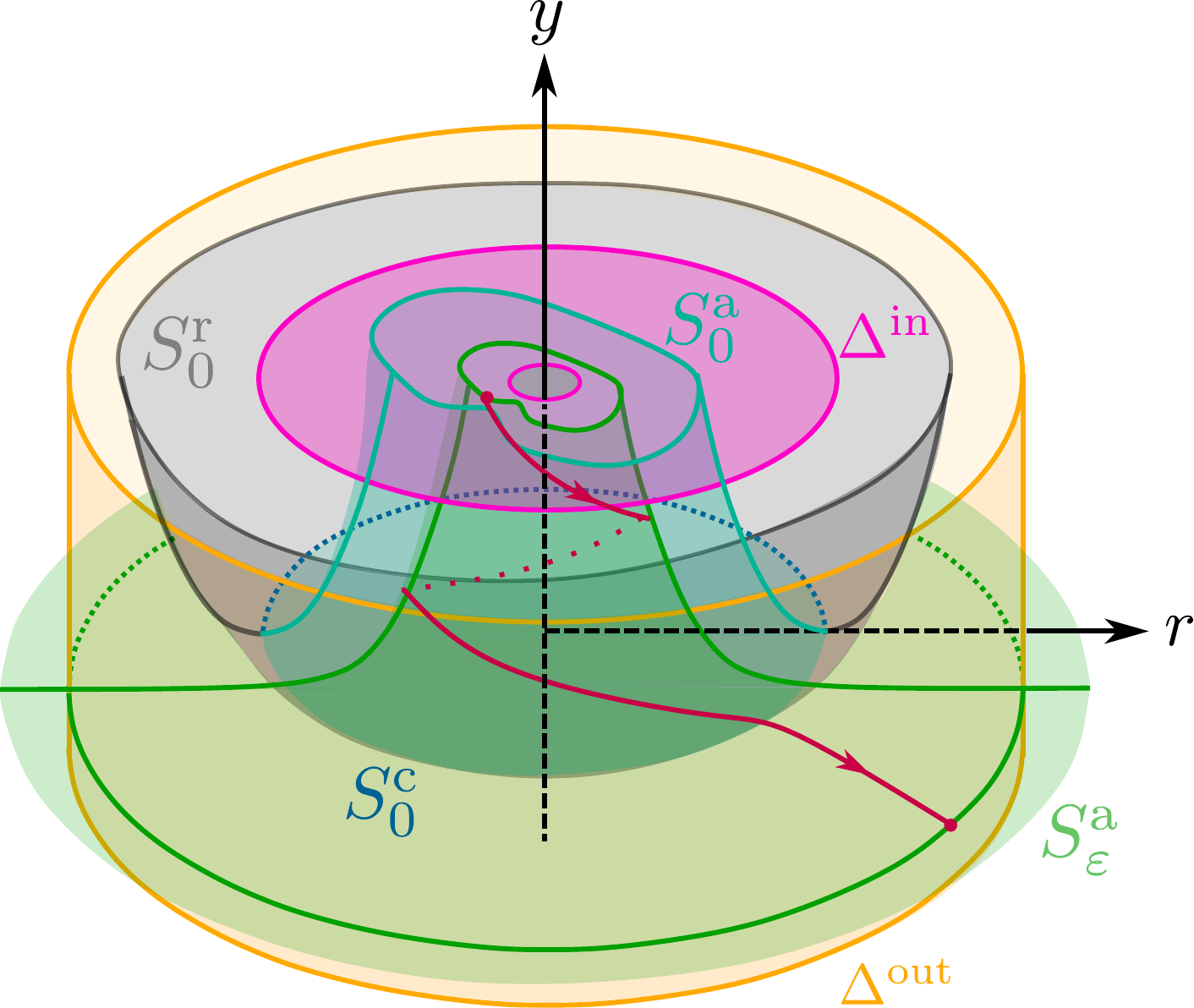}
			\vspace{2mm}
			\caption{Case (C2). Sketch of the flow of system \eqref{eq:systemcontinuous} for $ \alpha =3$.
			The solution sketched in red makes \SV{$N_\textup{rot}^{(3)}(r,\theta,\eps) = \lfloor \frac{b(\theta)}{a(\theta) c(\theta)} R + \mathcal O(R^2) + \mathcal O(\eps^3 \ln \eps) \rfloor$} rotations about the $y$-axis before leaving the neighbourhood close to $S_\eps^\textup{a} \cap \Delta^{\textup{out}}$ at an angle approximated by the expression for $\SV{h_\theta^{(3)}}(r,\theta,\eps)$ in Theorem \ref{thm:main} Assertion \SV{(b)}.}
			\label{fig:flow_2}
		\end{minipage}
		\begin{minipage}[c]{0.04\textwidth}
			\textcolor{white}{.}
		\end{minipage}
		\begin{minipage}[c]{0.47\textwidth}
			\centering
			\vspace{0pt}
			\includegraphics[width=\textwidth]{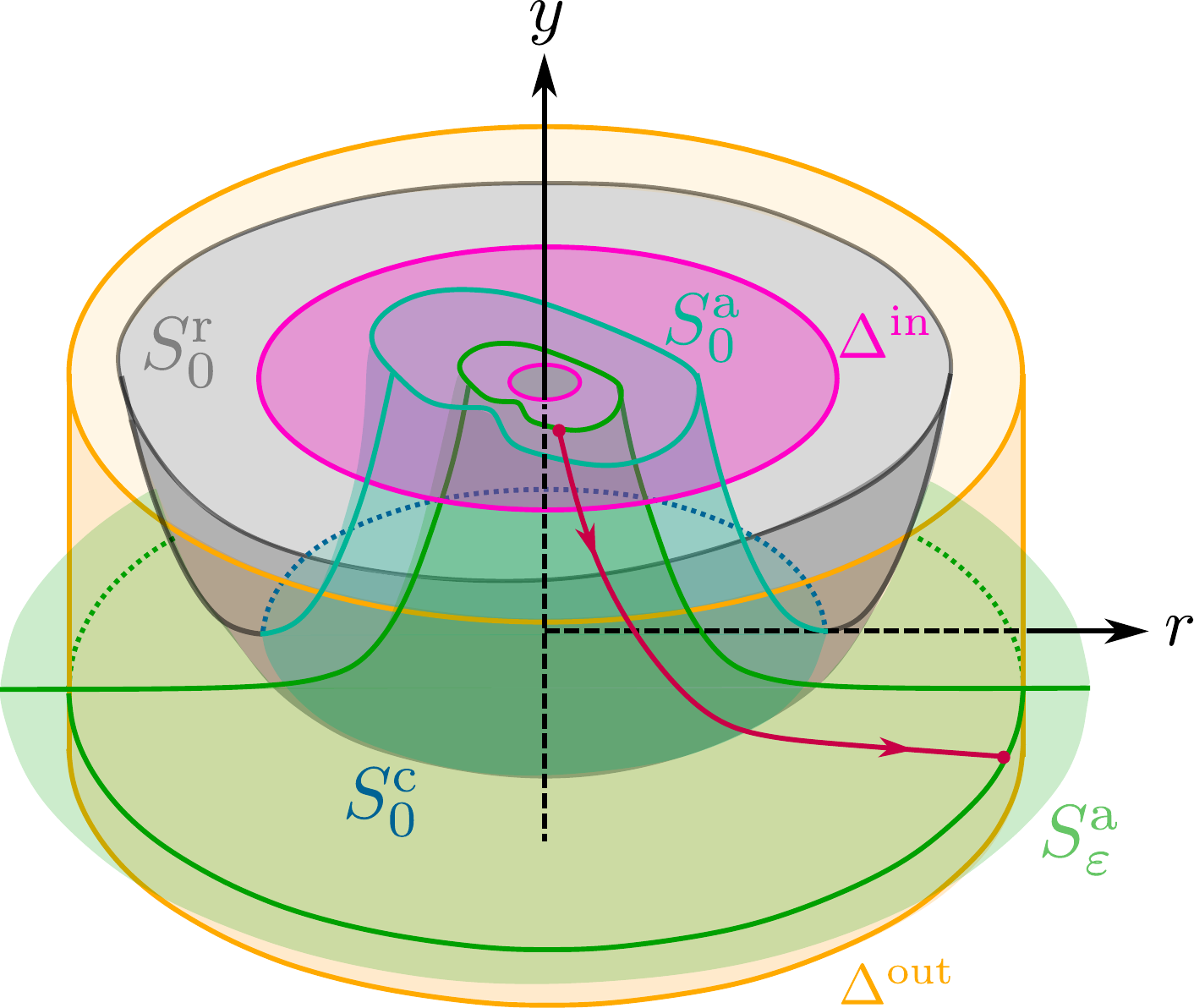}
			\vspace{1mm}
			\caption{Case (C3). Sketch of the flow of system \eqref{eq:systemcontinuous}.
			The solution sketched in red makes \SV{$N_\textup{rot}^{(\alpha)}(r,\theta,\eps) = 0$} rotations about the $y$-axis before leaving the neighbourhood close to $S_\eps^\textup{a} \cap \Delta^{\textup{out}}$ at an angle approximated by the expression for $\SV{h_\theta^{(\alpha)}(r,\theta,\eps)}$ in Theorem \ref{thm:main} Assertion \SV{(b) ($\alpha \geq 4$)}.}
			\label{fig:flow_3}
		\end{minipage}
	\end{figure}
	
	\SJJ{For each $\alpha \in \mathbb N_+$, the size of the leading order term in the} asymptotics for the parameter drift in $y$ \SJJ{is also} the same as for the stationary regular fold point, at least up to $\mathcal O(\eps^3 \ln \eps)$, since\SJJ{
	\[
	h_y^{(\alpha)}(\theta,\eps) = h_y^{(\alpha)}\left(\theta,\eps_{\textup{KSW}}^{1/3}\right) = 
	\mathcal O \left(\eps_{\textup{KSW}}^{2/3}\right) + \mathcal O(\eps_{\textup{KSW}} \ln \eps_{\textup{KSW}}) ,
	\]
	which} agrees with the asymptotic estimates in \SJJ{\cite{Krupa2001a,Mishchenko1975,Szmolyan2004}} (recall from Remark \ref{rem:KS_eps} that \SJJ{$\eps_{\textup{KSW}} = \eps^3$}, where \SJJ{$\eps_{\textup{KSW}}$} denotes the small parameter in \SJJ{\cite{Krupa2001a,Szmolyan2004}}). \SJJ{The strong contraction property in Assertion~(c), which does not depend on $\alpha$, is also the same. These two facts explain} the similarity between the dynamics observed in the $(r,y)$-plane and the dynamics near a stationary regular fold point; c.f.~Figure \ref{fig:setting_2d} and \cite[Figure~2.1]{Krupa2001a}. \SJJ{In contrast to the stationary fold, however, the coefficients of $h_y^{(\alpha)}(\theta,\eps)$ depend on $\theta$, and~\eqref{eq:y_asymptotics} gives the precise form for the leading order coefficient if $\alpha > 1$, i.e.~if $\alpha \in \mathbb N_+ \setminus \{1\}$. Moreover, Theorem~\eqref{thm:main} provides detailed asymptotic information about the angular coordinate $\theta$ via \eqref{eq:theta_asymptotics}. If $\alpha > 2$ then estimates are sharp as $\eps \to 0$.}
	
	\SJJ{Finally, t}he dynamics in different cases, i.e.~for differing values of $\alpha$, are distinguished via the angular dynamics and in particular, the number of complete rotations about the $y$-axis during the transition from $\Delta^{\textup{in}}$ to $\Delta^{\textup{out}}$. \SJJ{We can estimate this number by
	introducing the function $N_{\textup{rot}}^{(\alpha)} : [\beta_-, \beta_+] \times \R / \mathbb Z \times (0, \eps_0] \to \mathbb N$ via
	\[
	N_{\textup{rot}}^{(\alpha)}(r,\theta,\eps) \coloneqq 
	\big\lfloor \tilde h^{(\alpha)}_\theta (r, \theta, \eps) - \theta \big\rfloor .
	\]

	We obtain the following corollary as an immediate consequence of Theorem \ref{thm:main}.
	
	\begin{corollary}
	    \label{cor:num_rotations}
	    Let $\gamma: \mathbb{R} \rightarrow \mathbb{R}_{\geq 0} \times \mathbb{R}/\mathbb{Z} \times \mathbb{R} $ be a solution of system \eqref{eq:systemcontinuous} with initial condition \mbox{$\gamma(0) = (r,\theta,R) \in \Delta^\textup{in}$}. Then $\gamma(t)$ undergoes
	    
	    \[
	    N_{\textup{rot}}^{(\alpha)}(r,\theta,\eps) = 
	    \begin{cases}
	        \SV{\big\lfloor \frac{R^2}{c_0} \eps^{-2} + \mathcal O(\ln \eps) \big\rfloor }, & \alpha = 1, \\
	        \SV{\big\lfloor \frac{R^2}{c_0} \eps^{-1} + \mathcal O(\eps \ln \eps) \big\rfloor }, & \alpha = 2, \\
	        \big\lfloor \frac{b(\theta)}{a(\theta) c(\theta)} R + \mathcal O(R^2) + \mathcal O(\eps^3 \ln \eps) \big\rfloor, & \alpha = 3, \\
	        0 , & \alpha \geq 4 ,
	    \end{cases}
	    \]
	    complete rotations about the $y$-axis during its passage to $\Delta^{\textup{out}}$.
	\end{corollary}
	
	\begin{proof}
	    This follows immediately from the asymptotic estimates for $\tilde h^{(\alpha)}(r,\theta,\eps)$ in Theorem \ref{thm:main} and the definition of $N_{\textup{rot}}^{(\alpha)}$.
	\end{proof}
	}

	\section{Proof of Theorem \ref{thm:main}}
	\label{sec:blowupfold}
	
	Our aim is to investigate the three time-scale system \eqref{eq:systemcontinuous} \SJJ{with $\alpha \in \{1,2\}$} using the blow-up method developed for fast-slow systems in \cite{Dumortier1996,Krupa2001a,Krupa2001c,Krupa2001b}; we refer again to \cite{Jardon2019b} for a recent survey. Many aspects of the proof rely in particular on arguments used in the blow-up analysis of the (stationary) regular fold point in \cite{Krupa2001a}. \SJJ{However, many of the calculations are complicated by the fact that the angular variable $\theta$ cannot be treated locally. In particular, for $\alpha \in \{1,2\}$, we cannot simply Taylor expand the equations about a fixed value of $\theta$. Consequently, a local transformation into the local normal form in \cite{Szmolyan2004} is not possible. A similar feature arises when blowing up the fold cycle in the periodically forced van der Pol equation in the `intermediate frequency regime' \cite{Burke2016}, except that in our case, there is no decoupling of the angular dynamics in the leading order.} We adopt the (now well established) notational conventions introduced in \cite{Krupa2001a,Krupa2001c,Krupa2001b}.
	
	The blow-up transformation is defined in Section \ref{sub:blow-up_and_local_coordinate_charts}, as are the three local coordinate charts that we use for calculations. The geometry and dynamics in all three coordinate charts are considered in turn in Sections \ref{sec:51}, \ref{sec:52} and \ref{sec:53}. Theorem \ref{thm:main} is proved in Section \ref{sec:54} using the information obtained in local coordinate charts.

	\subsection{Blow-up and Local Coordinate Charts}
	\label{sub:blow-up_and_local_coordinate_charts}
	
	As is standard in blow-up approaches, we consider the extended system obtained from \eqref{eq:systemcontinuous} after appending the trivial equation $\eps' = 0$, i.e.~\SJJ{
	\begin{equation} \label{eq:system}
		\begin{aligned}
			r' &= - a(\theta) y + b(\theta) r^2 + \widetilde{\mathcal R}_r(r,\theta,y,\varepsilon), \\
			\theta' &= \varepsilon^\alpha, \\
			y' &= \varepsilon^3(- c(\theta) + \widetilde{\mathcal R}_y(r,\theta,y,\varepsilon)), \\
			\varepsilon' &= 0 ,
		\end{aligned}  
	\end{equation}
	}where \SJJ{$\alpha \in \{1,2\}$ is fixed, $\widetilde{\mathcal R}_r(r,\theta,y,\varepsilon) = \mathcal{O}(r^3,y^2,ry, \eps^\alpha r^2, \eps^\alpha y, \varepsilon^3)$} and \SJJ{$\widetilde{\mathcal R}_y(r,\theta,y,\varepsilon) = \mathcal{O}(r,y,\varepsilon^\alpha)$}.
	
	In order to describe the map \SJJ{$\pi^{(\alpha)} : \Delta^{\textup{in}} \rightarrow \Delta^{\textup{out}}$}, we introduce extended entry and exit sections\SJJ{
	\begin{equation} 
		\label{eq:Delta_in_eps}
		\Delta^{\textup{in}}_\eps \coloneqq \left\{(r,\theta,R^2,\eps) : r \in [\beta_-, \beta_+] , \theta \in \R / \mathbb Z, \eps \in [0,\eps_0] \right\} ,
	\end{equation}
	}and
	\begin{equation}
		\label{eq:Delta_out_eps}
		\Delta^{\textup{out}}_\eps \coloneqq \left\{(R,\theta,y,\eps) : \theta \in \R / \mathbb Z , y \in [-y_0,y_0] , \eps \in [0,\eps_0] \right\} ,
	\end{equation}
	respectively. The entry, exit sections $\Delta^{\textup{in}}$, $\Delta^{\textup{out}}$ defined in Section \ref{sec:main_results} (recall equations \eqref{eq:Delta_in} and \eqref{eq:Delta_out}) can be viewed as constant $\eps$ sections of $\Delta^{\textup{in}}_\eps$, $\Delta^{\textup{out}}_\eps$, respectively. Based on this simple relationship, we study the map \SJJ{$\pi^{(\alpha)} : \Delta^{\textup{in}} \rightarrow \Delta^{\textup{out}}$} described in Theorem \ref{thm:main} via the extended map \SJJ{$\pi^{(\alpha)}_\varepsilon: \Delta_\varepsilon^{\textup{in}} \rightarrow \Delta_\varepsilon^{\textup{out}}$} induced by the flow of initial conditions in $\Delta_\varepsilon^{\textup{in}}$ up to $\Delta_\varepsilon^{\textup{out}}$ under system \eqref{eq:system}.
	
	\

	We now define the relevant blow-up transformation. Let $I = [0,\SJJ{\rho_0}]$, \SJJ{where $\rho_0 > 0$ is fixed small enough for the validity of local computations,}
	let
	\begin{equation*}
		B_0 \coloneqq S^2 \times \R/\mathbb{Z} \times [0,\SJJ{\rho_0}] ,
	\end{equation*}
	and define the (weighted) blow-up transformation via
	\begin{equation} \label{eq:blowup}
		\varphi: B_0 \rightarrow \R_{\geq - \SJJ{\rho_0}} \times \R/\mathbb{Z} \times \R \times \R_{\geq 0}, \quad (\bar{r},\bar{y},\bar{\varepsilon},\theta,\rho) \mapsto (r,\theta,y,\varepsilon) = (\rho \bar{r}, \theta, \rho^2 \bar{y}, \rho \bar{\varepsilon}),
	\end{equation}
	where $(\bar{r},\bar{y},\bar{\varepsilon}) \in S^2$. The blow-up map $\varphi$ is a diffeomorphism for $\rho \in (0,\SJJ{\rho_0}]$, but not for $\rho = 0$. In particular, the preimage of the non-hyperbolic cycle $S_0^\textup{c}$ under $\varphi$ is a `torus of spheres' \mbox{$S^2 \times \R / \mathbb Z \times \{0\} \cong S^2 \times S^1$}.
	
	For calculational purposes, we introduce local coordinate charts in order to describe the dynamics on
	\begin{equation*}
		B_{\bar{y}}^+ \coloneqq B_0 \cap \{\bar{y} > 0\}, \qquad 
		B_{\bar{\varepsilon}}^+ \coloneqq B_0 \cap \{\bar{\varepsilon} > 0\}, \qquad B_{\bar{r}}^+ \coloneqq B_0 \cap \{\bar{r} > 0\}.    
	\end{equation*}
	Following \cite{Krupa2001a}, we introduce affine projective coordinates via
	\begin{align*}
		&K_1: \; B_{\bar{y}}^+ \rightarrow \R \times \R/\mathbb{Z} \times \R \times \R, \quad (\bar{r},\bar{y},\bar{\varepsilon},\rho,\theta) \mapsto (r_1,\theta_1,\rho_1,\varepsilon_1) = (\bar{r}\bar{y}^{-\frac{1}{2}}, \theta, \rho \bar{y}^{\frac{1}{2}},\bar{\varepsilon} \bar{y}^{-\frac{1}{2}} ), \\
		&K_2: \; B_{\bar{\varepsilon}}^+ \rightarrow \R \times \R/\mathbb{Z} \times \R \times \R, \quad (\bar{r},\bar{y},\bar{\varepsilon},\rho,\theta) \mapsto (r_2,\theta_2,y_2,\rho_2) = (\bar{r}\bar{\varepsilon}^{-1}, \theta, \bar{y} \bar{\varepsilon}^{-2},\rho \bar{\varepsilon}), \\
		&K_3: \; B_{\bar{r}}^+ \rightarrow \R \times \R/\mathbb{Z} \times \R \times \R, \quad (\bar{r},\bar{y},\bar{\varepsilon},\rho,\theta) \mapsto (\rho_3,\theta_3,y_3,\varepsilon_3) = (\rho\bar{r}, \theta, \bar{y} \bar{r}^{-2},\bar{\varepsilon} \bar{r}^{-1}).
	\end{align*}
	Here we permit a small abuse of notation by introducing a new variable $\eps_1$, which should not be confused with the small parameter with the same notation in Sections \ref{sec:Introduction}-\ref{sec:setting}. This leads to the following coordinates:
	\begin{equation}
		\label{eq:chart_coordinates}
		\begin{split}
			K_1 : \ & (r,\theta,y,\varepsilon) = (\rho_1 r_1, \theta_1, \rho_1^2, \rho_1 \varepsilon_1), \\
			K_2 : \ & (r,\theta,y,\varepsilon) = (\rho_2 r_2, \theta_2, \rho_2^2 y_2, \rho_2), \\
			K_3 : \ & (r,\theta,y,\varepsilon) = (\rho_3, \theta_3, 	\rho_3^2 y_3, \rho_3 \varepsilon_3).
		\end{split}
	\end{equation}
	In the analysis, it will be necessary to change coordinates between different charts. The following lemma provides the relevant change of coordinates formulae. 
	\begin{lemma} \label{lem:change}
		The change of coordinate maps $\kappa_{ij}$ from $K_i$ to $K_j$ are diffeomorphisms given by
		\begin{equation}
			\label{eq:kappa_formulae}
			\begin{aligned}
				\kappa_{12}:& \quad &&r_2 = r_1 \varepsilon_1^{-1}, \quad &&\theta_2 = \theta_1, \quad &&y_2 = \varepsilon_1^{-2}, \quad && \rho_2 = \rho_1 \varepsilon_1, &&\quad \textup{for } \varepsilon_1 > 0, \\
				\kappa_{12}^{-1}:& \quad &&r_1 = r_2 y_2^{-\frac{1}{2}}, \quad &&\theta_1 = \theta_2, \quad &&\rho_1 = \rho_2 y_2^{\frac{1}{2}}, \quad && \varepsilon_1 = y_2^{-\frac{1}{2}}, &&\quad \textup{for } y_2 > 0, \\
				\kappa_{23}:& \quad &&\rho_3 = \rho_2 r_2, \quad &&\theta_3 = \theta_2,\quad &&y_3 = y_2 r_2^{-2}, \quad && \varepsilon_3 = r_2^{-1}, &&\quad \textup{for } r_2 > 0, \\
				\kappa_{23}^{-1}:& \quad &&r_2 = \varepsilon_3^{-1}, \quad &&\theta_2 = \theta_3, \quad &&y_2 = y_3 \varepsilon_3^{-2}, \quad && \rho_2 = \rho_3 \varepsilon_3, &&\quad \textup{for } \varepsilon_3 > 0.
			\end{aligned}
		\end{equation}
	\end{lemma}
	
	\begin{proof}
		This follows from the local coordinate expressions in \eqref{eq:chart_coordinates}.
	\end{proof}
	
	\begin{remark}
		\SJJ{The blow-up transformation \eqref{eq:blowup} has the same form as the blow-up map used to study dynamics near a regular fold curve in $\R^3$ in \cite{Szmolyan2004}, except that the domain of the uncoupled variable $\theta$ is $\R / \mathbb Z$ instead of $\R$. Note also that since} $\theta$ is
		unaffected by
		\eqref{eq:blowup}\SJJ{, i}t follows that $\varphi$ can be defined more succinctly in terms of its action on the remaining variables, i.e.~via the map
		\[
		S^2 \times [0,\SJJ{\rho_0}] \rightarrow \R_{\geq -\SJJ{\rho_0}} \times \R \times \R_{\geq 0}, \qquad (\bar{r},\bar{y},\bar{\varepsilon},\rho) \mapsto (r,y,\varepsilon) = (\rho \bar{r}, \rho^2 \bar{y}, \rho \bar{\varepsilon}),
		\]
		which is precisely the blow-up map used to study the regular fold point in \cite{Krupa2001a} (recall that \SJJ{$\eps_{\textup{KSW}} = \eps^3$} by the discussion following the statement of Theorem \ref{thm:main}).
	\end{remark}
	
	\begin{remark}
		By construction, the blown-up vector field induced by the pushforward of the vector field induced by system \eqref{eq:system} under $\varphi$ is invariant in the hyperplanes $\{\rho = 0\}$ and $\{\bar \eps = 0\}$. The former corresponds to blown-up preimage of the non-hyperbolic cycle $S^\textup{c}_0$, i.e.~the torus of spheres $S^2 \times S^1$. The latter contains the blown-up preimage of the critical manifold $S_0$. Since $\bar r^2 + \bar y^2 = 1$ in $\{\bar \eps = 0\}$, the preimage of $\varphi$ in $\{\bar \eps = 0\}$ is $S^1 \times \R / \mathbb Z \times \R_{\geq 0} \cong S^1 \times S^1 \times \R_{\geq 0}$. Thus, the part of the blown-up singular cycle $S_0^\textup{c}$ within $\{\bar \eps = 0\}$ is a torus.
	\end{remark}
	
	\begin{remark}
		\label{rem:const_motion}
		Since $\eps = const.$ in system \eqref{eq:system} we have constants of the motion defined by $\eps = \rho_1 \eps_1$, $\eps = \rho_2$, $\eps = \rho_3 \eps_3$ in charts $K_1$, $K_2$, $K_3$, respectively.
	\end{remark}
	
	We turn now to the dynamics in charts $K_i$, $i=1,2,3$.

		\subsection{Dynamics in the Entry Chart $K_1$}
	\label{sec:51}
	
		In chart $K_1$ we analyse solutions which track the blown-up preimage of the attracting Fenichel slow manifold $S_\varepsilon^\textup{a}$ as they enter a neighbourhood of the 		non-hyperbolic circle $S_0^\textup{c}$.
	
	\begin{lemma} \label{lem:blowupK1}
		Following the positive transformation of time $\rho_1 \dd t = \dd t_1$, the desingularized equations in chart $K_1$ are given by
		\begin{equation} \label{eq:chart1desing}
			\SVV{\begin{aligned}
					r_1' &= -a(\theta_1)+b(\theta_1)r_1^2+\frac{1}{2} c(\theta_1)  \varepsilon_1^3 r_1+ \mathcal{O}(\rho_1), \\
					\theta_1' &= \rho_1^{\alpha-1}\varepsilon_1^\alpha, \\
					\rho_1' &= - \frac{1}{2}\rho_1 \varepsilon_1^3(c(\theta_1)+\mathcal{O}(\rho_1)), \\
					\varepsilon_1' &= \frac{1}{2}\varepsilon_1^4(c(\theta_1) + \mathcal{O}(\rho_1)) ,
			\end{aligned}}
		\end{equation}
		where by a slight abuse of notation we now write $(\cdot)' = \dd / \dd t_1$.
	\end{lemma}
	
	\begin{proof}
		This follows after direct differentiation of the local coordinate expressions in \eqref{eq:chart_coordinates} and subsequent application of the desingularization $\rho_1 \dd t = \dd t_1$.
	\end{proof}
	
	The analysis in chart $K_1$ is restricted to the set
	\begin{equation*}
		\mathcal{D}_{1} \coloneqq \{(r_1, \theta_1, \rho_1, \varepsilon_1) \in \R \times \R/\mathbb{Z} \times \R \times \R:  \; 0 \leq \rho_1 \leq R, \; 0 \leq \varepsilon_1 \leq E\},
	\end{equation*}
	where $R$ is the constant which defines the entry set $\Delta^{\textup{in}}$ in \eqref{eq:Delta_in} and $E = \eps_0 / R>0$ due to the relationship $\eps = \rho_1 \eps_1$ (recall Remark \ref{rem:const_motion}). The set $\mathcal{D}_1$ is sketched with other geometric and dynamical objects in chart $K_1$ in Figure \ref{fig:chart1_wo_theta}. 
	
	System \eqref{eq:chart1desing} is well-defined within $\{\rho_1 = 0\}$ (recall that \SJJ{$\alpha \in \{1,2\}$}), which corresponds to the part of the blow-up surface in $K_1$. The subspace $\{\eps_1 = 0\}$ is also invariant, and contains two two-dimensional critical manifolds
	\SVV{\begin{equation}
			\begin{split}
			\label{eq:S_a1}
			S_1^{\textup{a}} &\coloneqq \left\{ \left(r_1, \theta_1, \rho_1, 0 \right) \in \mathcal{D}_{1} : r_1 = - \left(\tfrac{a(\theta_1)}{b(\theta_1)}\right)^\frac{1}{2} + \mathcal O(\rho_1) \right\} , \\
			S_1^{\textup{r}} &\coloneqq \left\{ \left(r_1, \theta_1, \rho_1, 0 \right) \in \mathcal{D}_{1} : r_1 = \left(\tfrac{a(\theta_1)}{b(\theta_1)}\right)^\frac{1}{2} + \mathcal O(\rho_1) \right\} .
			\end{split}
	\end{equation}
	
	The manifolds $S_1^\textup{a}$ and $S_1^\textup{r}$ correspond} to the blown-up preimages of the critical manifolds $S_0^{\textup{a}}$ and $S_0^{\textup{r}}$ in chart $K_1$, respectively. Both $S_1^{\textup{a}}$ and $S_1^{\textup{r}}$ are topologically equivalent to cylindrical segments, and they are normally hyperbolic and attracting resp.~repelling up to and including the 
	sets
	\SVV{\[
		P_{\textup{a}} \coloneqq \left\{\left(-\left(\frac{a(\theta_1)}{b(\theta_1)}\right)^\frac{1}{2},\theta_1,0,0 \right) : \theta_1 \in \R / \mathbb Z \right\} , \qquad 
		P_{\textup{r}} \coloneqq \left\{ \left(\left(\frac{a(\theta_1)}{b(\theta_1)}\right)^\frac{1}{2},\theta_1,0,0\right) : \theta_1 \in \R / \mathbb Z \right\} ,
		\]
	}which intersect with the blow-up surface; see Figures \ref{fig:chart1_wo_theta} and \ref{fig:chart1_eps_0}. The linearisation of system \eqref{eq:chart1desing} along $P_{\textup{a}}$ is\SJJ{
		\begin{equation} 
			\label{eq:linearisation}
			J(P_{\textup{a}}) = \begin{pmatrix} -2 \sqrt{a(\theta_1) b(\theta_1)} & \left( a(\theta_1) b'(\theta_1) - a'(\theta_1) b(\theta_1) \right) / b(\theta_1) & 0 & 0 \\ 0 & 0 & 0 & 0^{\alpha-1} \\ 0 & 0 & 0 & 0 \\ 0 & 0 & 0 & 0 \end{pmatrix},
		\end{equation}
	where we write $a' = \partial a / \partial \theta_1$ and $b' = \partial b / \partial \theta_1$, and we write
	\[
	0^{\alpha - 1} = 
	\begin{cases}
	    1, & \alpha = 1, \\
	    0, & \alpha = 2.
	\end{cases}
	\]
	}In both cases, the set $P_{\textup{a}}$ is partially hyperbolic with a single non-trivial eigenvalue \SJJ{$-2 \sqrt{a(\theta_1) b(\theta_1)} < 0$.}
	The remaining three eigenvalues along $P_{\textup{a}}$ are identically zero\SJJ{.} 
	Thus in the blown-up space (i.e.~for system~\eqref{eq:chart1desing}), we have regained partial hyperbolicity. This allows us to extend the attracting center manifold with base along $S_1^{\textup{a}}$ up onto the blow-up surface using center manifold theory.
	
		\begin{figure}[t!]
		\centering
		\includegraphics[width=0.8\textwidth]{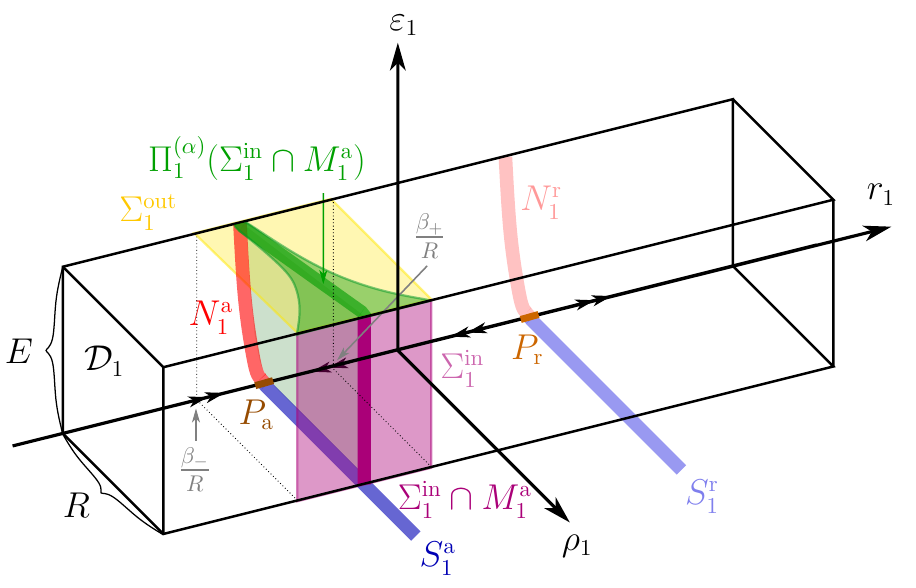}
		\caption{Geometry and dynamics within $\mathcal D_1$, shown in $(r_1,\rho_1,\varepsilon_1)$-space. Projections of the two-dimensional attracting and repelling critical manifolds $S_1^\textup{a} , S_1^\textup{r} \subset \{\eps_1 = 0\}$ (shown in blue and shaded blue) extend up to their intersection with the blow-up surface at $P_\textup{a} , P_\textup{r} \subset \{\rho_1 = \eps_1 = 0\}$, respectively. Projections of the two-dimensional center manifolds $N_1^\textup{a}, N_1^\textup{r} \subset \{\rho_1 = 0\}$ emanating from $P_\textup{a}, P_\textup{r}$, as described in Lemma \ref{lem:manifolds}, are \SV{sketched} in red and shaded red \SV{for the case $\alpha = 2$, but they look similar in the case $\alpha = 1$.} Entry and exit sections $\Sigma_1^\textup{in}$ and $\Sigma_1^{\textup{out}}$ are shown in shaded magenta and yellow, respectively. The projected three-dimensional center manifold $M_1^\textup{a}$ and its image under $\SV{\Pi_1^{(\alpha)}}$ is described by Proposition \ref{prop:pi1}. This is shown along with the image $\SV{\Pi_1^{(\alpha)}}(\Sigma_1^{\textup{in}}) \subset \Sigma_1^\textup{out}$ in green.}
		\label{fig:chart1_wo_theta}
	\end{figure}

	\begin{lemma} \label{lem:manifolds} 
		Consider system \eqref{eq:chart1desing} on $\mathcal D_1$ with $E,R>0$ sufficiently small. There exists a three-dimensional center-stable manifold $M_1^\textup{a}$ such that
		\[
		M_1^\textup{a} |_{\eps_1 = 0} = S_1^{\textup{a}} , \qquad 
		M_1^\textup{a} |_{\rho_1 = 0} = N_1^{\textup{a}} ,
		\]
		where $S_1^{\textup{a}} \subset \{ \eps_1 = 0 \}$ is the two-dimensional attracting critical manifold in \eqref{eq:S_a1} and $N_1^{\textup{a}} \subset \{ \rho_1 = 0\}$ is a unique two-dimensional center-stable manifold for the restricted system \eqref{eq:chart1desing}$|_{\rho_1=0}$ emanating from $P_{\textup{a}}$. The manifold $M_1^\textup{a}$ is given locally as a
		graph \SJJ{$r_1 = h^{(\alpha)}_{r_1}(\theta_1,\rho_1,\varepsilon_1)$}, where\SJJ{
			\begin{equation}
				\label{eq:M_a1_graph}
				h^{(1)}_{r_1}(\theta_1,\rho_1,\varepsilon_1) = - \left(\frac{a(\theta_1)}{b(\theta_1)}\right)^\frac{1}{2} +
				\frac{a(\theta_1) b'(\theta_1) - a'(\theta_1) b(\theta_1)}{4 a(\theta_1) b(\theta_1)} \varepsilon_1 + \mathcal{O}\left(\rho_1,\varepsilon_1^2\right)
			\end{equation}
		and}
			\SVV{\begin{equation}
				\label{eq:M_a2_graph}
				h_{r_1}^{(2)}(\theta_1,\rho_1,\varepsilon_1) = - \left(\frac{a(\theta_1)}{b(\theta_1)}\right)^\frac{1}{2} - \frac{c(\theta_1)}{4b(\theta_1)} \varepsilon_1^3 + \mathcal{O}\left(\rho_1,\varepsilon_1^6\right) .
			\end{equation}
			\SJJ{In both cases,}} there exists a constant \SJJ{$\varrho \in (0,\vartheta)$, where $\vartheta \coloneqq \min_{\theta_1 \in [0,1)} 2 \sqrt{a(\theta_1) b(\theta_1)} > 0$}, such that initial conditions in $\Sigma^{\textup{in}}_1$ are attracted to $M_1^{\textup{a}}$ along one-dimensional stable fibers faster than \SJJ{$\me^{-\varrho t_1}$}.
	\end{lemma}
	
	\begin{proof}
		The existence of a three-dimensional center manifold $M_1^\textup{a}$ follows from
		the linearisation \eqref{eq:linearisation} and center manifold theory \cite{Kuznetsov2013}. The fact that $M_1^\textup{a}$ contains two-dimensional manifolds $N_1^\textup{a}$ and $S_1^\textup{a}$ with the properties described follows after direct restriction to the invariant hyperplanes $\{\rho_1 = 0\}$ and $\{\eps_1 = 0\}$, respectively. 
		
		The graph representation\SJJ{s} in \eqref{eq:M_a1_graph} \SJJ{and \eqref{eq:M_a2_graph} are} obtained using the standard approach based on formal matching. More precisely, we substitute a power series ansatz of the form
		\SVV{\begin{equation*}
				r_1 = h_{r_1}^{(\alpha)} (\theta_1,\rho_1,\varepsilon_1) = \sum_{n=0}^{\infty} \mu_n^{(\alpha)}(\theta_1) \varepsilon_1^n + \mathcal{O}(\rho_1)
		\end{equation*}
		}into the equation for \SJJ{$r_1'$} in system \eqref{eq:chart1desing} and determine the coefficients \SVV{$\mu_n^{(\alpha)}(\theta_1)$ by comparing
			\begin{align*}
				r_1' &=  \sum_{n=0}^{\infty} \left( \frac{ \dd \mu_n^{(\alpha)}(\theta_1)}{\dd \theta_1} \theta_1' \varepsilon_1^n + \mu_{n+1}^{(\alpha)}(\theta_1) (n+1) \varepsilon_1^{n} \varepsilon_1'\right)+ \mathcal{O}(\rho_1)\\ 
				&= \sum_{n =0}^{\infty} \left(\frac{ \dd \mu_n^{(\alpha)}(\theta_1)}{\dd \theta_1} \rho_1^{\alpha-1}\varepsilon_1^{n+\alpha} + \frac{1}{2} \mu_{n+1}^{(\alpha)}(\theta_1) (n+1) \varepsilon^{n+4}\right)+ \mathcal{O}(\rho_1) ,
			\end{align*}
			with
			\begin{align*}
				r_1'&= b(\theta)\left(\sum_{n=0}^\infty \mu_n^{(\alpha)}(\theta_1) \varepsilon_1^n \right)^2 - a(\theta_1) + \frac{1}{2}c(\theta_1) \varepsilon_1^3 \sum_{n = 0}^\infty \mu_n^{(\alpha)}(\theta_1) \varepsilon_1^n + \mathcal{O}(\rho_1)\\
				&= b(\theta)\sum_{n=0}^\infty \left(\sum_{k=0}^n \mu_k^{(\alpha)}(\theta_1) \mu_{n-k}^{(\alpha)}(\theta_1) \right) \varepsilon_1^n -a(\theta_1) + \frac{1}{2} c(\theta_1) \sum_{n=0}^\infty \mu_n^{(\alpha)}(\theta_1) \varepsilon_1^{n+3}+ \mathcal{O}(\rho_1).
			\end{align*}
		}Depending \SJJ{on whether $\alpha = 1$ or $\alpha = 2$, the approximations in \eqref{eq:M_a1_graph} or \eqref{eq:M_a2_graph} are obtained (respectively)}. The details are omitted for brevity.
		
		The strong contraction along stable fibers at a rate greater than \SJJ{$\me^{- \varrho t_1}$} for some \SJJ{$\varrho \in (0,\vartheta)$}, \SV{ where $\vartheta \coloneqq \min_{\theta_1 \in [0,1)} 2 \sqrt{a(\theta_1) b(\theta_1)} > 0$} follows from Fenichel theory \cite[Theorem 9.1]{Fenichel1979} and the fact that the stable leading eigenvalue is \SJJ{$\lambda = -2 \sqrt{a(\theta_1) b(\theta_1)}$}, recall \eqref{eq:linearisation}.
	\end{proof}
	
	The two-dimensional center manifold $N_1^{\textup{a}}$ is sketched within $\{\rho_1 = 0\}$ in Figure \ref{fig:chart1_rho_0}.

	\begin{remark}
		Similar to Lemma \ref{lem:manifolds}, there exists a three-dimensional center-unstable manifold $M_1^\textup{r}$ at $P_{\textup{r}}$ that contains the repelling critical manifold $S_1^\textup{r} \subset \{ \eps_1 = 0\}$ and a repelling center manifold \mbox{$N_1^\textup{r} \subset \{ \rho_1 = 0\}$}. These objects are shown in Figures \ref{fig:chart1_eps_0} and \ref{fig:chart1_rho_0}.
		In $\mathcal{D}_1$, the manifold $M_1^\textup{r}$ is given as a graph \SJJ{\mbox{$r_1 = \tilde{h}^{(\alpha)}_{r_1}(\theta_1,\rho_1,\eps_1)$}}, where\SJJ{
		\[
			\tilde h_{r_1}^{(1)}(\theta_1,\rho_1,\varepsilon_1) = \left(\frac{a(\theta_1)}{b(\theta_1)}\right)^\frac{1}{2} + \frac{a(\theta_1) b'(\theta_1) - a'(\theta_1) b(\theta_1)}{4 a(\theta_1) b(\theta_1)} \varepsilon_1 + \mathcal{O}\left(\rho_1,\varepsilon_1^2\right)
			\]
		and}\SVV{
		\begin{equation*}
				\tilde{h}^{(2)}_{r_1}(\theta_1,\rho_1,\eps_1) = \left(\frac{a(\theta_1)}{b(\theta_1)}\right)^\frac{1}{2} - \frac{c(\theta_1)}{4b(\theta_1)} \varepsilon_1^3+ \mathcal{O}(\varepsilon_1^6,\rho_1) .
			\end{equation*}
		}In contrast to $N_1^\textup{a}$, the manifold $N_1^\textup{r}$ is not unique. We omit the explicit treatment of the dynamics near $M_1^{\textup{r}}$ since it is not relevant to the proof of Theorem \ref{thm:main}. 
	\end{remark}
	
	Lemma \ref{lem:manifolds} can be used to describe the $r_1$-component of the map \SJJ{$\Pi^{(\alpha)}_1 : \Sigma_1^{\textup{in}} \to \Sigma_1^{\textup{out}}$} induced by the flow of initial conditions in
	\SV{\begin{equation*}
			\Sigma_1^{\textup{in}} \coloneqq \left\{(r_1,\theta_1,R,\varepsilon_1)\in \mathcal{D}_1:\; r_1 \in [\beta_- / R, \beta_+ / R] \right\},
	\end{equation*}
	}i.e.~the representation of the entry section $\Delta^{\textup{in}}_\eps$ defined in \eqref{eq:Delta_in_eps} in $K_1$ coordinates, up to the exit section
	\SV{\begin{equation*}
			\Sigma_1^{\textup{out}} \coloneqq \{(r_1,\theta_1,\rho_1,E)\in \mathcal{D}_1: \; r_1 \in [\beta_- / R, \beta_+ / R] \}.
	\end{equation*}
	The constants $\beta_-<\beta_+<0$ are chosen in such a way that the $r_1$-coordinate of the center manifold $M_1^\text{a}$ in $\mathcal{D}_1$ is always contained in the interval $(\beta_-/R, \beta_+/R)$, cf.\ Definition \eqref{eq:Delta_in} of the entry section $\Delta^\text{in}$.}

	The remaining $\theta_1$, $\rho_1$ and $\varepsilon_1$ components of the map \SJJ{$\Pi^{(\alpha)}_1$} \SV{and the transition time taken for solutions with initial conditions in $\Sigma_1^{\textup{in}}$ to reach $\Sigma_1^{\textup{out}}$} can be \SJJ{estimated} directly.
	\SV{To obtain better estimates, we rewrite the positive, 1-periodic smooth function $c(\theta_1)$ as
	\begin{equation} \label{eq:c_theta_fourier}
	    c(\theta_1) = c_0 + c_\textup{rem}(\theta_1),
	\end{equation}
	where $c_0 \coloneqq \int_0^1 c(\theta_1(t_1)) \, \dd t_1$ is the mean value of $c(\theta_1)$ over one period and $c_\textup{rem}(\theta_1) \coloneqq c(\theta_1)-c_0$ is the smooth, 1-periodic remainder with mean zero. The} \SJJ{estimates} are provided by the following result.
	
		\begin{figure}[t!]
		\centering
		\begin{minipage}[c]{0.47\textwidth}
			\centering
			\vspace{0pt}
			\includegraphics[width=0.95\textwidth]{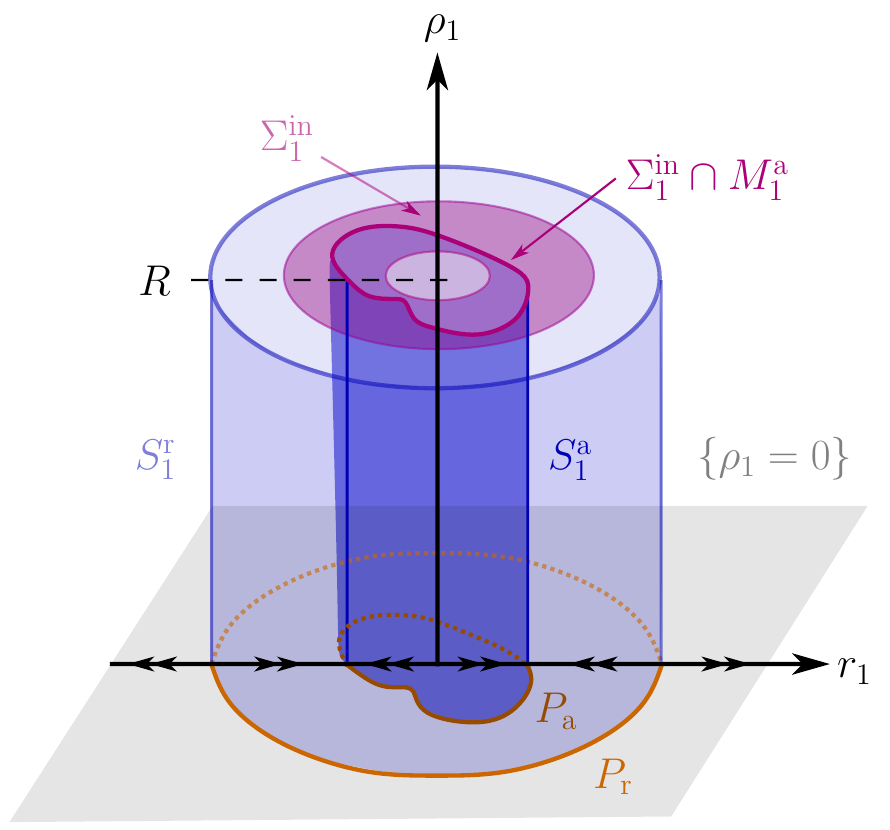}
			\caption{Geometry and dynamics within the invariant hyperplane $\{\varepsilon_1 = 0\}$, projected into $(r_1,\theta_1,\rho_1)$-space;
			c.f. Figure~ \ref{fig:chart1_wo_theta}.}
			\label{fig:chart1_eps_0}
		\end{minipage}
		\begin{minipage}[c]{0.04\textwidth}
			\textcolor{white}{.}
		\end{minipage}
		\begin{minipage}[c]{0.47\textwidth}
			\centering
			\includegraphics[width=\textwidth]{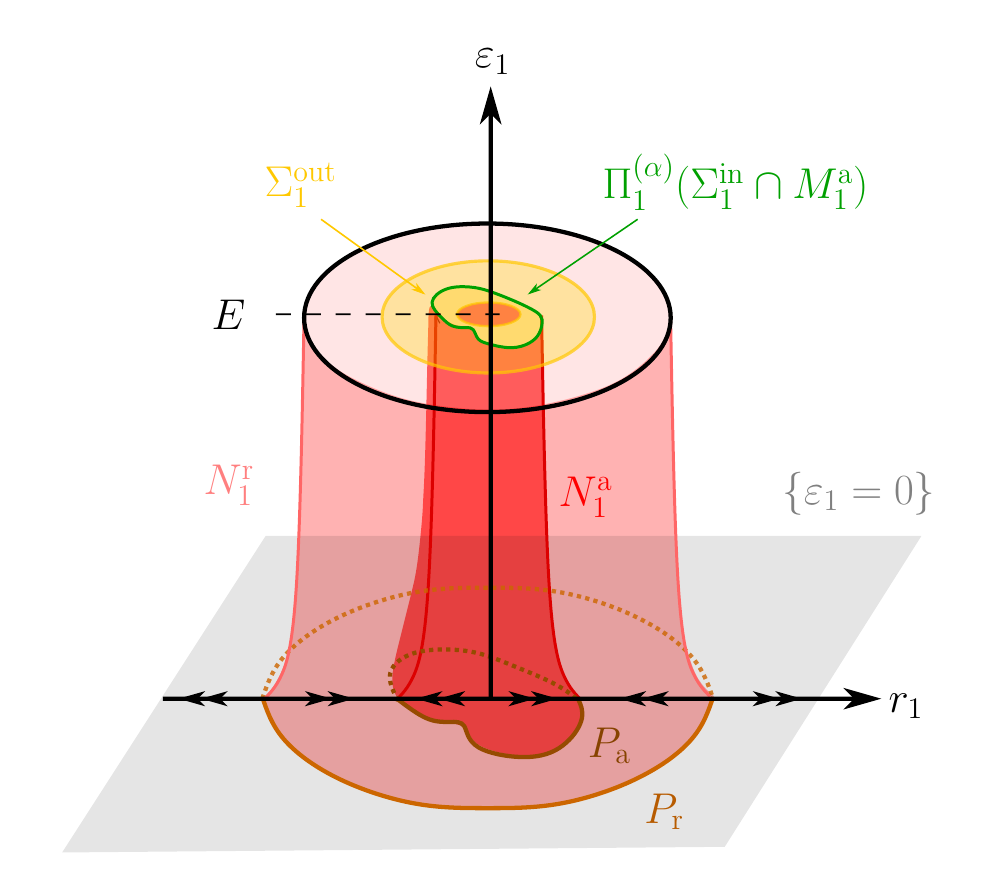}
			\caption{Geometry and dynamics within the invariant hyperplane $\{\rho_1 = 0\}$, projected into $(r_1,\theta_1,\eps_1)$-space; 
			c.f.~Figures \ref{fig:chart1_wo_theta} and \ref{fig:chart1_eps_0}.}
		    \label{fig:chart1_rho_0}
		\end{minipage}
	\end{figure}
	
	\begin{lemma} \label{lem:explicitsolutions}
		\SV{Consider an initial condition $(r_1, \theta_1, \rho_1, \varepsilon_1)(0) = (r_1^\ast, \theta_1^\ast , R, \varepsilon_1^\ast) \in \Sigma_1^{\textup{in}}$ for system \eqref{eq:chart1desing}. Then
		\[
		\begin{split}
		    \theta_1(t_1) &= \theta_1^\ast + \frac{(R\varepsilon_1^\ast)^{\alpha-1} \eps_1^{\ast-2}}{c_0} \left( 1 - \left( 1 - \frac{3}{2} {\eps_1^\ast}^3 \left(\phi(t_1) + t_1\mathcal{O}(R) \right) \right)^{2/3} \right) \left( 1 + \mathcal O(R) \right) \mod{1}, \\
		    \rho_1(t_1) &= R\left(1-\frac{3}{2}{\varepsilon_1^\ast}^3\left( \phi(t_1) + t_1\mathcal{O}(R)\right)\right)^\frac{1}{3} , \\
	    	\varepsilon_1(t_1) &= {\varepsilon_1^\ast}\left( 1 - \frac{3}{2} {\varepsilon_1^\ast}^3\left( \phi(t_1) +t_1\mathcal{O}(R)\right)\right)^{-\frac{1}{3}} ,
		\end{split}
		\]
		where $\phi(t_1) \coloneqq \int_0^{t_1} c(\theta_1(s_1)) \, \dd s_1 = c_0 t_1 + \mathcal O(1)$.
		The notation $\mathcal O(R)$ is used to denote (possibly different) remainder terms which satisfy $|\mathcal O(R)| \leq C R$ for some constant $C > 0$ and all $t_1 \in [0,T_1]$, where
	\[
	T_1 = \frac{2}{3c_0}\left( \frac{1}{{\varepsilon_1^\ast}^3} - \frac{1}{E^3}\right) \left(1 + \mathcal{O}(R) \right) + \mathcal{O}(1)
    \]
    is the transition time taken for the solution to reach $\Sigma_1^{\textup{out}}$.}
	\end{lemma}
	
	\begin{proof}
		Consider system \eqref{eq:chart1desing}. By keeping track of the higher order terms when deriving the equation for $\eps_1'$ in system \eqref{eq:chart1desing} one can show that
		\SVV{\[
			\eps_1' = \frac{1}{2} \eps_1^4 \left( c(\theta_1) - \rho_1 \chi(r_1,\theta_1,\rho_1,\eps_1) \right) ,
			\]
		}where \SJJ{$\chi(r_1,\theta_1,\rho_1,\eps_1) \coloneqq \rho_1^{-1} \widetilde{\mathcal R}_y(\rho_1 r_1, \theta_1, \rho_1^2, \rho_1 \eps_1) = \mathcal O(r_1,\rho_1\SV{,\eps_1})$}. Directly integrating and rearranging a little leads to
		\SVV{\[
			\eps_1(t_1) = {\varepsilon_1^\ast}\left( 1 - \frac{3}{2} {\varepsilon_1^\ast}^3 \left( \phi(t_1) +t_1\psi(t_1)\right)\right)^{-\frac{1}{3}} ,
			\]
		}where \SJJ{$\phi(t_1)$ is defined as in the statement of the lemma and}
		\[
		t_1 \psi(t_1) = -\int_0^{t_1} \rho_1(s_1) \chi(r_1(s_1),\theta_1(s_1),\rho_1(s_1),\eps_1(s_1)) \,  \dd s_1 .
		\]
		The expression for $\rho_1(t_1)$ can be obtained directly from the expression for $\eps_1(t_1)$ using the fact that $\rho_1(t) \varepsilon_1(t) = \eps = R\varepsilon_1^\ast$ is a constant of the motion; recall Remark \ref{rem:const_motion}. One can show that $|\psi(t_1)| = \mathcal O(R)$ by appealing to the fact that $\chi(r_1(s_1),\theta_1(s_1),\rho_1(s_1),\eps_1(s_1))$ is bounded uniformly for all $s_1 \in [0,T_1]$, since $\rho_1 \in [0,R]$, $\eps_1 \in [0,E]$, $\theta \in \R / \mathbb Z$ and \SV{$r_1 \in [\beta_-/R, \beta_+/R]$} (the latter follows from the fact that \SV{$r_1'|_{r_1 = \beta_-/R} > 0$} and \SV{$r_1'|_{r_1= \beta_+/R} < 0$}). Therefore, there is a constant $C>0$ such that $|\psi(t_1)| \leq C R$ for all $t_1 \in [0,T_1]$, as required.
		
		It remains to estimate $\theta_1(t_1)$ and the transition time $T_1$. We have that
		\[
		\theta_1' = \rho_1(t_1)^{\alpha - 1} \eps_1(t_1)^\alpha 
		= (\eps_1^\ast R)^{\alpha - 1} \eps_1(t_1) .
		\]
		\SV{The expression for $\theta_1(t_1)$ is obtained by integrating the expression for $\eps_1(t_1)$ and using \eqref{eq:c_theta_fourier} to estimate 
		$$\phi(t_1) \coloneqq \int_0^{t_1} c_0 + c_\textup{rem}(\theta_1(s_1)) \, \dd s_1 = c_0t_1+ \mathcal{O}(1),$$
		as $\int_0^1 c_\textup{rem}(\theta_1(t_1)) \, \dd t_1 = 0$ and $\theta_1$ is bounded and 1-periodic. To estimate the transition time $T_1$, the boundary constraint $\eps_1(T_1) = E$ is used; c.f.~\cite[Lemma 2.7]{Krupa2001a}. Integrating $\varepsilon_1' = \frac{1}{2}\varepsilon_1^4 (c_0 + c_\textup{rem}(\theta_1) + \mathcal{O}(\rho_1))$ from $\varepsilon_1^*$ to $E$ leads to
		\[
		   \frac{1}{3}\left(\frac{1}{(\eps_1^*)^3}-\frac{1}{E^3}\right) = \frac{c_0}{2} T_1 + \frac{1}{2} \int_0^{T_1} c_\textup{rem}(\theta_1(t_1)) \, \dd t_1 + \frac{1}{2}\int_0^{T_1} \mathcal{O}(\rho_1(t_1)) \, \dd t_1 .
		\]
	    By \eqref{eq:c_theta_fourier}, the second term on the right-hand side is $\mathcal{O}(1)$ and the third term can be estimated by $T_1 \mathcal{O}(R)$. Rearranging yields the desired result. \qedhere
		}
	\end{proof}

	Combining Lemmas \ref{lem:manifolds} and \ref{lem:explicitsolutions} we obtain the following characterisation of the transition map \mbox{\SJJ{$\Pi^{(\alpha)}_1 : \Sigma_1^{\textup{in}} \to \Sigma_1^{\textup{out}}$}}, which summarises the dynamics in chart $K_1$.
	
	\begin{proposition} \label{prop:pi1}
		Fix $E,R > 0$ sufficiently small. Then the map \SJJ{$\Pi^{(\alpha)}_1:\Sigma_1^\textup{in} \rightarrow \Sigma_1^\textup{out}$} is well-defined  
		with the following properties:
		\begin{enumerate}[label=(\alph*)]
			\item \textup{(Asymptotics).} 
			We have\SJJ{
			\begin{equation*}
				\Pi^{(\alpha)}_1(r_1, \theta_1,R,\varepsilon_1) = 
				\left( \Pi^{(\alpha)}_{1,r_1}(r_1,\theta_1,\varepsilon_1), 
				\SV{ h_{\theta_1}^{(\alpha)}(r_1,\theta_1,\eps_1)}, \frac{R}{E} \varepsilon_1 , E \right),
			\end{equation*}
			where
			\begin{align*}
				\Pi^{(\alpha)}_{1,r_1} (r_1,\theta_1,\varepsilon_1) =  \SV{h_{r_1}^{(\alpha)}}\left(
			    \SV{h_{\theta_1}^{(\alpha)}(r_1,\theta_1,\eps_1)}
				, \frac{R}{E}\varepsilon_1,E \right) 
				+ \mathcal{O}\left(\me^{- \SV{\tilde \varrho }/ \varepsilon_1^3}\right)
			\end{align*}
			\SV{where $\tilde \varrho = \frac{2\varrho}{3c_0}$, the constant $\varrho$ and the function $h_{r_1}^{(\alpha)}$ are the same as those in Lemma \ref{lem:manifolds}, \SV{and $h_{\theta_1}^{(\alpha)}(r_1, \theta_1, \eps_1) = \tilde h_{\theta_1}^{(\alpha)}(r_1, \theta_1, \eps_1) \mod 1$, where}}}
			\begin{equation*}
    			\SVV{\tilde h_{\theta_1}^{(\alpha)}(r_1,\theta_1,\eps_1) = \theta_1 + \frac{(R \eps_1)^{\alpha - 1}}{\SV{c_0}} \left( \frac{1}{{\eps_1}^2} + \mathcal{O}(1)\right).}
			\end{equation*}
			\item \textup{(Strong contraction).} The $r_1$-component of \SV{$\Pi_1^{(\alpha)}$} is a strong contraction with respect to $r_1$. More precisely,\SJJ{
			\begin{equation*}
				\frac{\partial \SV{\Pi_{1,r_1}^{(\alpha)}}}{\partial r_1} (r_1,\theta_1,\eps_1) = \mathcal{O}\left(\me^{ - \SV{\tilde \varrho} / \varepsilon_1^3}\right) .
			\end{equation*}
			}The image $\SV{\Pi_1^{(\alpha)}}(\Sigma_1^{\textup{in}}) \subset \Sigma_1^{\textup{out}}$ is a wedge-like region about the intersection $M_1^\textup{a} \cap \Sigma_1^{\textup{out}}$.
		\end{enumerate}
	\end{proposition}
	
	\begin{proof}
		Consider an initial condition $(r_1, \theta_1, R, \eps_1) \in \Sigma_1^{\textup{in}}$. The form of the map $\SV{\Pi_{1}^{(\alpha)}}(r_1, \theta_1,R,\varepsilon_1)$ follows immediately after evaluating the solutions for $\theta_1(t_1), \rho_1(t_1)$, and $\eps_1(t_1)$ at $t_1 = T_1$ using Lemma \ref{lem:explicitsolutions} \SV{and defining $h_{\theta_1}^{(\alpha)}(r_1,\theta_1,\eps_1) = \theta_1(T_1)$}. The expression for $r_1(T_1) = \SV{\Pi_{1,r_1}^{(\alpha)}}(r_1,\theta_1,\eps_1)$ 
		follows from Lemma \ref{lem:manifolds}. Specifically, choosing $E, R$ sufficiently small ensures that the initial condition $(r_1, \theta_1, R, \eps_1)$ is contained in a fast fiber of $M_1^{\textup{a}}$. It follows that
		\[
		\left\| r_1(T_1) - \SV{h_{r_1}^{(\alpha)}}(\theta_1(T_1), \rho_1(T_1), E) \right\| = \SV{\mathcal O(\me^{-\varrho T_1})} 
		\]
		for the constant \SV{$\varrho$} of Lemma \ref{lem:manifolds}. Thus $r_1(T_1) = \SV{h_{r_1}^{(\alpha)}}(\theta_1(T_1), \rho_1(T_1), E) + \SV{\mathcal O(\me^{-\varrho T_1})} $, which yields the expression in Assertion (a) after substituting the expression for $T_1$ in Lemma \ref{lem:explicitsolutions}.
		
		Assertion (b) follows by direct differentiation. The estimate follows since $\SV{h_{r_1}^{(\alpha)}}(\theta_1,\rho_1,\eps_1)$ does not depend on $r_1$.
	\end{proof}
	
	\begin{remark}
		Strictly speaking, the arguments above only guarantee that $\SV{\Pi_1^{(\alpha)}}$ is well-defined for initial conditions with $\eps_1 \in (0,E]$. This is not problematic since we aim to derive results for $\eps > 0$.
	\end{remark}

	\subsection{Dynamics in the Rescaling Chart $K_2$}
	\label{sec:52}
	
	In chart $K_2$ we study solutions close to the extension of the center manifold $M_2^{\textup{a}} = \kappa_{12}(M_1^\textup{a})$. 
	
	\begin{lemma} \label{lem:blowupK2}
		Following the singular time rescaling $\rho_2 \dd t = \dd t_2$ or equivalently, $\rho_2 t = t_2$, the desingularized equations in chart $K_2$ are given by\SJJ{
		\begin{equation} \label{eq:chart2desing}
			\begin{aligned}
				r_2' &= - a(\theta_2) y_2 + b(\theta_2) r_2^2 + \mathcal{O}(\rho_2), \\
				\theta_2' &= \rho_2^{\alpha-1}, \\
				y_2' &= - c(\theta_2) + \mathcal{O}(\rho_2), \\
				\rho_2' &= 0,
			\end{aligned}
		\end{equation}
		}where by a slight abuse of notation we now write $(\cdot)' = \dd / \dd t_2$. Since $0 < \rho_2 = \eps \ll 1$, system \eqref{eq:chart2desing} can also be viewed as a perturbation problem in $(r_2,\theta_2,y_2)$-space as $\rho_2 \to 0$.
	\end{lemma}
	
	\begin{proof}
		This follows immediately after differentiating the defining expressions for local $K_2$ coordinates in~\eqref{eq:chart_coordinates} and applying $\rho_2 t = t_2$.
	\end{proof}

	\SJJ{In order to understand system \eqref{eq:chart2desing}, we first consider the limiting system}
	as $\rho_2 \to 0$, i.e.~\SJJ{
	\begin{equation} \label{eq:riccati_ext}
		\begin{aligned}
			r_2' &= - a(\theta_2) y_2 + b(\theta_2) r_2^2 , \\
			\theta_2' &= 0^{\alpha-1}, \\
			y_2' &= - c(\theta_2) .
		\end{aligned}
	\end{equation}
	}There are two possibilities for the angular dynamics, depending on $\alpha$, since
	\[
	\theta_2' = 
	\begin{cases}
		1 , & \alpha = 1, \\
		0 , & \SJJ{\alpha = 2}.
	\end{cases}
	\]
	\SJJ{We start with the case $\alpha = 2$. In this case, we obtain a $\theta_2$-family of planar systems
	\begin{equation}
		\label{eq:riccatti_planar}
		\begin{split}
			r_2' &= - a y_2 + b r_2^2 , \\
			y_2' &= - c ,
		\end{split}
	\end{equation}
	where $a = a(\theta_2)$, $b = b(\theta_2)$ and $c = c(\theta_2)$ define (parameter dependent) positive constants. The transformation
	\begin{equation}
		\label{eq:ricatti_coord_change}
		t_2 = (a b c)^{-1/3} T_2 , \qquad
		r_2 = \left( \frac{a c}{b^2} \right)^{1/3} R_2 , \qquad
		y_2 = \left( \frac{c^2}{a b} \right)^{1/3} Y_2 ,
	\end{equation}
	leads to
	\begin{equation}
		\label{eq:Ricatti_2}
		\begin{split}
			\frac{\dd R_2}{\dd T_2} &= - Y_2 + R_2^2 , \\
			\frac{\dd Y_2}{\dd T_2} &= - 1,
		\end{split}
	\end{equation}
	which} is precisely the Riccati-type equation which arises within the $K_2$ chart in the analysis of the regular fold point in \cite{Krupa2001a}. \SJJ{For each fixed $\theta \in \R / \mathbb Z$, s}olutions to system \SJJ{\eqref{eq:Ricatti_2} (and therefore also \eqref{eq:riccatti_planar})} can be written in terms of Airy functions, whose asymptotic properties are known \cite{Grasman1987,Mishchenko1975}. The properties that are relevant for our purposes are collected in \cite{Mishchenko1975} and reformulated in a notation similar to ours in \cite{Krupa2001a}. The following result is a direct extension of the latter formulation; we simply append the existing result with the decoupled angular dynamics induced by the equation for $\theta_2$.
	
	\begin{proposition} \label{prop:riccati}
		\SJJ{Fix $\alpha = 2$ and c}onsider the restricted system \eqref{eq:chart2desing}$|_{\rho_2 = 0}$ or, equivalently, the limiting system \eqref{eq:riccati_ext}. The following assertions are true:
		\begin{enumerate}[label=(\alph*)]
			\item Every orbit approaches \SJJ{a} two-dimensional horizontal asymptote/plane $y_2 = y_r$ from $y_2 > y_r$ as $r_2 \rightarrow \infty$. The value of $y_r$ depends on the initial conditions. \label{prop:riccati_a}
			\item \label{cor:riccati_b} There exists a unique, invariant two-dimensional surface
			\begin{equation}
				\label{eq:gamma_2}
				\gamma_2 \coloneqq \left\{\left(r_2,\theta_2, \SJJ{h^{(2)}_{y_2}(r_2)}, 0\right) : r_2 \in \R, \, \theta_2 \in \R / \mathbb Z  \right\} ,
			\end{equation}
			where \SJJ{$h^{(2)}_{y_2}(r_2)$} is smooth with asymptotics\SJJ{
			\begin{equation*}
				\begin{aligned}
					&h^{(2)}_{y_2}(r_2) = \frac{b}{a} r_2^2 + \frac{c}{\SV{2}b} \frac{1}{r_2} + \mathcal{O}\left( \frac{1}{r_2^4}\right) \qquad &&\textup{as } r_2 \rightarrow -\infty, \\
					&h^{(2)}_{y_2}(r_2) = - \left( \frac{c^2}{a b} \right)^{1/3} \Omega_0 + \frac{c}{b} \frac{1}{r_2} + \mathcal{O}\left(\frac{1}{r_2^3} \right) &&\textup{as } r_2 \rightarrow \infty ,
				\end{aligned}
			\end{equation*}
			}and $\Omega_0$ is the constant defined in Theorem \ref{thm:main}.
			\item All orbits with initial conditions to the right of $\gamma_2$ in the $(r_2,y_2)$-plane are backwards asymptotic to the paraboloid \SJJ{$\{(r_2, \theta_2, (b/a) r_2^2,0) : \, r_2 \geq 0, \theta_2 \in \R / \mathbb Z \}$}.
			\item All orbits with initial condition to the left of $\gamma_2$ in the $(r_2,y_2)$-plane are backwards asymptotic to a horizontal asymptote/plane $y_2 = y_l$. Specifically, $y_2(t_2) \to y_l$ from below and $r_2(t_2) \to -\infty$ as $t_2 \to -\infty$. The value of $y_l$ depends on the initial conditions, but satisfies $y_l > y_r$ for each fixed orbit. 
			\item The unique center manifold $N_1^{\textup{a}}$ described in Lemma \ref{lem:manifolds} coincides with the surface $\gamma_2$ where $K_1$ and $K_2$ overlap, i.e.~$\kappa_{12}(N_1^{\textup{a}}) = \gamma_2$ on $\{y_2 > 0\}$.
		\end{enumerate}
	\end{proposition}
	
	\begin{proof}
		See \cite{Mishchenko1975} and in particular \cite[Prop.~2.3]{Krupa2001a}, which cover Assertions (a)-(d) in \SJJ{the} planar case, \SJJ{for the transformed system \eqref{eq:Ricatti_2}. The corresponding statements for system \eqref{eq:riccatti_planar} can be obtained directly from these results using the transformation in \eqref{eq:ricatti_coord_change}.} Since the angular \SJJ{variable $\theta_2 = const.$ when $\alpha = 2$,}
		Assertions (a)-(d) are obtained as higher dimensional analogues from these results after a simple rotation through $\theta_2 \in [0,1)$.
		
		Assertion (e) is a straightforward adaptation of \cite[Prop.~2.6 Assertion (5)]{Krupa2001a}.
	\end{proof}
	
	The Riccati dynamics described in Proposition \ref{prop:riccati} are sketched for the decoupled planar system~\eqref{eq:riccatti_planar} in Figure \ref{fig:chart2_2d}, and for the three-dimensional limiting system \eqref{eq:riccati_ext} in Figure \ref{fig:chart2_3d}.

	\begin{figure}[t!]
		\centering
		\begin{minipage}[t]{0.47\textwidth}
			\centering
			\vspace{0pt}
			\includegraphics[width=\textwidth]{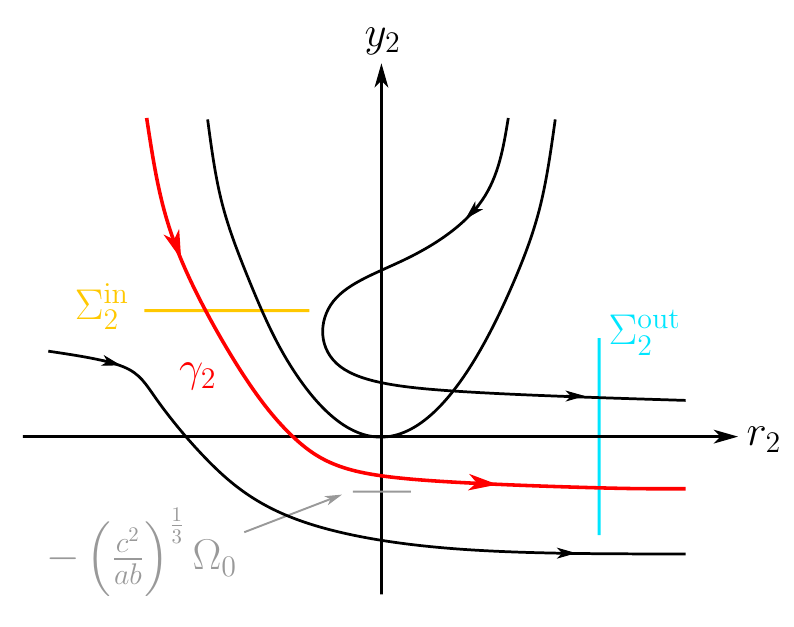}
			\vspace{0pt}
			\caption{Dynamics for the Riccati equation \eqref{eq:riccatti_planar} in the $(r_2,y_2)$-plane \SV{for $\alpha=2$ and a fixed choice of $\theta_2$ ($\theta_2$ is a parameter when $\alpha = 2$ and $\rho_2 = 0$). A qualitatively similar figure is obtained for each $\theta_2 \in \R / \mathbb Z$}. The distinguished solution $\gamma_2$ with asymptotics described by Proposition \ref{prop:riccati} is shown in red. The parabola \SV{$y_2 = (b\SJJJ{(\theta_2)}/a\SJJJ{(\theta_2)}) r_2^2$} which separates solutions with different asymptotic properties is also shown. Projections of the entry and exit sections $\Sigma_2^\textup{in}$ and $\Sigma_2^\textup{out}$ are shown in yellow and cyan, respectively.}
			\label{fig:chart2_2d}
		\end{minipage}
		\begin{minipage}[t]{0.04\textwidth}
			\textcolor{white}{.}
		\end{minipage}
		\begin{minipage}[t]{0.47\textwidth}
			\centering
			\vspace{9pt}
			\includegraphics[width=0.9\textwidth]{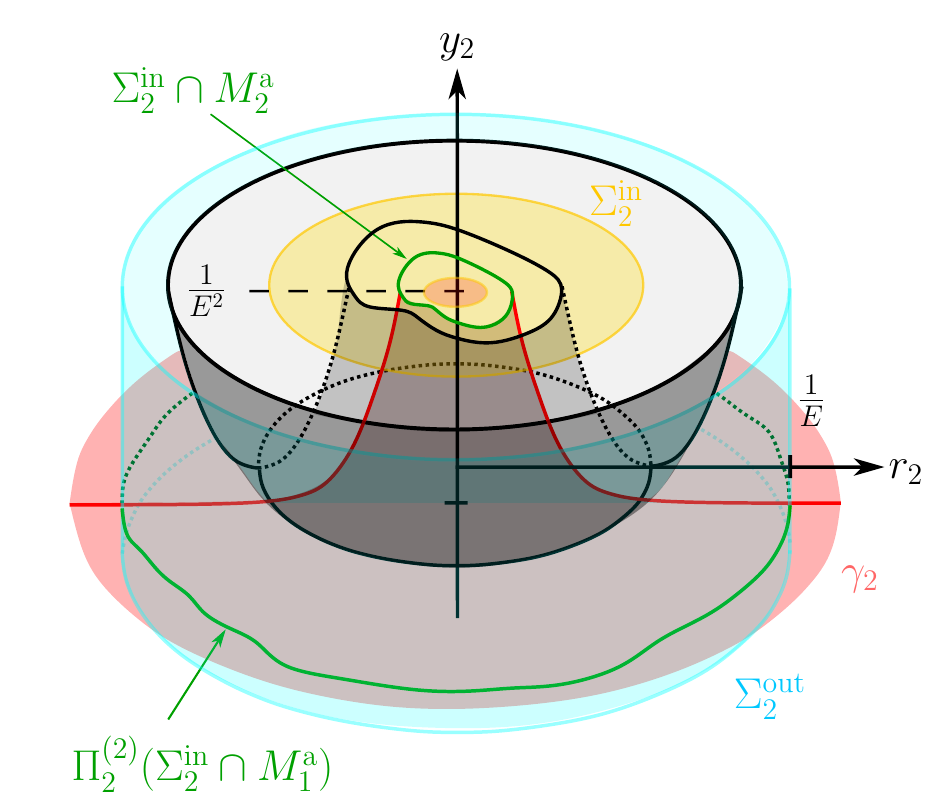}
			\vspace{9pt}
			\caption{Geometry and dynamics of the limiting Riccati equation \eqref{eq:chart2desing}$|_{\rho_2 = 0}$ (or equivalently \eqref{eq:riccati_ext}) in $(r_2,\theta_2,y_2)$-space.
			The two-dimensional surface $\gamma_2$ is shown in shaded red, and coincides with the extension of the two-dimensional center manifold $N_1^\textup{a}$ described in Lemma \ref{lem:manifolds} into chart $K_2$ according to Proposition \ref{prop:riccati} Assertion (e). The evolution of initial conditions in $\SV{\Sigma_2^\textup{in}}$ up to $\Sigma_2^\textup{out}$ is described by Proposition \ref{prop:pi2}, which implies that the image \mbox{$\SV{\Pi_2^{(2)}}(\Sigma_2^\text{in} \cap M_1^\textup{a}) \subset \Sigma_2^\textup{out}$}, shown here in green, is contained within the \SV{surface defined by the graph $y_2 = - (c^2(\theta_2)/(a(\theta_2)b(\theta_2)) )^{1/3} \Omega_0 + E$ over $\theta_2 \in \R / \mathbb Z$} (see also Lemma \ref{lem:hy20}).}
			\label{fig:chart2_3d}
		\end{minipage}
	\end{figure}

	\
	
	\SJJ{
		We now consider the case $\alpha = 1$, for which system \eqref{eq:riccati_ext} can be written as the non-autonomous planar system
		\begin{equation}
			\label{eq:ricatti_na}
			\begin{split}
				r_2' &= - \tilde a(t_2) y_2 + \tilde b(t_2) r_2^2 , \\
				y_2' &= - \tilde c(t_2) ,
			\end{split}
		\end{equation}
		where the functions
		\[
		\tilde a(t_2) \coloneqq a(\theta_2(0) + t_2 \mod 1) , \quad
		\tilde b(t_2) \coloneqq b(\theta_2(0) + t_2 \mod 1) , \quad
		\tilde c(t_2) \coloneqq c(\theta_2(0) + t_2 \mod 1) ,
		\]
		are smooth, positive and $1$-periodic in $t_2$ \SJJJ{due to the positivity and $1$-periodicity of $a(\theta)$, $b(\theta)$ and $c(\theta)$; recall Proposition \ref{prop:normal_form}}.
		Using
		\begin{equation}
			\label{eq:y2_soln}
			y_2(t_2) = y_2(0) - \varphi(t_2), \qquad 
			\varphi(t_2) \coloneqq \int_0^{t_2} \tilde c(\xi) \, \dd \xi ,
		\end{equation}
		we may write \eqref{eq:ricatti_na} as a Riccati equation
		\begin{equation}
			\label{eq:ricatti}
			\frac{\dd r_2}{\dd \SV{t_2}} = \tilde a(t_2) (\varphi(t_2) - y_2(0)) + \tilde b(t_2) r_2^2 .
		\end{equation}
		We now define the constants
		\begin{equation}
			\label{eq:AB_bounds}
			\mathcal A_- \coloneqq \inf_{t_2 \in [0,1)} \frac{\tilde a(t_2)}{\tilde c(t_2)} , \qquad
			\mathcal A_+ \coloneqq \sup_{t_2 \in [0,1)} \frac{\tilde a(t_2)}{\tilde c(t_2)} , \qquad
			\mathcal B_- \coloneqq \inf_{t_2 \in [0,1)} \frac{\tilde b(t_2)}{\tilde c(t_2)} , \qquad
			\mathcal B_+ \coloneqq \sup_{t_2 \in [0,1)} \frac{\tilde b(t_2)}{\tilde c(t_2)} ,
		\end{equation}
		and use equation \eqref{eq:ricatti} in the derivation of the following result.
		
		\begin{proposition} \label{prop:ricatti_alpha_1}
			Fix $\alpha = 1$ and consider the restricted system \eqref{eq:chart2desing}$|_{\rho_2 = 0}$ or, equivalently, the limiting system \eqref{eq:riccati_ext}. For sufficiently small but fixed $E, R > 0$, solutions with initial conditions
			\[
			\left( r_2(0), \theta_2(0), y(0), 0 \right) \in \widetilde{\Sigma}_2^{\textup{in}} \coloneqq \left\{ \left(r_2, \theta_2, E^{-2}, 0 \right) : r_2 \in [\beta_- / ER, \beta_+ / ER], \theta_2 \in \R / \mathbb Z \right\}
			\]
			reach the set
			\[\SV{
			\widetilde{\Sigma}_2^{\textup{out}} \coloneqq \left\{ \left( E^{-1}, \theta_2, y_2, 0 \right) : \theta_2 \in \R / \mathbb Z, y_2 \in [-\nu E^{-2}, \nu E^{-2}] \right\} }
			\]
			in finite time\SV{, where $\nu>0$ is a constant}. In particular, we have that $\kappa_{12}(N_1^{\textup{a}} \SV{\cap \Sigma_1^\textup{out}}) \subset \widetilde{\Sigma}_2^{\textup{in}} \cap \{r_2 \in (\beta_- / ER, \beta_+ / ER)\}$, where $N_1^\textup{a}$ is the unique center manifold described in Lemma \ref{lem:manifolds}.
		\end{proposition}
		
		\begin{proof}
			The idea is to bound solutions to the Riccati equation \eqref{eq:ricatti} between upper and lower solutions with known asymptotics as $r_2 \to \infty$. Specifically, upper and lower solutions will be obtained as concatenations of solutions to Riccati equations of the form
			\begin{equation}
				\label{eq:ricatti_const}
				\frac{\dd r_2}{\dd t_2} = \tilde c(t_2) \left( \lambda (\SV{\varphi(t_2)} - y_2(0)) + \mu r_2^2 \right) ,
			\end{equation}
			where $\lambda, \mu \neq 0$ are suitably chosen constants. Note that it is equivalent to compare solutions of the two planar autonomous systems \eqref{eq:ricatti_na} and
			\begin{equation}
				\begin{split}
					r_2' &= \tilde c(t_2) (\SV{-} \lambda y_2 + \mu r_2^2) , \\
					y_2' &= - \tilde c(t_2) ,
				\end{split}
			\end{equation}
			which has asymptotic properties similar to those of system \eqref{eq:riccatti_planar} with $a = \lambda$, $b = \mu$ and $c=1$ (the factor of $\tilde c(t_2)$ does not effect the phase portrait since $dr_2/dy_2$ has no explicit time dependence).
			
			We start by considering solutions of system \eqref{eq:ricatti_na} with initial conditions $(r_2(0),E^{-2})$ corresponding to points in $\widetilde{\Sigma}_2^{\textup{in}}$ with $r_2(0) \in [\beta_- / ER, \beta_+ / ER]$. Our first task is to bound the $r_2$-component of solutions to~\eqref{eq:ricatti_na} on the interval $t_2 \in [0,T_0]$, where $T_0 > 0$ is the unique solution to $\varphi(T_0) = E^{-2}$\SV{. If a solution exists until $t_2 = T_0$ (and has not blown up in a finite time smaller than $T_0$), the definition of $T_0$ corresponds} to the intersection of solutions with \SV{$\{y_2 = 0\}$} (recall that $y_2(t_2) = E^{-2} - \varphi(t_2)$, see \eqref{eq:y2_soln}). \SV{We get a} left-bound \SV{to the Riccati equation \eqref{eq:ricatti}} by identifying a lower solution $\underline{r_2}(t_2) \leq r_2(t_2)$. This is obtained by letting $\underline{r_2}(t_2)$ be a solution of \eqref{eq:ricatti_const} with $(\underline{r_2}(0), y_2(0)) = (\beta_- / ER, E^{-2})$ and $\lambda, \mu$ replaced by
			\[
			\underline \lambda = \mathcal A_+ , \qquad
			\underline \mu = \mathcal B_- ,
			\]
			respectively. This ensures that $\underline{r_2}(0) \leq r_2(0)$ and
			\begin{equation}
				\label{eq:lower_inequality}
				\frac{\dd \underline{r_2}}{\dd t_2} = \tilde c(t_2) \underline \lambda \SV{(\varphi(t_2) - E^{-2})} + \tilde c(t_2) \underline \mu \underline {r_2}^2 
				\leq \SV{\tilde a(t_2) (\varphi(t_2)-E^{-2})} + \tilde b (t_2) \underline{r_2}^2
			\end{equation}
			for all $t_2 \in [0,T_0]$, as \SV{$\varphi(t_2) - E^{-2} \leq 0$ for $t_2 \in [0,T_0]$}. \SV{\SJJ{Solutions of equation  \eqref{eq:ricatti_const} with $\lambda = \underline \lambda$ and $\mu = \underline \mu$ are described by Proposition \ref{prop:riccati} after a simple transformation similar to that in \eqref{eq:ricatti_coord_change}. In particular, there exists
			an} invariant two-dimensional surface \SJJ{$\underline \gamma_2$ with properties similar to the surface $\gamma_2$, which intersects} $\{y_2 = 0\}$. \SJJ{C}hoosing \SJJ{$E$ and $R$} sufficiently small \SJJ{guarantees} that the initial condition $\underline {r_2}(0) = \beta_- / ER$ is smaller than the ${r_2}$-coordinate of the \SJJ{intersection $\underline \gamma_2 \cap \{y_2 = E^{-2}\}$. This implies that the solution $(\underline{r_2}(t_2), y_2(t_2))$ also intersects $\{y_2 = 0\}$.}}
		    Thus, \SV{using \eqref{eq:ricatti}, we conclude that} $\underline{r_2}(t_2)$ is a lower solution for $r_2(t_2)$, with $\underline{r_2}(t_2) \leq r_2(t_2)$ for all $t_2 \in [0,T_0]$.
			
			An upper solution $\overline{r_2}(t_2)$ with the initial condition $(\overline{ r_2}(0), y_2(0)) = (\beta_+ / ER, E^{-2})$ can be constructed in a similar way, by identifying $\overline{r_2}(t_2)$ as the solution
			of the Riccati equation \eqref{eq:ricatti_const} with $\lambda, \mu$ replaced by
			\[
			\overline \lambda = \mathcal A_- , \qquad
			\overline \mu = \mathcal B_+ ,
			\]
			respectively. We have that $\overline {r_2}(0) \geq r_2(0)$ and
			\begin{equation}
				\label{eq:upper_inequality}
				\frac{\dd \overline{r_2}}{\dd t_2} = \tilde c(t_2) \overline \lambda \SV{(\varphi(t_2)-E^{-2})} + \tilde c(t_2) \overline \mu \overline{r_2}^2
				\geq \SV{\tilde a(t_2)  ( \varphi(t_2)-E^{-2})} + \tilde b(t_2) \overline{r_2}^2
			\end{equation}
			for all $t_2 \in [0,T_0]$. \SV{There are now two \SJJ{different} possibilities, depending on \SJJ{whether $\overline{r_2}(t_2)$} blows up in finite time \SVVV{for $\overline{r_2} \rightarrow \infty$} (case A), or exists \SJJ{up to the intersection with $\{y_2 = 0\}$, i.e.~}for all $t_2 \in [0,T_0]$ (case B). \SVVV{Blow-up for $\overline{r_2} \rightarrow - \infty$ is not possible as solutions of equation \eqref{eq:ricatti_const} with $\lambda = \overline \lambda$ and $\mu = \overline \mu$ are described by Proposition \ref{prop:riccati} after a simple transformation similar to that in \eqref{eq:ricatti_coord_change}. \SJJJ{In particular}, there exists
			an invariant two-dimensional surface $\overline \gamma_2$ with properties similar to the surface $\gamma_2$, which intersects $\{y_2 = 0\}$. Choosing $E$ and $R$ sufficiently small guarantees that the initial condition $\overline {r_2}(0) = \beta_+ / ER$ is larger than the ${r_2}$-coordinate of the intersection $\overline \gamma_2 \cap \{y_2 = E^{-2}\}$, implying that $\overline{r_2}$ cannot blow up as $\overline{r_2} \rightarrow - \infty$. Hence it is sufficient to study case A and case B in the following.}
			
			In case A, the solution \SJJ{$\overline{r_2}(t_2)$}
			blows up in finite time $T_0^\ast < T_0$. The analogous statement to Proposition~\ref{prop:riccati} for the Riccati equation with $a = \overline{ \lambda}$, $b = \overline{\mu}$, $c = 1$ implies that \SJJ{$(\overline{r_2}(t_2), y_2(t_2))$} converges to a horizontal asymptote $y_2 = y_2^\ast$, \SJJ{therefore intersecting $\{r_2 = E^{-1}\}$ transversally (assuming $E > 0$ is small enough).} Using \eqref{eq:ricatti}, we conclude that $\overline{r_2}(t_2) \geq r_2(t_2)$ for all $t_2 \in [0,T_0^\ast)$. As $y_2' <0$ it holds \SJJ{$y_2(t_2) \leq y_2^\ast$} for all \SJJ{$t_2 \geq T_0^*$}, as long as the solution of \eqref{eq:ricatti_na} exists. We show later that $r_2(t_2)$ can also be bounded by a lower solution transversally intersecting $\{r_2 = E^{-1}\}$, such that also $r_2(t_2)$ transversally intersects $\{r_2 = E^{-1}\}$. By choosing \SVVV{$E$ sufficiently small such that $y_2^\ast E^2<\nu$}, we can guarantee that the intersection takes place in $\{r_2 = E^{-1}, y_2 < \nu E^{-2}\}$.

			In case B, $\overline{r_2}(t_2)$
			exists for all $t_2 \in [0,T_0]$ and hence intersects $\{y_2 = 0\}$. Using \eqref{eq:ricatti}, we can conclude that $\overline{r_2}(t_2) \geq r_2(t_2)$ for all $t_2 \in [0,T_0]$}. Combining this with the \SV{results for the lower solution $\underline{r_2}(t_2)$}, we have that
			\[
			\underline{r_2}(t_2) \leq r_2(t_2) \leq \overline{r_2}(t_2) ,
			\]
			on $t_2 \in [0,T_0]$. In particular, along $y_2 = 0$ we have
			\[
			\underline{r_2}(T_0) \leq r_2(T_0) \leq \overline{r_2}(T_0) .
			\]
			The situation is sketched in Figure \ref{fig:chart2_alpha1A} for case A and in Figure \ref{fig:chart2_alpha1B} for case B.
			
			The functions $\underline{r_2}$ and $\overline{r_2}$ \SV{(the latter only in case B)} do not necessarily define lower and upper solutions when $y_2 < 0$ (i.e.~when $t_2 > T_0$), since the inequalities \eqref{eq:lower_inequality} and \eqref{eq:upper_inequality} may no longer be satisfied. We can, however, connect to different lower and upper solutions which we denote by $\underline{\tilde r_2}(t_2)$ and $\overline{\tilde r_2}(t_2)$ respectively, with initial conditions $\underline{\tilde r_2}(T_0) = \underline{r_2}(T_0)$ and $\overline{\tilde r_2}(T_0) = \overline {r_2}(T_0)$, obtained as segments of solutions to \eqref{eq:ricatti_const} with $\lambda, \mu$ replaced by
			\[
			\underline{\tilde \lambda} = \mathcal A_- , \qquad
			\underline{\tilde \mu} = \mathcal B_- , 
			\]
			and
			\[
			\overline{\tilde \lambda} = \mathcal A_+ , \qquad
			\overline{\tilde \mu} = \mathcal B_+ ,
			\]
			respectively. \SV{\SJJ{This time, we may apply} Proposition \ref{prop:riccati} for the Riccati equation with $a = \underline{\tilde \lambda}$, $b = \underline{\tilde \mu}$, $c=1$ \SJJ{and} $a = \overline{ \tilde \lambda}$, $b = \overline{ \tilde \mu}$, $c=1$, \SJJ{respectively. We find that} $\underline{ \tilde r_2}(t_2)$, $\overline{ \tilde r_2}(t_2) \rightarrow \infty$, \SJJ{and that orbits of the corresponding planar systems} approach horizontal asymptotes $y_2 = y_{r_\pm}$ \SJJ{in the $(r_2,y_2)$-plane} from above. \SJJ{Thus,} $\underline{ \tilde r_2}(t_2)$ and $\overline{ \tilde r_2}(t_2)$ blow up in finite time. The relation $y_2(t_2) = E^{-2} - \varphi(t_2)$ leads to the following implicit equations for the blow-up times $t_2^{\pm}$:
			\[
			\varphi(t_2^-) \coloneqq E^{-2} - y_{r_-} , \qquad
			\varphi(t_2^+) \coloneqq E^{-2} - y_{r_+}.
			\]
			With initial conditions $\underline{\tilde r_2}(T_0) = \underline{r_2}(T_0)$ and $\overline{\tilde r_2}(T_0) = \overline {r_2}(T_0)$, the constructed upper and lower solutions guarantee that $\underline{\tilde r}_2(t_2) \leq r_2(t_2) \leq \overline{\tilde r_2}(t_2)$ for all $t_2 \in [T_0, t_2^+]$, see Figure \ref{fig:chart2_alpha1B}. Consequently there exists $t_2^c \in [t_2^+,t_2^-]$, such that $r_2(t_2)  \rightarrow \infty$ as $t_2 \rightarrow t_2^c$ from below and $r_2(t_2)$ crosses the hyperplane $\{r_2 = E^{-1}\}$. 
			
			In case B, \SVVV{choosing $E$ sufficiently small such that} $\nu > \vert y_{r_-} \vert E^2$ guarantees that $\underline{\tilde r}_2(t_2)$ and $\overline{\tilde r_2}(t_2)$ and thus also $r_2(t_2)$ \SJJ{transversally intersect}  $\{r_2 = E^{-1}\}$ in the section $\tilde \Sigma_2^\text{out}$.
			
			In case A, \SVVV{choosing $E$ sufficiently small such that} such that $\nu>\max\{\vert y_{r_-} \vert, y^\ast\} E^2$ guarantees that $\underline{\tilde r}_2(t_2)$ and $\overline{r_2}(t_2)$ and thus also $r_2(t_2)$ transversally intersect  $\{r_2 = E^{-1}\}$ in the section $\tilde \Sigma_2^\text{out}$.}
		\end{proof}
	}
	\SVVV{
	\begin{remark}
	    \SJJJ{As pointed out by an anonymous referee, the proof of Proposition \ref{prop:ricatti_alpha_1} can be simplified by fixing $\nu \geq 1$, in which case the existence of lower solutions $\underline{r_2}$ and $\underline{\tilde r_2}$ and the fact that $y_2' < 0$ are sufficient to show that 
	    solutions intersect $\widetilde{\Sigma}_2^\textup{in}$. 
	    The existence of upper solutions $\overline{r_2}$ and $\overline{\tilde r_2}$ allows for greater control of the solutions when $\nu < 1$, but it is not strictly necessary for the proof of Theorem \ref{thm:main}.}
	\end{remark}
	}
		\begin{figure}[t!]
		\centering
		\begin{minipage}[c]{0.47\textwidth}
			\centering
			\vspace{0pt}
			\includegraphics[width=\textwidth]{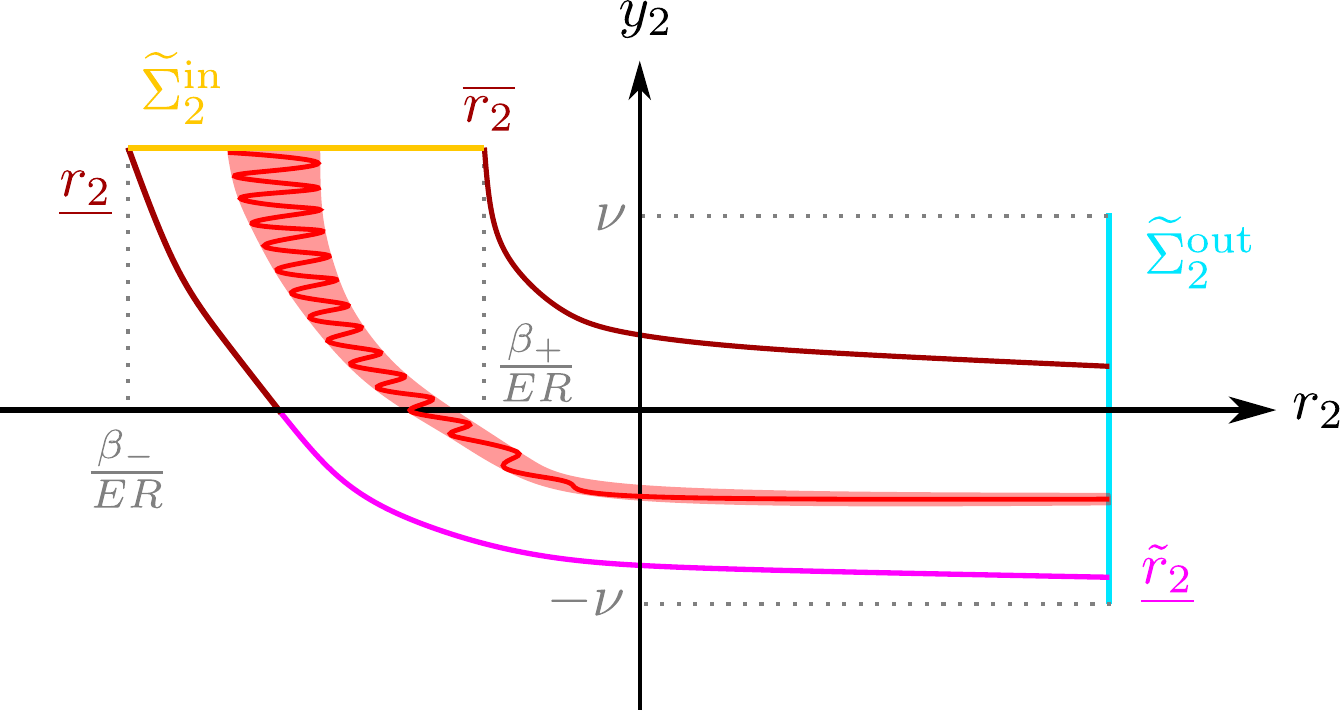}
			\caption{\SV{Case A of the Riccati dynamics for $\alpha = 1$. The upper solution $\overline{r_2}$ does not intersect with $\{y_2 = 0\}$ and converges to a horizontal asymptote intersecting $\Sigma_2^\text{out}$ transversally. The lower solution $\underline{r_2}$ intersects $\{y_2 = 0\}$ at $\underline{r_2}(T_0)$, where it connects to the lower solution $\underline{ \tilde r_2}$, which also converges to a horizontal asymptote intersecting $\Sigma_2^\text{out}$ transversally. Solutions to the Riccati equation \eqref{eq:ricatti} with initial conditions in $\kappa_{12}(N_1^{\textup{a}} \cap \Sigma_1^\textup{out}) \subset \widetilde{\Sigma}_2^{\textup{in}}$ are sketched in shaded red and have $r_2$-coordinates which are bounded between $\underline{ r_2}$ and $\overline{r_2}$ ($\underline{ \tilde r_2}$ and $\overline{r_2}$) when $y_2 \geq 0$ ($y_2 \leq 0)$. For fixed initial condition $\theta_2^\ast$, a sample trajectory is shown in red.}}
			\label{fig:chart2_alpha1A}
		\end{minipage}
		\begin{minipage}[c]{0.04\textwidth}
			\textcolor{white}{.}
		\end{minipage}
		\begin{minipage}[c]{0.47\textwidth}
			\centering
			\vspace{0pt}
			\includegraphics[width=\textwidth]{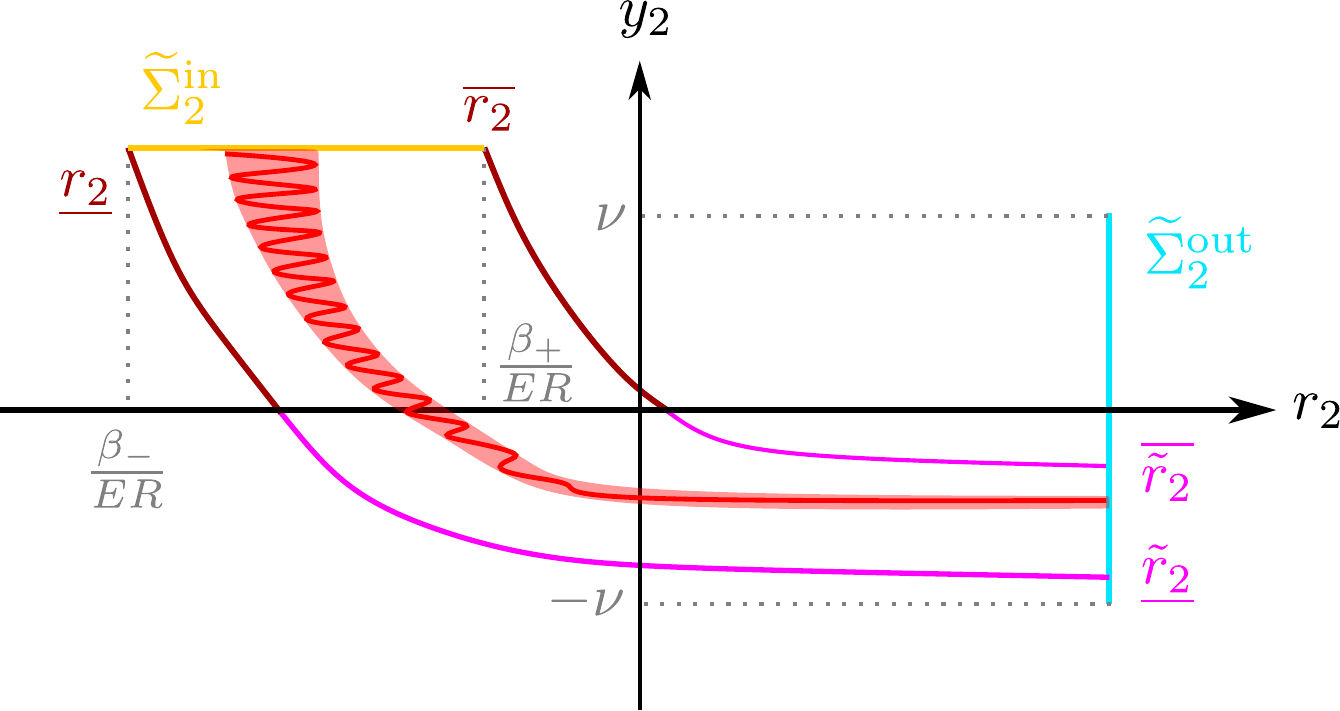}
			\caption{\SV{Case B of the Riccati dynamics for $\alpha = 1$. The upper solution $\overline{r_2}$ intersects $\{y_2 = 0\}$ at $\overline{r_2}(T_0)$, where it connects to the upper solution $\overline{ \tilde r_2}$, which converges to a horizontal asymptote intersecting $\Sigma_2^\text{out}$ transversally. The lower solution $\underline{r_2}$ intersects $\{y_2 = 0\}$ at $\underline{r_2}(T_0)$, where it connects to the lower solution $\underline{ \tilde r_2}$, which also converges to a horizontal asymptote intersecting $\Sigma_2^\text{out}$ transversally. Solutions to the Riccati equation \eqref{eq:ricatti} with initial conditions in $\kappa_{12}(N_1^{\textup{a}} \cap \Sigma_1^\textup{out}) \subset \widetilde{\Sigma}_2^{\textup{in}}$ are sketched in shaded red and have $r_2$-coordinates which are bounded between $\underline{ r_2}$ and $\overline{r_2}$ ($\underline{ \tilde r_2}$ and $\overline{\tilde r_2}$) when $y_2 \geq 0$ ($y_2 \leq 0)$. For a fixed initial condition $\theta_2^*$, a sample trajectory is shown in red.}}
			\label{fig:chart2_alpha1B}
		\end{minipage}
	\end{figure}
	
	Proposition\SJJ{s \ref{prop:riccati} and \ref{prop:ricatti_alpha_1}} can be used to describe the limiting behaviour of the $y_2$-component of the map \SJJ{$\Pi^{(\alpha)}_2 : \Sigma_2^{\textup{in}} \to \Sigma_2^{\textup{out}}$} induced by the flow of initial conditions in\SV{
	\[
	\Sigma_2^{\textup{in}} \coloneqq \left\{ \left(r_2, \theta_2, E^{-2}, \rho_2 \right) : r_2 \in [\beta_- / ER, \beta_+ / ER], \theta_2 \in \R / \mathbb Z, \rho_2 \in [0, ER] \right\},
	\]
	}which is related to the exit section in chart $K_1$ via $\Sigma_2^{\textup{in}} = \kappa_{12}(\Sigma_1^{\textup{out}})$, up to the exit section
	\[
	\Sigma_2^{\textup{out}} \coloneqq \left\{ \left( E^{-1}, \theta_2, y_2, \rho_2 \right) : \theta_2 \in \R / \mathbb Z, y_2 \in [-\nu E^{-2}, \nu E^{-2}] , \rho_2 \in [0, ER] \right\} ,
	\]
	where $\nu > 0$ is \SV{the constant in Proposition \ref{prop:ricatti_alpha_1}}.
	
	\

	Similarly to the analysis in $K_1$, the remaining components of the image \SJJ{$\Pi^{(\alpha)}_2(r_2, \theta_2, E^{-2}, \rho_2)$}
	can be \SJJ{estimated} directly.
	
	\begin{lemma} \label{lem:y2theta2}
		Consider an initial condition $(r_2,\theta_2,E^{-2},\rho_2)(0) = (r_2^\ast, \theta_2^\ast, E^{-2}, \rho_2) \in \Sigma_2^{\textup{in}}$ for system~\eqref{eq:chart2desing}. Then\SJJ{
		\begin{align*}
			\theta_2(t_2) &= \theta_2^\ast + \rho_2^{\alpha-1}t_2 \hspace{5mm}\mod{1}, \\  
			y_2(t_2) &= 
			\begin{cases}
				E^{-2} - \varphi(t_2) + \SV{\mathcal O(\rho_2 t_2)} , & \alpha = 1 , \\
				E^{-2} - \SV{t_2 (c(\theta^\ast_2)+ \mathcal{O}(\rho_2)) ,} & \SV{\alpha =2,}
			\end{cases}
		\end{align*}
		where $c(\theta^\ast_2) = const. > 0$ and $\varphi(t_2)$ is the function defined in \eqref{eq:y2_soln}.}
	\end{lemma}
	
	\begin{proof}
		The solutions are obtained by direct integration.
	\end{proof}
	
	\SJJ{In the following we define the function  $h_{y_2}^{(\alpha)}(r_2,\theta_2)$ so that (i) 
	\[
	\Pi_2^{(1)}(\kappa_{12}(N_1^{\textup{a}} \cap \Sigma_2^{\textup{out}})) = (E^{-1}, \theta_2(T_2), h_{y_2}^{(1)}(E^{-1}, \theta_2), 0) \in \Sigma_2^{\textup{out}}, 
	\]
	(Proposition \ref{prop:ricatti_alpha_1} guarantees the existence of this intersection), and (ii) $h_{y_2}^{(2)}(E^{-1},\theta_2)$ is precisely the function $h_{y_2}^{(2)}$ defining the special Riccati solution in Proposition \ref{prop:riccati} (we simply suppressed the parameter dependence on $\theta_2$ in the notation there).} Combining Proposition\SJJ{s \ref{prop:riccati}-\ref{prop:ricatti_alpha_1}} and Lemma \ref{lem:y2theta2}, we obtain the following characterization of the transition map \SJJ{$\Pi^{(\alpha)}_2 : \Sigma_2^{\textup{in}} \to \Sigma_2^{\textup{out}}$}, which summarizes the dynamics in chart $K_2$.

	\begin{proposition} \label{prop:pi2}
		\SJJ{For sufficiently small but fixed
		$E, R > 0$, $\Pi^{(\alpha)}_2 : \Sigma_2^{\textup{in}} \to \Sigma_2^{\textup{out}}$} is a well-defined diffeomorphism of the form\SJJ{
		\[
		\Pi_2^{(\alpha)} \left(r_2,\theta_2, E^{-2},\rho_2\right) =
		\left( E^{-1}, h_{\theta_2}^{(\alpha)}(r_2, \theta_2, \rho_2) , h_{y_2}^{(\alpha)}(E^{-1}, \theta_2) + \mathcal O(r_2 - r_{2,c}, \rho_2) ,\rho_2\right) ,
		\]
		where
		$r_{2,c}$ is the \SJJ{(generally $\theta_2$-dependent)} $r_2$-coordinate of the intersection $\kappa_{12}(N_1^{\textup{a}} \cap \SV{\Sigma_1^{\textup{out}}}) \in  \Sigma_2^{\textup{in}}$, which is defined implicitly via the relation $h_{y_2}^{(\alpha)}(r_{2,c}, \theta_2) = E^{-2}$, and $h_{\theta_2}^{(\alpha)}(r_2, \theta_2, \rho_2) = \tilde h_{\theta_2}^{(\alpha)}(r_2, \theta_2, \rho_2) \mod 1$ where}
		\[
		\SJJ{\tilde h_{\theta_2}^{(\alpha)}(r_2, \theta_2, \rho_2) = 
		\begin{cases}
		   \SV{\theta_2} +  \mathcal O(1), & \alpha = 1, \\
		    \theta_2 + \frac{\rho_2}{c(\theta_2)} \left(E^{-2} - \SV{h_{y_2}^{(2)}}(E^{-1}, \theta_2)\right) + \rho_2 \mathcal{O}(r_2 - r_{2,c},\rho_2) , & \alpha = 2 .
		\end{cases}}
		\]
	\end{proposition}
	
	\begin{proof}
		\SJJ{We start with the case $\alpha = 2$ and l}et $(r_{2,c}, \theta_2, E^{-2}, 0)$ denote the $K_2$ coordinates of the intersection $\gamma_2 \cap \Sigma_2^{\textup{in}}$. It follows from Proposition \ref{prop:riccati} and Lemma \ref{lem:y2theta2} that 
		\[
		\SV{\Pi_2^{(2)}}(r_{2,c}, \theta_2, E^{-2}, 0) = 
		\left( E^{-1}, \theta_2 + \rho_2 T_2 \hspace{-2mm} \mod{1} , \SJJ{h^{(2)}_{y_2}(E^{-1},\theta_2)}, 0 \right) \in \textup{Int} \ \Sigma_2^{\textup{out}} ,
		\]
		where $T_2 > 0$ is the transition time taken from $\Sigma_2^{\textup{in}}$ to reach $\Sigma_2^{\textup{out}}$, $\textup{Int} \ \Sigma_2^{\textup{out}}$ denotes the interior of $\Sigma_2^{\textup{out}}$ in $\{E^{-1}\} \times \R / \mathbb Z \times \R \times \R_{\geq 0}$, and the containment in $\textup{Int} \ \Sigma_2^{\textup{out}}$ follows from the asymptotics \SJJ{$\SV{h_{y_2}^{(2)}}(\eta^{-1}, \theta_2) = - (c^2 / ab)^{1/3} \Omega_0 + (c/b) \eta + \mathcal{O}(\eta^3)$} as $\eta \rightarrow 0^+$ in Proposition \ref{prop:riccati} Assertion (b) \SJJ{(note that $a = a(\theta_2)$, $b = b(\theta_2)$ and $c = c(\theta_2)$ are constant functions of the initial value $\theta_2$ here)}. The transition time \SV{for an initial condition with $\rho_2 = 0$} can be estimated using Lemma \ref{lem:y2theta2}. In particular,\SJJ{
		\begin{align*}
			y_2(T_2) = h^{(2)}_{y_2}(E^{-1},\theta_2) = y_2(0) - c(\theta_2) T_2 =  E^{-2} - c(\theta_2) T_2 \qquad \Rightarrow \qquad T_2 = \frac{E^{-2} - h^{(2)}_{y_2}( E^{-1}\SV{,\theta_2})}{c(\theta_2)} ,
		\end{align*}
		}which implies that\SJJ{
		\begin{equation}
			\label{eq:Pi_2_gamma_2}
			\SV{\Pi_2^{(2)}}(r_{2,c}, \theta_2, E^{-2}, 0) = 
			\left( E^{-1}, \theta_2 + \frac{\rho_2}{c(\theta_2)} \left( E^{-2} - h^{(2)}_{y_2}(E^{-1},\theta_2) \right) \hspace{-2mm} \mod{1} , h^{(2)}_{y_2}(E^{-1}, \theta_2), 0 \right) .
		\end{equation}
		}Since the transition time $T_2$ is finite and system \eqref{eq:riccati_ext} is a regular perturbation problem, a neighbourhood of $(r_{2,c}, \theta_2, E^{-2}, 0) \in \Sigma_2^{\textup{in}}$ is mapped diffeomorphically to a neighbourhood of $\gamma_2 \cap \Sigma_2^{\textup{out}}$ in $\Sigma_2^{\textup{out}}$. The form of the map in Proposition \ref{prop:pi2} follows.
		
		\SJJ{Now fix $\alpha = 1$. In this case the transition time $T_2$ satisfies $0 < T_2^- \leq T_2 \leq T_2^+ < \infty$, where $T_2^\pm$ are the transition times associated to the lower/upper solutions used in the proof of Proposition \ref{prop:ricatti_alpha_1}. The result from here follows similarly to the case $\alpha = 2$, using the fact that $T_2$ is finite and that system \eqref{eq:riccati_ext} is a regular perturbation problem.}
	\end{proof}
	
	The extension of $M_2^{\textup{a}}$ up to $\Sigma_2^{\textup{out}}$, as described by Proposition \ref{prop:pi2}, is shown in projections onto $(r_2,\theta_2,y_2)$- and $(r_2,y_2,\rho_2)$-space \SV{for $\alpha = 2$} in Figures \ref{fig:chart2_3d} and \ref{fig:chart2_wo_theta}, respectively.
	
	\begin{figure}[t!]
		\centering
		\includegraphics[width=0.8\textwidth]{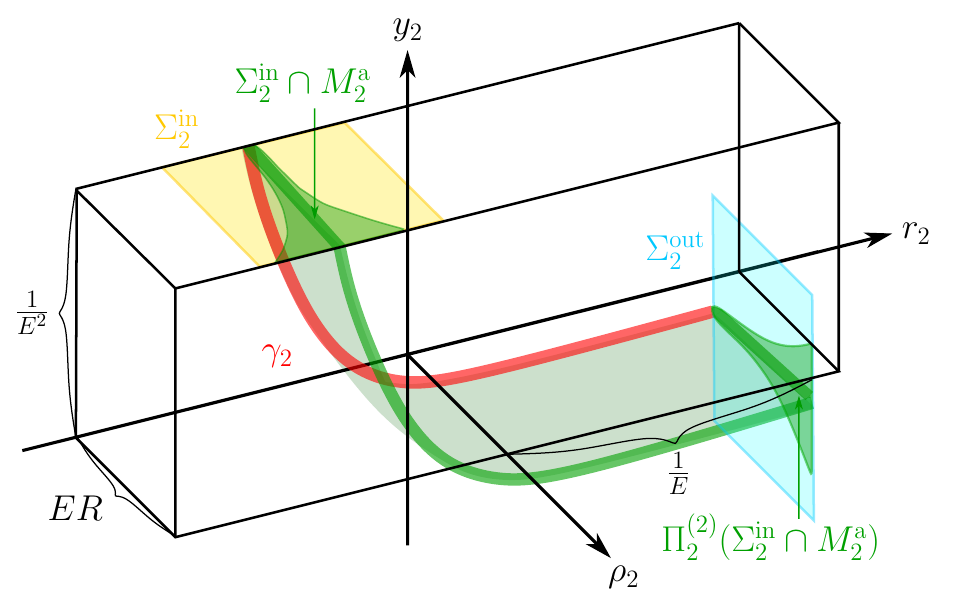}
		\caption{Geometry and dynamics projected into $(r_2,y_2,\rho_2)$-space \SV{for $\alpha=2$}; c.f.~Figures \ref{fig:chart2_2d} and \ref{fig:chart2_3d}. The extension of the (projected) three-dimensional manifold $M_2^\textup{a}$ is sketched in shaded green. The image of the wedge-shaped region $\kappa_{12}(\SV{\Pi_1^{(2)}}(\Sigma_1^{\textup{in}})) \subset \Sigma_2^{\textup{in}}$ (recall Proposition \ref{prop:pi1}) under $\SV{\Pi_2^{(2)}}$ is also shown in shaded green.}
		\label{fig:chart2_wo_theta}
	\end{figure}
	
	\ 
	
	Finally we note that \SJJ{if $\alpha = 2$, then} the expression for \SJJ{$h^{(2)}_{y_2}(E^{-1}, \theta_2)$} (and therefore the map \SJJ{$\Pi^{(2)}_2$}) can be simplified using the following result, which can be found in \cite{Krupa2001a}.
	
	\begin{lemma} \label{lem:hy20}
		\SJJ{For $\alpha = 2$, t}he Riccati function \SJJ{$h^{(2)}_{y_2}$} defined in Proposition \ref{prop:riccati} satisfies
		\begin{equation*}
			\SJJ{h^{(2)}_{y_2}\left(E^{-1}, \theta_2 \right) = - \SV{\left( \frac{c^2}{a b} \right)^{1/3}} \Omega_0 + E,}
		\end{equation*}
		\SV{where $a = a(\theta_2)$, $b = b(\theta_2)$ and $c = c(\theta_2)$.}
	\end{lemma}
	
	\begin{proof}
		This can be proven using the asymptotic properties of the decoupled Riccati equation \eqref{eq:riccatti_planar} described in Proposition \ref{prop:riccati}. Since the proof does not depend on the angular dynamics \SJJ{($\theta_2 = const.$ when $\rho_2 = 0$ and $\alpha = 2$)}, we refer to \cite[Rem.~2.10 and Prop.~2.11]{Krupa2001a}.
	\end{proof}

	\subsection{Dynamics in the Exit Chart $K_3$} 
	\label{sec:53}
	
	In chart $K_3$ we study solutions close to the extension of the manifold $M_3^{\textup{a}} = \kappa_{23}(M_2^{\textup{a}})$ as it leaves a neighbourhood of the singular cycle $S_0^{\textup{c}}$.
	
	\begin{lemma} \label{lem:blowupK3}
		Following the positive transformation of time $\rho_3 \dd t = \dd t_3$, the desingularized equations in chart $K_3$ are given by
		\begin{equation} \label{eq:chart3desing}
			\begin{aligned}
				\rho_3' &=\rho_3F(\rho_3,\theta_3,y_3,\varepsilon_3), \\
				\theta_3' &= \rho_3^{\alpha-1}\varepsilon_3^\alpha, \\
				y_3' &= -2y_3F(\rho_3,\theta_3,y_3,\varepsilon_3)+\varepsilon_3^3(-\SJJ{c(\theta)}+ \mathcal{O}(\rho_3)), \\
				\varepsilon_3' &= -\varepsilon_3F(\rho_3,\theta_3,y_3,\varepsilon_3),
			\end{aligned}
		\end{equation}
		where \SJJ{$F(\rho_3,\theta_3,y_3,\varepsilon_3) = b(\theta) - a(\theta) y_3+ \mathcal{O}(\rho_3)$}.
	\end{lemma}
	
	\begin{proof}
		This follows after direct differentiation of the local coordinate expressions in \eqref{eq:chart_coordinates} and subsequent application of the desingularization $\rho_3 \dd t = \dd t_3$.
	\end{proof}

	Except for the angular coordinate $\theta_3$, the analysis in $K_3$ is entirely local. Specifically, we focus on dynamics within the set
	\begin{equation*}
		\mathcal{D}_3 \coloneqq \left\{(\rho_3,\theta_3,y_3,\varepsilon_3) : \rho_3 \in [0,R], \theta_3 \in \R / \mathbb Z, y_3 \in [-\nu, \nu], \eps_3 \in [0,E]  \right\},
	\end{equation*}
	where $E$, $R$ and $\nu$ are the same positive constants used to define the entry and exit sections in charts $K_1$ and $K_2$. System \eqref{eq:chart3desing} has a circular critical manifold
	\begin{equation*} 
		Q \coloneqq \{(0,\theta_3,0,0) : \theta_3 \in \R / \mathbb Z \} ,
	\end{equation*}
	with the following properties.

	\begin{lemma} \label{lem:linearizec}
		Consider system \eqref{eq:chart3desing}. $Q$ is normally hyperbolic and saddle-type. More precisely, the linearization along $Q$ has eigenvalues\SJJ{
		\[
		\lambda_1 = b(\theta), \qquad \lambda_2 = 0, \qquad \lambda_3 = -2 b(\theta), \qquad \lambda_4 = -b(\theta),
		\]
		}and corresponding eigenvectors $(1,0,0,0)$, $(0,1,0,0)$, $(0,0,1,0)$ and \SV{$(0,-\alpha \rho_3^{\alpha -1} \eps_3^{\alpha-1}/b(\theta), 0, 1)$}, respectively.
	\end{lemma}
	
	\begin{proof}
		Direct calculation.
	\end{proof}
	
	Lemma \ref{lem:linearizec} implies the existence of local center, stable and unstable manifolds at $Q$. Specifically, there is a one-dimensional center manifold $W_{3}^{\textup{c}}(Q)$, a two-dimensional stable manifold $W_{3}^{\textup{s}}(Q)$, and a one-dimensional unstable manifold $W_{3}^{\textup{u}}(Q)$. The geometry is sketched in different three-dimensional subspaces and projections in Figures \ref{fig:chart3_wo_theta}, \ref{fig:chart3_2d} and \ref{fig:chart3_3d}.

	\
	
	\SV{If $\alpha =2$, we can} show that the extension of the surface $\gamma_2$ connects to the critical manifold $Q$ tangentially to the cylinder spanned by the center and weakly stable eigenvectors.
	
	\begin{lemma} \label{lem:gamma2chart3}
		\SV{Fix $\alpha =2$}. Then the extension of $\gamma_2$ under system \eqref{eq:chart3desing}, i.e.~$\gamma_3 = \kappa_{23}(\gamma_2)$, connects to $Q$ tangentially to the cylinder segment $\{ (0,\theta_3, 0, \eps_3) \in \mathcal D_3 \}$.
	\end{lemma}
	
	\begin{proof}
		Applying the change of coordinates formula in \eqref{eq:kappa_formulae} together with the Riccati asymptotics as $r_2 \to \infty$ in Proposition \ref{prop:riccati} Assertion (b), we obtain\SJJ{
		\begin{equation*}
			\kappa_{23}(\gamma_2) \cap \mathcal D_3 = 	\left\{ \left(0, \theta_3, \eps_3^2 \left( - \left(\frac{c^2}{a b} \right)^{1/3} \Omega_0 + \frac{c}{b} \eps_3 + \mathcal O(\eps_3^3) \right) , \eps_3 \right) : \theta_3 \in \R / \mathbb Z, \eps_3 \in [0,E] \right\} ,
		\end{equation*}
		where $a = a(\theta_3)$, $b = b(\theta_3)$ and $c = c(\theta_3)$ are constant ($\theta_3' = 0$ in $\{\rho_3 = 0\}$),} which converges to the critical manifold $Q$ tangentially to $\{ (0,\theta_3, 0, \eps_3) \in \mathcal D_3 \}$ as $\eps_3 \to 0$.
	\end{proof}
	
	\begin{figure}[t!]
		\centering
		\includegraphics[width=0.8\textwidth]{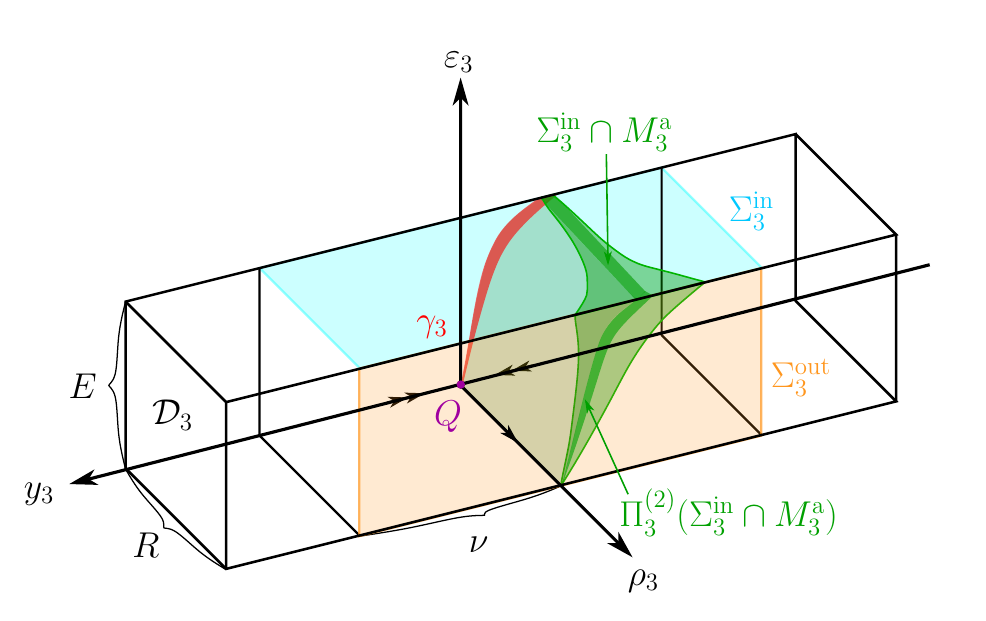}
		\caption{Geometry and dynamics within $\mathcal D_3$, shown in $(r_3,\rho_3,\varepsilon_3)$-space \SV{for $\alpha=2$}. The (projection of the) two-dimensional surface $\gamma_3 = \kappa_{23}(\gamma_2)$ shown in red connects to the circular saddle-type critical manifold $Q$, which is indicated by the purple dot at the origin. Entry and exit sections $\Sigma_3^\textup{in}$ and $\Sigma_3^{\textup{out}}$ are shown in shaded cyan and orange, respectively. The wedge-shaped region within $\Sigma_2^\textup{out}$ (shown in shaded green in Figure \ref{fig:chart2_wo_theta}) is shown here in $\Sigma_3^\textup{in}$ (again in shaded green), along with its image in $\Sigma_3^\textup{out}$ under $\SV{\Pi_3^{(2)}}$. The projection of the three-dimensional manifold $M_3^\textup{a}$ with base along $\gamma_3 \cup \{y_3 = \eps_3 = 0, \rho_3 \geq 0\}$ is also shown.}
		\label{fig:chart3_wo_theta}
	\end{figure}

	Having described the geometry \SV{for $\alpha = 2$}, we turn our attention to the characteristics of the map $\SV{\Pi_3^{(\alpha)}} : \Sigma_3^{\textup{in}} \to \Sigma_3^{\textup{out}}$ induced by the flow of initial conditions in
	\begin{equation*}
		\Sigma_3^\textup{in} \coloneqq \left\{ (\rho_3,\theta_3,y_3,E) \in \mathcal{D}_3 \right\} ,
	\end{equation*}
	which corresponds to the exit section $\Sigma_2^{\textup{out}}$ in chart $K_2$ via $\Sigma_3^{\textup{in}} = \kappa_{23}(\Sigma_2^{\textup{out}})$, up to the exit section
	\begin{equation*}
		\Sigma_3^\textup{out} \coloneqq \left\{ (R,\theta_3,y_3,\varepsilon_3)\in \mathcal{D}_3 \right\} ,
	\end{equation*}
	which is precisely the representation of the exit section $\Delta_\varepsilon^\textup{out}$ defined in \eqref{eq:Delta_out_eps} after blow-up in chart $K_3$ if we set $y_0 = R^2 \nu$. Similarly to the $K_3$ analysis for the regular fold point in \cite{Krupa2001a}, $\SV{\Pi_3^{(\alpha)}}$ can be analysed directly after a second positive transformation of time $F(\rho_3,\theta_3,y_3,\varepsilon_3) \dd t_3 = \dd \tilde{t}_3$, which amounts to division of the right-hand side in system \eqref{eq:chart3desing} by $F(\rho_3,\theta_3,y_3,\varepsilon_3)$ (which is strictly positive in $\mathcal D_3$). This leads to the orbitally equivalent system\SJJ{
	\begin{equation} \label{eq:chart3desing2}
		\begin{aligned}
			\rho_3' &= \rho_3, \\
			\theta_3' &= \frac{\rho_3^{\alpha-1}\varepsilon_3^\alpha}{b(\theta_3) - a(\theta_3) y_3 + \mathcal{O}(\rho_3)} = \frac{\rho_3^{\alpha-1}\varepsilon_3^\alpha}{b(\theta_3)} + \mathcal{O}(\rho_3^{\alpha-1}\varepsilon_3^\alpha y_3,\rho_3^{\alpha}\varepsilon_3^\alpha),  \\
			y_3' &= -2 y_3+\frac{\varepsilon_3^3(-c(\theta_3)+\mathcal{O}(\rho_3))}{b(\theta_3) - a(\theta_3) y_3 + \mathcal{O}(\rho_3)} =-2 y_3- \varepsilon_3^3 \frac{c(\theta_3)}{b(\theta_3)} + \mathcal{O}(\varepsilon_3^3 y_3,\rho_3\varepsilon_3^3) , \\
			\varepsilon_3' &= -\varepsilon_3,
		\end{aligned} 
	\end{equation}
	}where by a slight abuse of notation the prime notation now refers to differentiation with respect to $\tilde t_3$.
	
	\
	
	System \eqref{eq:chart3desing2} is used to obtain the following result.
	
	\begin{figure}[t!]
		\centering
		\begin{minipage}[c]{0.47\textwidth}
			\centering
			\vspace{1pt} 
			\includegraphics[width=\textwidth]{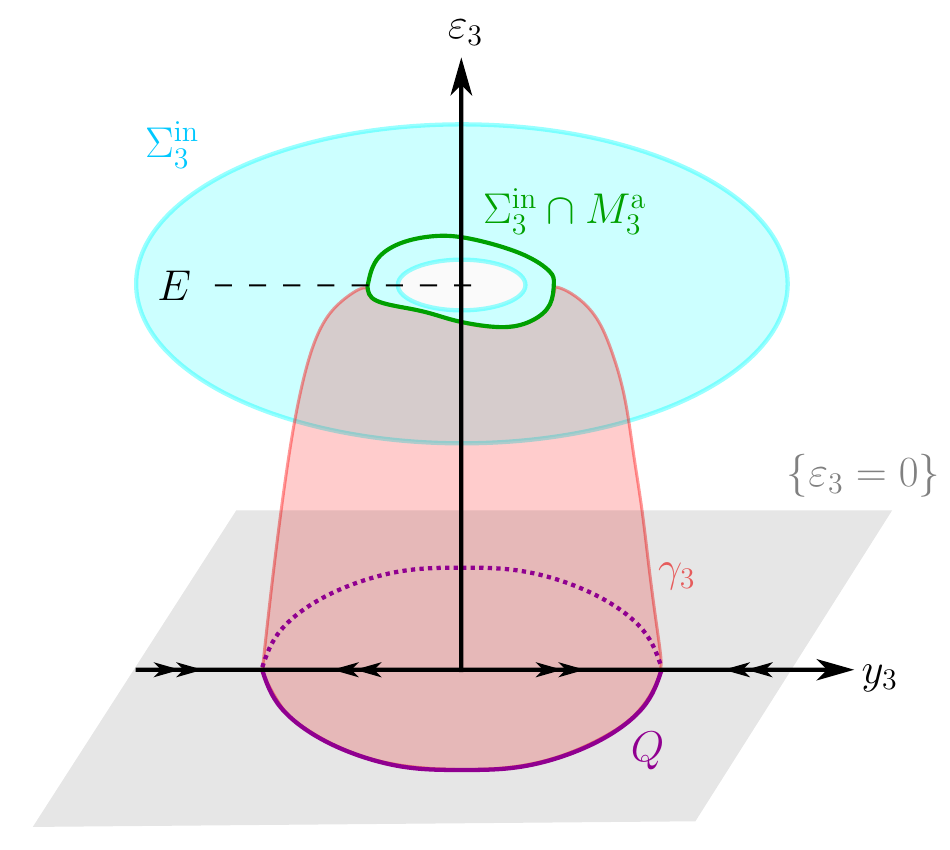}
			\caption{Geometry and dynamics within the invariant hyperplane $\{\rho_3 = 0\}$ \SJJ{with $\alpha = 2$},
			projected into $(y_3,\theta_3,\eps_3)$-space; c.f.~Figure \ref{fig:chart3_wo_theta}.}
			\label{fig:chart3_2d}
		\end{minipage}
		\begin{minipage}[c]{0.04\textwidth}
			\textcolor{white}{.}
		\end{minipage}
		\begin{minipage}[c]{0.47\textwidth}
			\centering
			\vspace{10pt}
			\includegraphics[width=\textwidth]{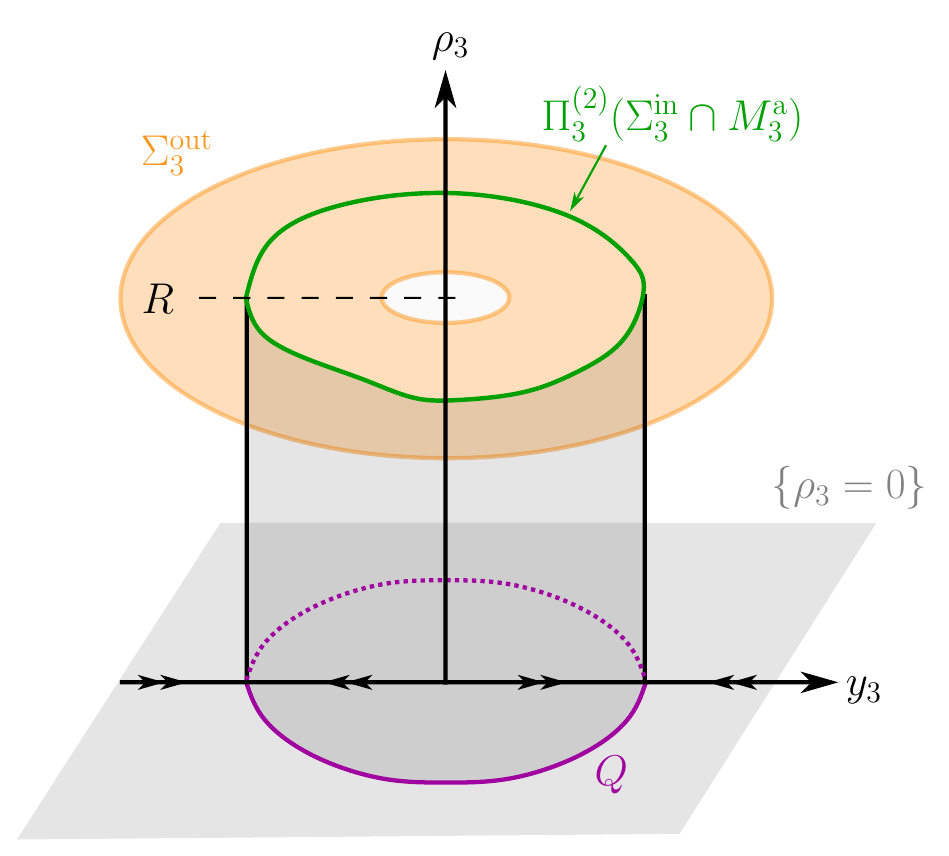}
			\caption{Geometry and dynamics within the invariant hyperplane $\{\eps_3 = 0\}$ \SV{with $\alpha = 2$},
			projected into $(y_3,\theta_3,\rho_3)$-space; c.f.~Figures \ref{fig:chart3_wo_theta} and~\ref{fig:chart3_2d}.}
			\label{fig:chart3_3d}
		\end{minipage}
	\end{figure}

	\begin{proposition}  \label{prop:pi3}
		Fix \SV{$\nu >0$ and} $E, R> 0$ sufficiently small. Then the map $\SV{\Pi_3^{(\alpha)}}: \Sigma_3^\textup{in} \to \Sigma_3^\textup{out}$ is well-defined 
		and given by\SJJ{
		\begin{equation*}
			\Pi_3^{(\alpha)}(\rho_3,\theta_3, y_3, E) = \left(R,h_{\theta_3}^{(2)}(\rho_3,\theta_3,y_3),h_{y_3}^{(\alpha)}(\rho_3,\theta_3,y_3),\frac{E}{R} \rho_3 \right) .
		\end{equation*}
		We have that $h_{\theta_3}^{(\alpha)}(\rho_3,\theta_3,y_3) = \tilde h_{\theta_3}^{(\alpha)}(\rho_3,\theta_3,y_3) \mod{1}$, where
		\[
		\tilde h_{\theta_3}^{(\alpha)}(\rho_3,\theta_3,y_3) =
		\begin{cases}
		    \theta_3 + \mathcal O \left(\ln \left( \rho_3^{-1} \right) \right) , & \alpha = 1, \\
		    \theta_3 + \mathcal{O}\left(\SV{\rho_3} \ln\left(\rho_3^{-1}\right)\right) , & \alpha = 2.
		\end{cases}
		\]
		Moreover, there exist constants $\sigma_- < \sigma_+$ such that in case $\alpha = 1$ we have
		\begin{equation}
		    \label{eq:y3_bounds}
    		\SVV{\left(-\nu+\sigma_- E^3\right) \left(1 + \mathcal{O}(\rho_3)\right)\left( \frac{\rho_3}{R} \right)^2  \leq
		h_{y_3}^{(1)}(\rho_3,\theta_3,y_3) \leq 
	    \left(\nu+\sigma_+ E^3\right)\left(1 + \mathcal{O}(\rho_3)\right) \left( \frac{\rho_3}{R} \right)^2. }
		\end{equation}
		In case $\alpha = 2$ we have}
		\begin{equation*}
			\SV{h_{y_3}^{(2)}} (\rho_3,\theta_3,y_3) = \left(y_3-E^3\right) \left(\frac{\rho_3}{R} \right)^2 + \mathcal{O} \left( \rho_3^3 \ln\left( \rho_3^{-1} \right)\right).
		\end{equation*}
	\end{proposition}

	\begin{proof}
		We start with the case \SV{$\alpha =2 $}, and 
		consider solutions of system \eqref{eq:chart3desing2} which satisfy $(\rho_3,\theta_3,y_3,\eps_3)(0) = (\rho_3^\ast,\theta_3^\ast,y_3^\ast,E) \in \Sigma_3^\textup{in}$ and $(R, \theta_3, y_3,\varepsilon_3)(T_3) \in \Sigma_3^{\textup{out}}$ for some $T_3 > 0$. The equations for $\rho_3$ and $\eps_3$ can be solved directly, leading to \SJJ{$\rho_3(\tilde t_3) = \rho_3^\ast \me^{\tilde t_3}$} and \SJJ{$\eps_3(\tilde t_3) = E \me^{- \tilde t_3}$}, and the boundary constraint $\rho_3(T_3) = R$ leads to the following expression for the transition time:
		\begin{equation}
		\label{eq:T3}
			T_3 = \ln\left( \frac{R}{\rho_3^\ast}\right) .
		\end{equation}
		It follows that $\eps_3(T_3) = E \rho_3^\ast / R$, as required. We need to estimate $y_3(T_3)$ and $\theta_3(T_3)$ at \SJJ{$\tilde t_3 = T_3$}. \SJJ{
		Since \SV{$\alpha =2$}, we have that
		\[
		\SV{\theta_3' = \rho_3^\ast \left(E^2 \me^{- \tilde t_3} \left( b(\theta_3) - a(\theta_3) y_3 + \mathcal{O}(\rho_3) \right)\right)^{-1} .}
		\]
		One can show with direct estimates that the term in parentheses is uniformly bounded by a positive constant, which implies that
		\[
		\SV{\rho_3^\ast \zeta_-} \tilde t_3 \leq \tilde \theta(\tilde t_3) - \theta_3^\ast \leq \SV{\rho_3^\ast\zeta_+} \tilde t_3 ,
		\]
		for all $\tilde t_3 \in [0,T_3]$, where $\SV{\zeta_+} \geq \SV{\zeta_-} > 0$ and $\tilde \theta_3$ is defined via $\theta_3(\tilde t_3) = \tilde \theta_3(\tilde t_3) \mod 1$. Using $\rho_3(\tilde t_3) = \rho_3^\ast \me^{\tilde t_3}$, it follows that $\tilde \theta_3$ can be written in terms of $\rho_3$, specifically, we obtain
		\begin{equation}
			\label{eq:theta_3}
			\SV{\tilde \theta_3(\rho_3) = \theta_3^\ast + \mathcal O\left( \rho_3^\ast \ln \left( \frac{\rho_3}{\rho_3^\ast} \right) \right) ,}
		\end{equation}
		where we permit a slight abuse of notation in switching the argument from $\tilde t_3$ to $\rho_3$. Substituting $\rho_3(\tilde t_3) = \rho_3^\ast \me^{\tilde t_3}$ and evaluating the resulting expression at $\tilde t_3 = T_3$ yields the desired estimate for \mbox{$\theta_3(T_3) = \tilde h_{\theta_3}^{(2)}(\rho_3,\theta_3,y_3) \mod 1$}.
		
		It remains to estimate $y_3(T_3)$. It will be helpful to use \SV{\eqref{eq:theta_3}} in order to write\SV{
		\begin{equation}
			\label{eq:abc}
			a(\theta_3) = a_3 + \mathcal O\left( {\rho_3^\ast} \ln \left( \frac{\rho_3}{\rho_3^\ast} \right) \right), 
			\ \ 
			b(\theta_3) = b_3 + \mathcal O\left( {\rho_3^\ast} \ln \left( \frac{\rho_3}{\rho_3^\ast} \right) \right), 
			\ \ 
			c(\theta_3) = c_3 + \mathcal O\left( {\rho_3^\ast} \ln \left( \frac{\rho_3}{\rho_3^\ast} \right) \right), 
		\end{equation}
		}where $a_3 = a(\theta_3^\ast)$, $b_3 = b(\theta_3^\ast)$ and $c_3 = c(\theta_3^\ast)$. From here, we use and adaptation of the `partial linearisation' method used in the proof of \cite[Prop.~2.11]{Krupa2001a}. Define a new variable $\tilde \eps_3 = \eps_3^3$, and consider the resulting equations within the invariant hyperplane $\{\rho_3 = 0\}$:
		\begin{equation}
			\label{eq:linearized_eqns_K3}
			\begin{split}
				\theta_3' &= 0, \\
				y_3' &= -2 y_3 - \frac{c_3}{\SV{b_3}} \tilde \eps_3 + \mathcal O(y_3 \tilde \eps_3) , \\
				\tilde \eps_3' &= - 3 \tilde \eps_3 ,
			\end{split}
		\end{equation}
		which may be considered as a \SVVV{$\theta_3^\ast$}-family of planar systems in the $(y_3, \tilde \eps_3)$ variables, with 
		constants $a_3$, $b_3$ and $c_3$ \SVVV{depending on the parameter $\theta_3^\ast$}. Since the Jacobian at the equilibrium $(0,0)$ is hyperbolic and non-resonant with the eigenvalues $-2$ and $-3$, there exists a parameter-dependent, near-identity transformation of the form
		\begin{equation}
			\label{eq:y_trans}
			\tilde y_3 = \psi(y_3, \tilde \eps_3\SVVV{, \theta_3^\ast}) = y_3 + \mathcal O(y_3 \tilde \eps_3) , \qquad
			y_3 = \tilde \psi(\SV{\tilde y_3}, \tilde \eps_3\SVVV{, \theta_3^\ast}) = \tilde y_3 + \mathcal O(\tilde y_3 \tilde \eps_3) ,
		\end{equation}
		such that the transformed system has been linearized, i.e.
		\begin{equation}
			\label{eq:linearized_eqns_K3_2}
			\begin{split}
				\tilde y_3' &= -2 \tilde y_3 - \frac{c_3}{b_3} \tilde \eps_3 , \\
				\tilde \eps_3' &= - 3 \tilde \eps_3 ,
			\end{split}
		\end{equation}
		\SJJJ{see e.g.~\cite[Thm.~1]{Ilyashenko1991}.} Applying the transformation \eqref{eq:y_trans} to system \eqref{eq:chart3desing2} and using \eqref{eq:abc} leads to \SV{
		\begin{equation}
			\label{eq:chart3desing3}
			\begin{aligned}
				\rho_3' &= \rho_3, \\
				\theta_3' &= \frac{\rho_3\varepsilon_3^2}{b(\theta_3)} + \mathcal{O}(\rho_3\varepsilon_3^2 \tilde y_3,\rho_3^2\varepsilon_3^2),  \\
				\tilde y_3' &= -2 \tilde y_3 - \SV{\frac{c_3}{b_3}} \eps_3^3 + \eps_3^3 \rho_3 H(\rho_3, \theta_3, \tilde y_3, \eps_3) , \\
				\varepsilon_3' &= - \varepsilon_3 ,
			\end{aligned} 
		\end{equation}
		}where the remainder function $H(\rho_3, \theta_3, \tilde y_3, \SV{\eps_3})$ is smooth and uniformly bounded} \SV{over $\tilde t_3 \in [0,T_3]$.} Substituting the solutions for $\rho_3(\tilde t_3)$ and $\eps_3(\tilde t_3)$ into the equations for $\theta_3$ and \SJJ{$\tilde y_3$} in system \SV{\eqref{eq:chart3desing3}} leads to the planar non-autonomous system\SV{
		\begin{align*} 
			\theta_3' &=  \frac{{\rho_3^\ast}E^2}{b(\theta_3)} \me^{-\tilde{t}_3} + {\rho_3^\ast} \mathcal{O}( \tilde y_3 \me^{-\tilde{t}_3}, \rho_3^\ast), \\
			\tilde y_3' &= -2 \tilde y_3 - \frac{c_3}{b_3} E^3 \me^{-3\tilde t_3} + \rho_3^\ast E^3 \me^{-2 \tilde t_3} H( \rho_3^\ast \me^{\tilde{t}_3}, \theta_3, \tilde y_3, E \me^{-\tilde{t}_3}).
		\end{align*}
		Now} introduce a new time-dependent variable $z_3$ via\SJJ{
		\begin{equation*} 
			\tilde y_3(\tilde{t}_3) = \left( y_3^\ast - E^3 + z_3(\tilde{t}_3) \right) \me^{-2\tilde{t}_3} + \frac{c_3}{b_3} E^3 \me^{-3\tilde{t}_3}.
		\end{equation*}
		}Differentiating \SJJ{$\tilde y_3(\tilde{t}_3)$} and using $\rho_3(\tilde{t}_3) = \rho_3^\ast \me^{\tilde{t}_3}$ leads to the following ODE in $z_3$:\SJJ{
		\begin{equation*} 
			z_3' = \rho_3^\ast E^3 \tilde H \left(\rho_3^\ast, \theta_3, y_3^\ast, z_3, \tilde t_3\right) ,
		\end{equation*}
		with $\tilde H (\rho_3^\ast, \theta_3, y_3^\ast, z_3, \tilde t_3 ) \coloneqq H(\rho_3^\ast \SV{\me^{\tilde{t}_3}}, \theta_3, ( y_3^\ast - E^3 + z_3(\tilde{t}_3) ) \me^{-2\tilde{t}_3} + (c_3 / b_3) E^3 \me^{-3\tilde{t}_3}, E \me^{-\tilde{t}_3})$ and} initial condition $z_3(0) =  0$. \SJJ{One can show that $\tilde H \left(\rho_3^\ast, \theta_3, y_3^\ast, z_3, \tilde t_3\right)$ is uniformly bounded over $\tilde t_3 \in [0,T_3]$, so that integration gives} 
		\begin{equation*}
			z_3(\tilde{t}_3) = \mathcal{O}(\rho_3^\ast \tilde{t}_3) \quad \implies \quad z_3({T}_3)
			= \mathcal{O} \left( \rho_3^\ast \ln\left( \frac{1}{\rho_3^\ast}\right)\right) .
		\end{equation*}
		Changing variables back to \SJJ{$\tilde y_3$} and evaluating at $\tilde t_3 = T_3$ yields
		\[
		\tilde y_3({T}_3) = (y_3^\ast - E^3 + z_3({T}_3))\me^{-2{T}_3} + \SJJ{\frac{c_3}{b_3}} E^3 \me^{-3{T}_3} \nonumber\\
		= (y_3^\ast-E^3) \left(\frac{\rho_3^\ast}{R} \right)^2 + \mathcal{O} \left( (\rho_3^\ast)^3  \ln\left( \frac{1}{\rho_3^\ast}\right)\right) ,
		\]
		\SJJ{and therefore
		\[
		y_3(T_3) = (y_3^\ast-E^3) \left(\frac{\rho_3^\ast}{R} \right)^2 + \mathcal{O} \left( (\rho_3^\ast)^3  \ln\left( \frac{1}{\rho_3^\ast}\right)\right)
		\]
		}as required.
		
		\
		
		\SJJ{Now fix $\alpha = 1$. The fact that $\tilde \theta_3(T_3) = \mathcal O(\ln {\rho_3^\ast}^{-1})$ follows from direct estimates using the boundedness of the right-hand side of the equation for $\theta_3'$ and the expression for the transition time $T_3$ in \eqref{eq:T3}. It remains to bound $y_3(T_3)$. Note that we cannot apply the same partial linearisation approach as we did for \SV{$\alpha =2$}, since $\theta_3$ is dynamic (i.e.~not constant) in $\{\rho_3 = 0\}$. We can, however, obtain a coarser estimate via a more direct approach. Notice that for sufficiently small $E$ we have $y_3'\vert_{y_3 = -\nu} > 0$ and $y_3'\vert_{y_3 = \nu} < 0$. It follows that $\Pi_3^{(1)}$ is well-defined, since solutions are inflowing along the faces of $\mathcal D_3$ defined by $\{y = \pm \nu\}$ for all $\tilde t_3 \in [0,T_3]$. Moreover, there exist constants \SV{$\sigma_- < \sigma_+$ such that
		\[
		- 2 y_3 + \sigma_- \eps_3^3 
		\leq y_3'
		\leq - 2 y_3 + \sigma_+ \eps_3^3 ,
		\]
		}for all $y \in [-\nu,\nu]$ and $\eps_3 \in [0,E]$. Substituting $\eps_3(\tilde t_3) = E \me^{- \tilde t_3}$ and solving the first order linear equations\SV{
		\[
		{y_3^\pm}' + 2 y_3^\pm = \sigma_\pm E^3 \me^{-3 \tilde t_3} ,
		\]
		}we obtain lower and upper solutions on $\tilde t_3 \in [0,T_3]$ with the property that
    	\[
			\SV{y_3^-(\tilde t_3) = \left(-\nu+\sigma_- E^3\right) \me^{-2 \tilde t_3} - \sigma_- E^3\me^{-3 \tilde t_3}  \leq 
		    y_3(\tilde t_3) \leq 
		    y_3^+(\tilde t_3) = \left(\nu+\sigma_+ E^3\right) \me^{-2 \tilde t_3} - \sigma_+ E^3\me^{-3 \tilde t_3} .}
		\]
		Evaluating these expressions at $\tilde t_3 = T_3 = \ln(R/\rho_3^\ast)$ yields
		\[
			\SV{\left(-\nu+\sigma_- E^3\right) \left(1 + \mathcal{O}(\rho_3^\ast)\right)\left( \frac{\rho_3^\ast}{R} \right)^2  \leq
		y_3(T_3) \leq 
	    \left(\nu+\sigma_+ E^3\right)\left(1 + \mathcal{O}(\rho_3^\ast)\right) \left( \frac{\rho_3^\ast}{R} \right)^2 .  }\qedhere
		\]
		}
	\end{proof}
	
	Note that systems \eqref{eq:chart3desing} and \eqref{eq:chart3desing2} are orbitally equivalent. Since the result in Proposition \ref{prop:pi3} only depends on the initial condition in $\Sigma_3^\textup{in}$, and not on the transition time, it holds for both systems \eqref{eq:chart3desing} and \eqref{eq:chart3desing2}.
	
	\begin{remark}
		Similarly to the arguments used to describe the map $\SV{\Pi_1^{(\alpha)}}$ in proof of Proposition \ref{prop:pi1}, the arguments used to describe the map $\SV{\Pi_3^{(\alpha)}}$ in Proposition \ref{prop:pi3} are (strictly speaking) only valid for initial conditions with $\rho_3 \in (0,R]$. This is not problematic since we aim to derive results for $\eps > 0$. Moreover, $\SV{\Pi_3^{(\alpha)}}$ can be smoothly and uniquely extended to $\rho_3 = 0$\SJJ{.}
	\end{remark}
	
	\subsection{Proof of Theorem \ref{thm:main}}
	\label{sec:54}
	
	We now combine the results obtained in the charts $K_1$, $K_2$ and $K_3$ in order to prove Theorem \ref{thm:main}. The idea is to describe the extended map $\SV{\pi_\varepsilon^{(\alpha)}}: \Delta_\varepsilon^{\textup{in}} \rightarrow \Delta_\varepsilon^{\textup{out}}$ defined in Section \ref{sub:blow-up_and_local_coordinate_charts}, the first three components of which coincide with components of $\SV{\pi^{(\alpha)}}: \Delta^{\textup{in}} \rightarrow \Delta^{\textup{out}}$, via its preimage in the blown-up space:
	\[
	\SV{\Pi^{(\alpha)} \coloneqq \Pi_3^{(\alpha)} \circ \kappa_{23} \circ \Pi_2^{(\alpha)} \circ \kappa_{12} \circ \Pi_1^{(\alpha)} : \Sigma_1^{\textup{in}} \to \Sigma_3^{\textup{out}} .}
	\]
	This can be done explicitly using the change of coordinate maps $\kappa_{ij}$ in Lemma \ref{lem:change} and the characterisation of \SV{$\Pi_1^{(\alpha)}$, $\Pi_2^{(\alpha)}$ and $\Pi_3^{(\alpha)}$} in Propositions \ref{prop:pi1}, \ref{prop:pi2} and \ref{prop:pi3}, respectively. 
	The geometry is sketched in Figure~\ref{fig:maintheorem}.
	
	\
	
	\SJJ{In order to prove Assertions (a)-(c), we need to} derive the form of the map \SV{$\Pi^{(\alpha)}$}. Initial conditions $(r_1, \theta_1, R, \eps_1) \in \Sigma_1^{\textup{in}}$ are mapped to $\Sigma_1^{\textup{out}}$ under \SV{$\Pi_1^{(\alpha)}$} as described by Proposition \ref{prop:pi1}. Since \mbox{$\Sigma_2^{\textup{in}} = \kappa_{12}(\Sigma_1^{\textup{out}})$}, we obtain the following input for the map \SV{$\Pi_2^{(\alpha)}$:
	\begin{align}
		\label{eq:K2_ICs}
			(r_2,\theta_2,y_2,\rho_2) &= \kappa_{12} \left( \Pi_1^{(\alpha)}(r_1, \theta_1, R, \eps_1) \right) \nonumber\\
			&= \left( \frac{1}{E} \left( h_{r_1}^{(\alpha)}\left(h_{\theta_1}^{(\alpha)}(r_1,\theta_1,\eps_1),\frac{R}{E}\eps_1,E\right) + \mathcal O\left(\me^{-\tilde \varrho/ \eps_1^3}\right) \right) , h_{\theta_1}^{(\alpha)}(r_1,\theta_1,\eps_1), \frac{1}{E^2}, R \eps_1 \right) ,
	\end{align}
	where \SJJ{we recall that} $h_{\theta_1}^{(\alpha)}(r_1,\theta_1,\eps_1) = \tilde	h_{\theta_1}^{(\alpha)}(r_1,\theta_1,\eps_1) \mod{1}$,
	\begin{equation}
	\label{eq:tilde_h_theta_1}
    	\tilde h_{\theta_1}^{(\alpha)}(r_1,\theta_1,\eps_1) = \theta_1 + \frac{(R \eps_1)^{\alpha - 1}}{\SV{c_0}} \left( \frac{1}{{\eps_1}^2} + \mathcal{O}(1)\right) ,
	\end{equation}
	\SJJ{and} $c_0 \coloneqq \int_0^1 c(\theta_1(t_1)) \, \dd t_1$. 
	Using 
	\eqref{eq:K2_ICs}, Lemma \ref{lem:change} and Proposition \ref{prop:pi2}, we obtain an expression for the input to the map $\Pi_3^{(\alpha)}$, namely
	\begin{align}
		\label{eq:K3_ICs}
			(\rho_3,\theta_3,y_3,\eps_3) &= \kappa_{23} \left(\Pi_2^{(\alpha)}(r_2,\theta_2,y_2,\rho_2)\right) \nonumber\\
			&= \left( \frac{R}{E} \eps_1 , h_{\theta_2}^{(\alpha)}\left(r_2,\theta_2,\rho_2\right) ,
			 E^2 h_{y_2}^{(\alpha)}(E^{-1},\theta_2) + \mathcal{O}(\eps_1), E \right),
	\end{align}
	where we used the fact that $r_2 - r_{2,c} =
	\mathcal O ( \me^{- \tilde \varrho / \eps_1^3} )$ in \SJJ{the derivation of} the third component in the right-hand side, and
	$h_{\theta_2}^{(\alpha)}(r_2, \theta_2, \rho_2) = \tilde h_{\theta_2}^{(\alpha)}(r_2, \theta_2, \rho_2) \mod 1$ with
	\begin{equation}
	\label{eq:tilde_h_theta_2}
		\tilde h_{\theta_2}^{(\alpha)}(r_2, \theta_2, \rho_2) = 
		\begin{cases}
		   \tilde h_{\theta_1}^{(1)}(r_1,\theta_1,\eps_1) +  \mathcal O(1), & \alpha = 1, \\
		    \tilde h_{\theta_1}^{(2)}(r_1,\theta_1,\eps_1) + \frac{R \eps_1}{c(\theta_2)} \left(E^{-2} + \left(\frac{c^2}{ab}\right)^{1/3} \Omega_0 - E\right) + \mathcal{O}\left(\eps_1^2\right) , & \alpha = 2 ,
		\end{cases}
	\end{equation}
	\SJJ{where $c(\theta_2) = c\left(h_{\theta_1}^{(2)}(r_1,\theta_1,\eps_1)\right)$. With regards to the third component in \eqref{eq:K3_ICs}, if $\alpha = 1$ then} we have
	\[
	-\nu \leq E^2 h_{y_2}^{(1)}(E^{-1},\theta_2) \leq \nu ,
	\]
    \SJJ{where $\nu>0$ is the constant appearing in the definition of $\Sigma_2^{\textup{out}}$ (recall Proposition \ref{prop:ricatti_alpha_1})}. If $\alpha = 2$, then by Lemma \ref{lem:hy20} it holds
	\[
	h_{y_2}^{(2)}(E^{-1},\theta_2) = \SJJ{-\left(\frac{c(\theta_2)^2}{a(\theta_2) b(\theta_2)}\right)^{1/3} }\Omega_0 + E, 
	\]
	where $\SJJ{a(\theta_2)} = a\left(h_{\theta_1}^{(\SJJJ{2})}(r_1,\theta_1,\eps_1)\right)$, $\SJJ{b(\theta_2)} = b\left(h_{\theta_1}^{(\SJJJ{2})}(r_1,\theta_1,\eps_1)\right)$ and $\SJJ{c(\theta_2)} = c\left(h_{\theta_1}^{(\SJJJ{2})}(r_1,\theta_1,\eps_1)\right)$.
	
	Substituting the right-most expression in \eqref{eq:K3_ICs} into the expression for $\Pi_3^{(\alpha)}(\rho_3,\theta_3,y_3,E)$ in Proposition~\ref{prop:pi3} yields an expression for $\Pi^{(\alpha)}(r_1,\theta_1,R,\eps_1)$. We obtain
	\begin{equation}
		\label{eq:Pi}
			\Pi^{(\alpha)} (r_1, \theta_1, R, \eps_1) = 
			\left( R, h_{\theta_3}^{(\alpha)}(\rho_3,\theta_3,y_3), h_{y_3}^{(\alpha)}(\rho_3,\theta_3,y_3), \eps_1 \right) ,
	\end{equation}
	\SJJ{and specify the second and third components in the following. We have} $h_{\theta_3}^{(\alpha)}(\rho_3, \theta_3, y_3) = \tilde h_{\theta_3}^{(\alpha)}(\rho_3, \theta_3, y_3) \mod 1$, with
	\[
		\tilde h_{\theta_3}^{(\alpha)}(\rho_3, \theta_3, y_3) = 
		\begin{cases}
		    \tilde h_{\theta_2}^{(1)}(r_1,\theta_1,\eps_1) +  \mathcal O\left(\ln \rho_3^{-1}\right) = \theta_1 + \frac{1}{c_0 \eps_1^2} + \mathcal{O}(\ln \eps_1^{-1}) , & \alpha = 1, \\
		    \tilde h_{\theta_2}^{(2)}(r_1,\theta_1,\eps_1) +  \mathcal O\left(\rho_3 \ln \rho_3^{-1}\right) = \theta_1 + \frac{R}{c_0 \eps_1} + \mathcal{O}(\eps_1 \ln \eps_1^{-1}) , &  \alpha = 2 ,
		\end{cases}
	\]
	\SJJ{where we used \eqref{eq:tilde_h_theta_1} and \eqref{eq:tilde_h_theta_2} in order to obtain the right-hand side. It remains to specify the third component in the right-hand side of \eqref{eq:Pi}. If} $\alpha = 1$, we can use the bounds in \eqref{eq:y3_bounds} to estimate
	\[
	\left(-\nu+\sigma_- E^3\right) \left( \frac{\eps_1}{E} \right)^2 + \mathcal{O}\left(\eps_1^3\right) \leq
		h_{y_3}^{(1)}(\rho_3,\theta_3,y_3) \leq 
	    \left(\nu+\sigma_+ E^3\right) \left( \frac{\eps_1}{E} \right)^2 + \mathcal{O}\left(\eps_1^3\right)
	\]
	for constants $\nu>0$ and $\sigma_-<\sigma_+$, such that $h_{y_3}^{(1)}(\rho_3,\theta_3,y_3) = \mathcal{O}(\eps_1^2)$. If $\alpha = 2$, then
	\begin{align*}
	    h_{y_3}^{(2)}(\rho_3,\theta_3,y_3) &= \left( E^2 h_{y_2}^{(2)}(E^{-1},\theta_2) -E^3\right)\left(\frac{\eps_1}{E}\right)^2 + \mathcal{O}\left(\eps_1^3\ln \eps_1^{-1}\right) \\
	    &= - \left(\frac{c(\theta_3)^2}{a(\theta_3) b(\theta_3)}\right)^{1/3}\Omega_0 \eps_1^2 + \mathcal{O}\left(\eps_1^3\ln \eps_1^{-1}\right), 
	\end{align*}
	\SJJ{where
	\begin{align*}
	    a(\theta_3) &= a\left(h_{\theta_3}^{(2)}(\rho_3, \theta_3, y_3)\right) = a(\theta_2) + \mathcal{O}(\eps_1 \ln \eps_1^{-1}), \\ 
	    b(\theta_3) &= b\left(h_{\theta_3}^{(2)}(\rho_3, \theta_3, y_3)\right) = b(\theta_2) + \mathcal{O}(\eps_1 \ln \eps_1^{-1}), \\ 
	    c(\theta_3) &= c\left(h_{\theta_3}^{(2)}(\rho_3, \theta_3, y_3)\right) = c(\theta_2) + \mathcal{O}(\eps_1 \ln \eps_1^{-1}),
	\end{align*}
	which follows from $a(\theta_2) = a(h_{\theta_1}^{(2)}(r_1,\theta_1,\eps_1))$, $b(\theta_2) = b(h_{\theta_1}^{(2)}(r_1,\theta_1,\eps_1))$, $c(\theta_2) = c(h_{\theta_1}^{(2)}(r_1,\theta_1,\eps_1))$ together with $\tilde h_{\theta_3}^{(2)}(\rho_3, \theta_3, y_3) = \tilde h_{\theta_1}^{(2)}(r_1,\theta_1,\eps_1) + \mathcal{O}(\eps_1 \ln \eps_1^{-1})$.}}
	
	Applying the blow-down transformations defined by the local coordinate formulae in \eqref{eq:chart_coordinates} to \eqref{eq:Pi} yields the map \SV{
	\begin{equation}
		\label{eq:pi_eps}
	    \pi_\eps^{(\alpha)}(r,\theta,R^2,\eps) =  
		\left( R, h_\theta^{(\alpha)}(r, \theta, \eps), h_y^{(\alpha)}(r, \theta, \eps), \eps \right) ,
	\end{equation}
	}for all $\eps \in (0,\eps_0]$ (the limiting value $\eps = 0$ is omitted because the blow-down transformation is only defined for $\eps > 0$). \SJJ{The estimates above `blow-down' to those given in} Assertion \SV{(b)} in Theorem \ref{thm:main}, since the map $\pi$ is identified with $\pi_\eps$ after omitting the trivial component $\eps \mapsto \eps$. \SJJ{Assertion (a) is a direct consequence of Assertion (b).}
	
	\
	
	\begin{figure}[t!]
		\centering
		\includegraphics[width=0.8\textwidth]{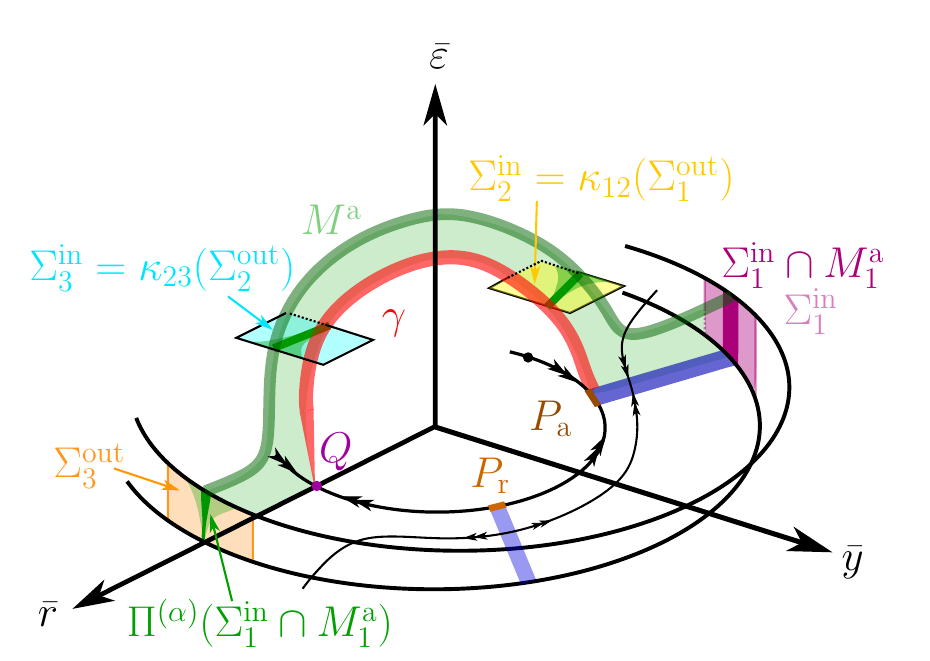}
		\caption{Global geometry and dynamics in the blown-up space projected into $(\bar r, \bar y, \bar \eps)$-space. The three-dimensional manifold $M^\textup{a}$ which extends from $\Delta_\eps^\textup{in}$ up to $\Delta_\eps^\textup{out}$ (which are identified with $\Sigma^{\textup{in}}_1$ and $\Sigma^{\textup{out}}_3$, respectively) is shown again in shaded green. Theorem \ref{thm:main} is proven by combining results obtained in charts $K_1$, $K_2$ and $K_3$: The flow from $\Sigma_1^\textup{in}$ to $\Sigma_1^{\textup{out}}$ is described by Proposition \ref{prop:pi1}, the flow from $\Sigma_2^\textup{in}$ to $\Sigma_2^{\textup{out}}$ is described by Proposition \ref{prop:pi2} and the flow from $\Sigma_3^\textup{in}$ to $\Sigma_3^{\textup{out}}$ is described by Proposition \ref{prop:pi3}.}
		\label{fig:maintheorem}
	\end{figure}
	
	It remains to prove the strong contraction property in Assertion \SV{(c)}. This can be done by differentiating the expression for \SV{$\Pi^{(\alpha)}$. Letting $\mathcal Q^{(\alpha)} \coloneqq \Pi_3^{(\alpha)} \circ \kappa_{23} \circ \Pi_2^{(\alpha)} \circ \kappa_{12}$, we obtain
	\[
	\frac{\partial \Pi^{(\alpha)}}{\partial r_1} (r_1, \theta_1, R, \eps_1) = 
	\frac{\partial \mathcal Q^{(\alpha)}}{\partial r_1} \left(\Pi_1^{(\alpha)}(r_1, \theta_1, R, \eps_1) \right) \frac{\partial \Pi_1^{(\alpha)}}{\partial r_1} (r_1, \theta_1, R, \eps_1) ,
	\]
	}after applying the chain rule. Direct calculations using Lemma \ref{lem:change}, Propositions \ref{prop:pi2} and \ref{prop:pi3} show that \SV{$(\partial \mathcal Q^{(\alpha)} / \partial r_1) (\Pi_1^{(\alpha)}(r_1, \theta_1, R, \eps_1))$} is (at worst) algebraically growing in $r_1$ and $\eps_1$. Since the second term \SV{$(\partial \Pi_1^{(\alpha)} / \partial r_1 )(r_1, \theta_1, R, \eps_1)$} is exponentially small due to the strong contraction in $K_1$, recall Proposition~\ref{prop:pi1} Assertion (b), it follows that
	\[
	\SV{\frac{\partial \Pi^{(\alpha)}}{\partial r_1} (r_1, \theta_1, R, \eps_1) = 
	\mathcal O \left( \me^{-\tilde \varrho / \eps_1^3} \right) ,}
	\]
	for a constant \SV{$\tilde \varrho > 0$}, assuming the constant $E$ bounding $\eps_1$ is sufficiently small. Assertion \SV{(c)} in Theorem~\ref{thm:main} follows for all $\eps \in (0,\eps_0]$ with $\eps_0$ sufficiently small after applying the blow-down transformation (in particular $\eps_1 = \eps / R$).
	\qed

    \section{Applications}
    \label{sec:applications}
    
    In this section we consider semi-oscillatory dynamics near folded limit cycle manifolds in two different applications. Specifically, we consider
    \begin{enumerate}
        \item Periodically forced oscillators of Li\'enard type;
        \item A toy model `normal form' for the study of tipping phenomena in climate systems proposed in \cite{Zhu2015}.
    \end{enumerate}
	The first of these will allow us to illustrate the application of the main results in Section \ref{sec:main_results}. The second is an example of a model exhibiting a folded limit cycle bifurcation in the partial singular limit $\eps_1 > 0$, $\eps_2 \to 0$ (recall system \eqref{eq:layer_delta}), which violates both Assumptions \ref{ass:1_new}-\ref{ass:2_new}.
	
	\subsection{Periodically Forced Oscillators}
	\label{sub:applications_1}
	
	Periodically forced oscillators in mechanical and electrical engineering can often be modelled using Li\'enard equations of the form
	\begin{equation}
	\label{eq:Lienard_oscillator}
    	x'' + \mu f(x) x' + g(x) = A(\omega t) ,
	\end{equation}
	where $A(\omega t) = A(\omega t + 1)$ for all $t \in \mathbb R$, the friction/resistance function $f(x)$ is smooth, and the parameters $\mu$ and $\omega$ determine the magnitude of the friction/resistance and the frequency of the (periodic) forcing term $A$, respectively.
	
	\begin{remark}
	    A very well-known and historically important example of a Li\'enard equation is the van der Pol equation \cite{vdP1920,vdp1926}. We refer to \cite{Burke2016,Guckenheimer2003,Haiduc2008} for detailed geometric studies of the (local and global) dynamics of the periodically forced \SVVV{van der Pol} equation in a number of frequency regimes.
	\end{remark}
	
	We recast equation \eqref{eq:Lienard_oscillator} as an autonomous dynamical system by (i) rescaling time via $\tau = \mu^{-1} t$, and (ii) introducing new variables via the \textit{Li\'enard transformation}
	\[
	y \coloneqq - \mu^{-1} x' + K(x) , \qquad s = \omega \mu \tau ,
	\]
	where $K(x) = - \int_0^x f(\xi) \, \dd\xi$. This leads to
	\[
	\begin{split}
		\dot x &= \mu^2 \left(- y + K(x) \right) , \\
		\dot s &= \mu \omega , \\
		\dot y &= g(x) - A(s) , 
	\end{split}
	\]
	where the overdot denotes now differentiation with respect to $\tau$. Exploiting the periodicity property \mbox{$A(s+1) = A(s)$} in order to replace $s$ with an angular variable $\theta \in \R / \mathbb Z$ and rewriting the system on the (new) fast time-scale defined by $\tilde t = \mu^{2} \tau$ leads to
	\[
	\begin{split}
		x' &= - y + K(x) , \\
		\theta' &= \mu^{-1} \omega , \\
		y' &= \mu^{-2} \left( g(x) - A(\theta) \right) , 
	\end{split}
	\]
	where by a slight abuse of notation, we allow the dash to denote differentiation with respect to the new time $\tilde t$. Our analysis applies in particular scaling regimes with $0 < \SV{\mu^{-1}\omega} , \mu^{-2} \ll 1$. Specifically, we derive results for \SV{scaling regimes} defined by
	\begin{equation}
	    \label{eq:lineard_eps}
	    	\eps^\alpha = \mu^{-1} \omega \vartheta^{-1}, \qquad 
        	\eps^3 = \mu^{-2} ,
	\end{equation}
	where $\alpha \in \mathbb N_+$ and $\vartheta > 0$ is a parameter such that $\vartheta = \mathcal O(1)$ as $\eps \to 0$. Note that this is equivalent to considering the scaling regimes defined by
	\begin{equation}
	\label{eq:lienard_scalings}
    	\omega = \vartheta \mu^{1 - \SV{\tfrac{2}{3}} \alpha} , \qquad 
	    \alpha = 1, 2, \ldots , \qquad 
	    0 < \mu^{-1} \ll 1 .
	\end{equation}
	After one more change of time-scale \SV{$\hat t = \vartheta \tilde t$}, our system becomes
	\[
	\begin{split}
		x' &= - \vartheta^{-1} y + \vartheta^{-1} K(x) , \\
		\theta' &= \eps^\alpha , \\
		y' &= \eps^3 \left( \vartheta^{-1} g(x) - \vartheta^{-1} A(\theta) \right) ,
	\end{split}
	\]
	\SV{where by a slight abuse of notation, we now allow the dash to denote differentiation with respect to the time $\hat t$}. Finally, we assume the following properties of the functions $K$, $g$ and $A$.
	
	\begin{assumption}
	\label{ass_app_1}
	    There exists a point $x_F$ such that
	    \[
	    \frac{\partial K}{\partial x}(x_F) = 0 , \qquad
	    \frac{\partial^2 K}{\partial x^2}(x_F) > 0, \qquad 
	    g(x_F) - A(\theta) < 0 ,
	    \]
	    for all $\theta \in \R / \mathbb Z$.
	\end{assumption}
	
	Assumption \ref{ass_app_1} implies that the critical manifold, given by $S_0 = \{ (x, \theta, K(x) ) : x \in \R, \theta \in \R / \mathbb Z \}$, contains a regular folded cycle $S_0^{\textup{c}} = \{ (x_F, \theta, K(x_F)) : \theta \in \R / \mathbb Z \}$. The set $S_0^{\textup{c}}$ can be moved to the \SV{`origin'} via the coordinate translation $(\tilde x, \tilde y) = (x - x_F, y - K(x_F))$, which yields the system
	\begin{equation}
	\label{eq:main_lienard}
	    \begin{split}
	    	\tilde x' &= - \SV{a(\theta) \tilde y} + \SV{b(\theta)} \tilde x^2 + \mathcal O(\tilde x^3) , \\
    		\theta' &= \eps^\alpha , \\
    		\tilde y' &= \eps^3 \left(- c(\theta) + \mathcal O(\tilde x) \right) ,
    	\end{split}
	\end{equation}
	which is -- up to relabelling $(\tilde x, \tilde y) \leftrightarrow (r,y)$ -- in the general form \eqref{eq:systemcontinuous} with
	\begin{equation}
	\label{eq:abc_Lienard}
	    a(\theta) \equiv \vartheta^{-1} , \qquad
	    b(\theta) \equiv \vartheta^{-1} \frac{1}{2} \frac{\partial^2 K}{\partial x^2}(x_F) =: \vartheta^{-1} \frac{1}{2} K''(x_F) , \qquad
    	c(\theta) = - \vartheta^{-1} \left( g(x_F) - A(\theta) \right) ,
	\end{equation}
	all of which are strictly positive due to Assumption \ref{ass_app_1}. 
	Using the same relabelling, we define entry and exit sections $\Delta^{\textup{in}}$ and $\Delta^{\textup{out}}$ analogously to those in \eqref{eq:Delta_in} and \eqref{eq:Delta_out}, respectively, and let $\pi^{(\alpha)} : \Delta^{\textup{in}} \rightarrow \Delta^{\textup{out}}$ denote the transition map induced by the forward flow of system \eqref{eq:main_lienard} for a given $\alpha \in \mathbb N_+$.
	
	\begin{thm}
		\label{thm:lienard} 
		Consider system \eqref{eq:main_lienard} with fixed $\alpha \in \mathbb N_+$. There exists an $\eps_0 > 0$ such that for all $\varepsilon \in (0, \varepsilon_0]$, the map $\pi^{(\alpha)}: \Delta^{\textup{in}} \rightarrow \Delta^{\textup{out}}$ is described by Theorem \ref{thm:main} after relabelling $(\tilde x, \tilde y) \leftrightarrow (r,y)$. In particular, we have
		\begin{equation*}
	    		\tilde h_{\theta}^{(\alpha)}(\SV{\tilde x},\theta,\varepsilon) = 
	    		\begin{cases}
	    		    \SV{\theta + \frac{R^2}{c_0} \eps^{-2} + \mathcal O(\ln \eps)}, & \alpha = 1, \\
		    	    \theta + \frac{R^2}{c_0} \eps^{-1} + \SV{\mathcal O(\eps \ln \eps)} , & \alpha = 2, \\
		    	    \psi(\theta) + \mathcal O(\eps^3 \ln \eps), & \alpha = 3, \\
		    	    \theta + \mathcal O(\eps^3 \ln \eps), & \alpha \geq 4 ,
			    \end{cases}
			\end{equation*}
		    where
		    \[
		    c_0 = - \SV{\vartheta}^{-1} g(x_F) + \SV{\vartheta}^{-1} \int_0^1 A(\theta) \dd \theta , \qquad 
		    \psi(\theta) = \theta + \frac{1}{2} \SV{\vartheta} \frac{K''(x_F)}{A(\theta) - g(x_F)} R + \mathcal O(R^2) ,
		    \]
		    and
			\begin{equation*}
    			h_{\tilde y}^{(\alpha)}(\tilde \theta,\eps) =
    			\begin{cases}
    			    \mathcal O(\eps^2), & \alpha = 1, \\
    			    - \left( \frac{2 ( g(x_F) - A(\tilde \theta) )^2}{K''(x_F)} \right)^{1/3} \Omega_0 \eps^2 + \mathcal O(\eps^3 \ln \eps) , & \alpha \geq 2 .
    			\end{cases}
			\end{equation*}
	\end{thm}

    \begin{proof}
        This is a direct application of Theorem \ref{thm:main}, using the notation of \eqref{eq:main_lienard} and the identities in~\eqref{eq:abc_Lienard}.
    \end{proof}
	
	Note that via \eqref{eq:lineard_eps}, the results in Theorem \ref{thm:lienard} can be reformulated and restated for the scaling regimes defined by \eqref{eq:lienard_scalings}. As for our main results in Section \ref{sec:main_results}, the results for scalings corresponding to semi-oscillatory dynamics ($\alpha \in \{1,2\}$) cannot be obtained using established theory for two time-scale systems. These results are new, to the best of our knowledge, even for the forced van der Pol example (the intermediate frequency regime considered in \cite{Burke2016} covers the case $\alpha = 3/2$).

	\subsection{Tipping Phenomena in Climate Systems}
	\label{sub:applications_2}
	
	The authors in \cite{Zhu2015} proposed the following simple normal form for studying early warning signs for saddle-node induced tipping in climate systems:
	\[
	z'(t) = \tilde a(t) - z^2 \SV{-} \mathcal A \sin(2 \pi \omega t) , \qquad \tilde a(t) = \tilde a(0) - \tilde \delta t ,
	\]
	where $0 < \tilde \delta \ll 1$ and the parameters $\mathcal A, \omega > 0$ determine the amplitude, frequency of the sinusoidal periodic forcing term, respectively. This can recast as an autonomous system of the form
	\begin{equation}
	\label{eq:tipping_1}
    	\begin{split}
	    	z' &= \tilde a - z^2 \SV{-} \mathcal A \sin(2 \pi \theta) , \\
    		\theta' &= \omega, \\
    		\tilde a' &= - \tilde \delta ,
    	\end{split}
	\end{equation}
	where $\theta \in \R / \mathbb Z$. We shall be interested in the dynamics when $0 < \tilde \delta, \omega \ll 1$,
	and in particular the semi-oscillatory case $0 < \tilde \delta \ll \omega \ll 1$. In this case, system \eqref{eq:tipping_1} is oscillatory with respect to the partial singular limit $\omega > 0$, $\tilde \delta \to 0$, but stationary with respect to the double singular limit $(\omega, \tilde \delta) \to (0,0)$.
	
	\
	
	We are particularly interested in system \eqref{eq:tipping_1} as a model example of a system which violates Assumptions \ref{ass:1_new}-\ref{ass:2_new}.
	
	\begin{proposition}
	    Consider system \eqref{eq:tipping_1} with $\omega > 0$ fixed and $0 < \tilde \delta \ll 1$. This system has a regular fold of limit cycles at $z = \tilde a = 0$. In particular, the necessary and sufficient conditions on the Poincar\'e map in \eqref{eq:map_sn_conds} and \eqref{eq:map_sn_conds_2} are satisfied, however Assumptions \ref{ass:1_new} and \ref{ass:2_new} are not.
	\end{proposition}
	
	\begin{proof}
	    The fact that Assumptions \ref{ass:1_new}-\ref{ass:2_new} are violated follows from direct calculations on the system obtained from \eqref{eq:tipping_1} in the partial singular limit $\omega > 0$, $\tilde \delta \to 0$. The details are omitted for brevity.
	    
	    In order to verify the existence of a regular fold of cycles, we derive an expression for the Poincar\'e map $P$ induced on a section contained in the plane $\{\theta = 0\}$. Using the fact that
	    \[
	    \frac{\dd z}{\dd \theta} = \omega^{-1} \left( \tilde a - z^2 \SV{-} \mathcal A \sin(2 \pi \theta) \right) , \qquad 
	    \frac{\dd \tilde a}{\dd \theta} = - \omega^{-1} \tilde \delta ,
	    \]
	    we obtain
	    \[
	    P(z, \tilde a, \omega, \tilde \delta) = 
	    \begin{pmatrix}
	        P_z(z, \tilde a, \omega, \tilde \delta) \\
	        P_{\tilde a}(z, \tilde a, \omega, \tilde \delta)
	    \end{pmatrix}
	    = 
	    \begin{pmatrix}
	        z \\
	        \tilde a
	    \end{pmatrix}
	    + \omega^{-1}
	    \begin{pmatrix}
	        \int_0^1\left( \tilde a(\theta) - z(\theta)^2 \right) \dd \theta  \\
	        - \tilde \delta
	    \end{pmatrix} .
	    \]
	    This expression can be used to check the necessary and sufficient conditions in \eqref{eq:map_sn_conds} and \eqref{eq:map_sn_conds_2} for a regular fold of cycles directly.
	\end{proof}
	
    Our aim in what follows is to demonstrate that problems of this kind can (to a large extent) be analysed with existing theory for two time-scale systems. We shall restrict attention to scaling regimes defined by
	\[
	\tilde \delta = \eps^3 \nu, \qquad \omega = \eps^\alpha, 
	\]
	where $\nu = \mathcal O(1)$ as $\eps \to 0$ and 
	$\alpha \in \mathbb N_+$. We shall also simplify the fast equation by defining a new variable $a \coloneqq \tilde a \SV{-} \mathcal A \sin(2 \pi \theta)$. This leads to
	\begin{equation}
	\label{eq:tipping_main}
    	\begin{split}
	    	z' &= a - z^2 , \\
	    	\theta' &= \eps^\alpha , \\
	    	a' &= - 2 \pi \eps^\alpha \mathcal A \cos(2 \pi \theta) - \eps^3\SV{\nu} .
	\end{split}
	\end{equation}
	The main results for this section will describe the dynamics of system \eqref{eq:tipping_main}. Letting $\eps \to 0$ yields the layer problem
	\begin{equation}
	    \begin{split}
	    \label{eq:tipping_layer}
		z' &= a - z^2, \\
		\theta' &= 0, \\
		a' &= 0.
		\end{split}
	\end{equation}
	The critical manifold is given by $S_0 = \{ (z,\theta,z^2) : z \in \R, \theta \in \R / \mathbb Z \}$, and the non-trivial eigenvalue of the linearisation along $S_0$ is $\lambda = - 2z$. This implies a folded critical manifold structure $S_0 = S_0^{\textup{a}} \cup S_0^{\textup{c}} \cup S_0^{\textup{r}}$, where $S_0^{\textup{a}} = S_0 \cap \{z > 0\}$ ($S_0^{\textup{r}} = S_0 \cap \{z < 0\}$) is normally hyperbolic and attracting (repelling), and $S_0^{\textup{c}} = \{(0,\theta,0): \theta \in \R / \mathbb Z \}$ is of regular fold type.
	
	The reduced problem on $S_0$ will differ depending on the scaling, i.e.~depending on the value of $\alpha$.
	We are primarily interested in the cases $\alpha \in \{1,2\}$,
	for which $0 < \SV{ \tilde\delta \ll \omega} \ll 1$; see however Remark \ref{rem:cm_other_alphas} below \SV{for the cases $\alpha \geq 3$}.
	Fixing $\alpha \in \{1,2\}$, rewriting \SV{system \eqref{eq:tipping_main}} on the slow time-scale $\tau = \eps^\alpha t$, and taking the limit $\eps \to 0$ yields the reduced problem
	\begin{align} \label{eq:tipping_reduced}
	\begin{split}
	    0 &= a - z^2 , \\
	    \dot \theta &= 1 , \\
	    \dot a &= - 2 \pi \mathcal A \cos (2\pi \theta) .
	\end{split}
	\end{align}
	This leads to the following reduced vector field on $S_0$, expressed in the $(z,\theta)$-coordinate chart:
	\begin{equation}
	    \label{eq:cm_reduced_1}
    	\begin{split}
    	    \dot z &= - \frac{\pi \mathcal A}{z} \cos(2 \pi \theta) , \\
    	    \dot \theta &= 1.
	    \end{split}
	\end{equation}
	It is typical for problems of this kind to consider the so-called \textit{desingularised reduced problem} \cite{Kuehn2015,Szmolyan2001,Wechselberger2019}. This may be obtained after a (singular) time transformation which amounts to multiplication of the vector field by $z$ (see again Remark \ref{rem:desingularisation}), leading to
	\begin{equation}
	    \label{eq:cm_desing_reduced_1}
    	\begin{split}
    	    \dot z &= - \pi \mathcal A \cos(2 \pi \theta) , \\
    	    \dot \theta &= z.
    	\end{split}
	\end{equation}
	System \eqref{eq:cm_desing_reduced_1} is orbitally equivalent to system \eqref{eq:cm_reduced_1} on $S_0 \setminus S_0^{\textup{c}}$, but the orientation of orbits on $S_0^{\textup{r}}$ (where $z < 0)$ is reversed. Direct calculations reveal the presence of two equilibria along $S_0^{\textup{c}}$, namely
	\[
	p_s : ( 0, 1/4 ) , \qquad 
	p_c : ( 0, 3/4 ) ,
	\]
	which are of neutral folded saddle and folded center type, respectively (see \cite{Szmolyan2001} for definitions), see Figure~\ref{fig:tipping_reduced}. Moreover, system \eqref{eq:cm_desing_reduced_1} is Hamiltonian, with Hamiltonian function
	\[
	H(z, \theta) = \frac{\mathcal A}{2} \sin(2 \pi \theta) + \frac{z^2}{2} .
	\]
	Except for the points $p_s$ and $p_c$, there are three different types of orbits shown in Figure \ref{fig:tipping_reduced}. These are identified in the following result.
	
	\begin{figure}[t!]
		\centering
		\includegraphics[width=0.8\textwidth]{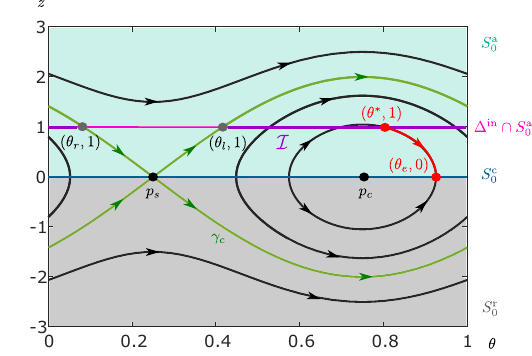}
		\caption{\SJJ{Dynamics for the reduced problem \eqref{eq:cm_reduced_1} on $S_0$, projected into the $(\theta,z)$-plane. We sketch the case with $\mathcal A = 2$ and $R = 1$, where we recall that $R > 0$ is the constant used to define $\Delta^{\textup{in}}$. The normally hyperbolic and attracting (repelling) branch $S_0^{\textup{a}}$ ($S_0^{\textup{r}}$) is indicated in shaded turquoise (gray). The fold cycle $S_0^{\textup{c}}$ which separates these two branches is \SV{sketched} in blue. There are two folded singularities on $S_0^{\textup{c}}$: a neutral folded saddle $p_s$, and a folded center $p_c$. The singular true and faux canards through $p_s$ coincide in this case (denoted $\gamma_c$ and shown in green), forming a limit cycle of period $2$, and the region bounded outside of $\gamma_c$ is foliated by periodic orbits of period $1$. Finally, we sketch the forward evolution of an initial condition in $\Delta^{\textup{in}} \cap \SV{S_0^{\textup{a}}}$ with initial angle $\theta^\ast \in \mathcal I$ (the interval $\mathcal{I}$ is shown in blue), i.e.~inside the region bounded by $\gamma_c$. Such points reach a regular jump point on $S_0^{\textup{c}}$ with angle $\theta_e \in (3/4,1) \cup [0,1/4)$ in finite time. For $\mathcal A = 2$ and $R = 1$ we have $\theta_l = 5/12$, $\theta_r = 1/12$. The initial condition in red has angle $\theta^\ast = 5/6$, for which our results imply a jump angle of $\theta_e(\theta^\ast) = \SV{11/12}$.}}
		\label{fig:tipping_reduced}
	\end{figure} 
	
	\begin{lemma}
	\label{lem:desing_reduced_prob}
	    Consider the reduced problem \eqref{eq:cm_reduced_1}. Orbits are contained within constant level sets $H(z,\theta) = K \geq - \mathcal A / 2$, where
	    \begin{enumerate}
	        \item $K = - \mathcal A / 2$ corresponds to the folded center $p_c$.
	        \item $K \in (- \mathcal A / 2, \mathcal A / 2)$ correspond to orbits that connect to the folded cycle $S_0^{\textup{c}}$ in both forward and backward time.
	        \item $K = \mathcal A / 2$ corresponds to two saddle homoclinic orbits in the desingularised reduced problem \eqref{eq:cm_desing_reduced_1}, and a \SJJJ{periodic orbit} $\gamma_c$ of period $\tau = 2$ along the true and faux canards through the neutral folded saddle $p_s$ in the reduced problem \eqref{eq:cm_reduced_1}.
	        \item $K > \mathcal A / 2$ corresponds to \SJJJ{periodic orbits} of period $\tau = 1$ outside of the region bounded by $\gamma_c$.
	    \end{enumerate}
	\end{lemma}
	
	\begin{proof}
	    This follows from direct calculations with the desingularised reduced problem \eqref{eq:cm_desing_reduced_1} and a reversal of time on \SV{$S_0^{\textup{r}}$}.
	\end{proof}
	
    We can use the information obtained above to describe the local dynamics near $S_0^{\textup{c}}$. Specifically, we want to describe the transition map $\pi^{(\alpha)} : \Delta^{\textup{in}} \to \Delta^{\textup{out}}$ induced by the forward flow of system \eqref{eq:tipping_main}, where
    \[
    \Delta^{\textup{in}} \coloneqq \left\{(z,\theta,R^2) : |\SV{z} - R| \leq \beta, \theta \in \R / \mathbb Z \right\} , \qquad
    \Delta^{\textup{out}} \coloneqq \left\{(-R,\theta,a) : \theta \in \R / \mathbb Z , a \in [-a_0,a_0] \right\} ,
    \]
    for small but fixed $R, \beta, a_0 > 0$. A \textit{singular transition map} $\pi_0^{(\alpha)} : \Delta^{\textup{in}} \to \Delta^{\textup{out}}$ for $\eps = 0$ can be constructed explicitly, by concatenating solutions to the layer and reduced problems described above. If we choose $R > 0$ so that
        \begin{equation}
        \label{eq:Delta_in_cond}
            R^2 < 2 \mathcal A ,
        \end{equation}
    then $\Delta^{\textup{in}} \cap S_0^{\textup{a}}$ intersects the singular canard solution $\gamma_c$ at the two points $(R,\theta_{l})$ and $(R,\theta_{r})$, where \mbox{$\theta_l \in (1/4, 3/4]$} and $\theta_r \in [3/4, 0) \cup [0, 1/4)$ solve
    \[
    \sin(2 \pi \theta_{l,r}) = 1 - \frac{R^2}{\mathcal A} ,
    \]
    see Figure \ref{fig:tipping_reduced}. Let $\mathcal I \subset \R / \mathbb Z$ denote the interior of the interval between $\theta_l$ and $\theta_r$ which is contained within the region bounded by $\gamma_c$.
    Different types of singular orbits can be constructed depending on whether or not one starts with an initial angle $\theta \in \mathcal I$, $\theta \in \{\theta_l, \theta_r\}$ or $\theta \notin \overline{\mathcal I}$.
    
    Consider an initial condition $(z^\ast, \theta^\ast, R^2\SV{ ) \in  \Delta^{\textup{in}}}$ with $\theta^\ast \in \mathcal I$. This is mapped by the layer flow \SV{of system \eqref{eq:tipping_layer}} to the point $(R, \theta^\ast, R^2) \in \Delta^{\textup{in}} \cap S_0^{\textup{a}}$. Since $\theta^\ast \in \mathcal I$, the point $(R, \theta^\ast, R^2)$ lies on an orbit of the desingularized reduced problem \eqref{eq:cm_desing_reduced_1} which is contained within a level curve $H(z,\theta) = K \in (-\mathcal A/2, \mathcal A/2)$. In fact, one can show directly that
    \begin{equation}
        \label{eq:K_theta}
    K = K(\theta^\ast) \coloneqq \frac{\mathcal A}{2} \sin( 2 \pi \theta^\ast) + \frac{R^2}{2} .
    \end{equation}
    Lemma \ref{lem:desing_reduced_prob} implies that the forward evolution of $(R, \theta^\ast, R^2)$ under the reduced problem \eqref{eq:cm_reduced_1} reaches $S_0^\textup{c}$ at a regular jump point $(0,\theta_e(\theta^\ast),0)$, where $\theta_e(\theta^\ast) \in (3/4, 1) \cup [0,1/4)$. In fact, the value of $\theta_e$ can be caluclated explicitly using
    \begin{equation}
        \label{eq:theta_e}
       H(0,\theta_e) = \frac{\mathcal A}{2} \sin(2 \pi \theta_e) = K(\theta^\ast) , \qquad 
       \theta_e(\theta^\ast) \in (3/4, 1) \cup [0,1/4) .
    \end{equation}
    Note that the containment condition on the right ensures that $\theta_e(\theta^\ast)$ is uniquely determined, i.e.~that the function $\theta_e : \mathcal I \to (3/4,1) \cup [0,1/4)$ is well-defined. Since $(0,\theta_e(\theta^\ast),0) \in S_0^\textup{c}$ is a (classical) regular jump point, we can connect to the layer flow of system \SV{\eqref{eq:tipping_layer}} and map it to the point $(-R, \theta_e(\theta^\ast), 0) \in \Delta^\textup{out}$. Thus we have shown that for all initial conditions with $(z^\ast, \theta^\ast, R^2)$ with $\theta^\ast \in \mathcal I$, the singular transition map is
    \[
    \pi_0^{(\alpha)}(z^\ast, \theta^\ast, R^2) = 
    \left( -R, \theta_e(\theta^\ast), 0 \right) ,
    \]
    where $\theta_e(\theta^\ast)$ is uniquely determined by the expressions \eqref{eq:K_theta} and \eqref{eq:theta_e}.
    
    Now consider an initial condition $(z^\ast, \theta^\ast, R^2\SV{)\in \Delta^\textup{in}}$ with $\theta^\ast \notin \mathcal I$. In this case the point $(z^\ast, \theta^\ast, R^2)$ is mapped via the layer flow to $(R, \theta^\ast, R^2) \in S_0^{\textup{a}}$, but this point lies on an orbit of the reduced problem with $H(z,\theta) \geq \mathcal A/2$. Specifically, if $\theta^\ast \in \{\theta_l, \theta_r\}$, then $(R, \theta^\ast, R^2) \in \gamma_c$. The map $\pi_0^{(\alpha)}$ is multi-valued in this case, since every point on $\gamma_c \cap S_0^\textup{r}$ can be concatenated with an orbit segment of the layer problem. In particular,
    \[
    \pi_0^{(\alpha)}(z^\ast, \theta_{l,r}, R^2) = 
    \left\{ \left( -R, \theta, \SV{\mathcal A} \left(1 - \sin(2 \pi \theta) \right) \right) : \theta \in \R / \mathbb Z \right\} .
    \]
    If $\theta^\ast \notin \overline{\mathcal I}$, then $(R, \theta^\ast, R^2)$ lies on a periodic orbit of the reduced problem which is bounded above $\gamma_c$. In this case solutions never leave $S_0^{\textup{a}}$, and the map $\pi_0^{(\alpha)}$ is not defined.
    
    \ 
    
    The following result describes the perturbation of singular orbits with initial conditions in the sector defined by requiring that $\theta \in \mathcal I$.
    
    \begin{proposition}
        \label{prop:cm_1}
        Consider system \eqref{eq:tipping_main} with $\alpha \in \{1,2\}$, assume the constant $R$ which defines $\Delta^{\textup{in}}$ satisfies \eqref{eq:Delta_in_cond}, and let
        \[
        \SV{\tilde \Delta^{\textup{in}}} \coloneqq \{ (z,\theta,R^2) \in \Delta^{\textup{in}} : \theta \in \mathcal J_1 \} ,
        \]
        for a closed interval $\mathcal J_1 \subset \SV{\mathcal{I}}$. There exists an $\eps_0 > 0$ such that for all $\eps \in (0,\eps_0)$, the restricted map $\tilde \pi^{(\alpha)} \coloneqq \pi^{(\alpha)}|_{\tilde \Delta^{\textup{in}}} : \tilde \Delta^{\textup{in}} \to \SV{\Delta^{\textup{out}}}$ is well-defined and given by
        \[
        \tilde \pi^{(\alpha)}(z, \theta, R^2) = 
        \left( -R, \tilde \pi^{(\alpha)}_\theta(z, \theta, \eps), \tilde \pi^{(\alpha)}_a(z, \theta, \eps) \right) ,
        \]
        where
        \[
        \tilde \pi_\theta^{(\alpha)}(z, \theta, \eps) = \theta_e(\theta) + \mathcal O(\eps^\alpha \ln \eps) , \qquad
        \tilde \pi_a^{(\alpha)}(z, \theta, \eps) = \mathcal O(\eps^{2 \alpha / 3}) ,
        \]
        $\tilde \pi^{(\alpha)} \to \pi_0^{(\alpha)}|_{\SV{\tilde \Delta^{\textup{in}}}}$ as $\eps \to 0$, and $\theta_e(\theta)$ is uniquely determined by equations \eqref{eq:K_theta} and \eqref{eq:theta_e}. Explicitly,
        \begin{equation}
        \label{eq:cm_theta}
            \theta_e(\theta) = \frac{1}{2 \pi} \arcsin \left( \sin(2 \pi \theta) + \frac{R^2}{\SV{\mathcal{A}}} \right) , \qquad 
            \theta \in (3/4, 1) \cup [0, 1/4) .
        \end{equation}
    \end{proposition}
    
    \begin{proof}
        Consider the forward evolution of initial conditions $(z^\ast, \theta^\ast, R^2) \in \SV{\tilde\Delta^{\textup{in}}}$ under system \eqref{eq:tipping_main}. By Fenichel theory, such solutions are initially attracted at an exponential rate to the Fenichel slow manifold $S_\eps^{\textup{a}}$ which perturbs from $S_0^{\textup{a}}$, after which they remain $\mathcal O(\eps^\alpha)$-close to solutions of the reduced vector field induced by system \SV{\eqref{eq:tipping_reduced}} until reaching a neighbourhood of the regular jump point $(-R, \theta_e(\theta^\ast),0)$. The local dynamics near $(-R, \theta_e(\theta^\ast),0)$ are governed by the system
        \[
        \begin{split}
            z' &= a - z^2 , \\
	    	\theta' &= \eps^\alpha , \\
	    	a' &= \eps^\alpha \left( - 2 \pi \mathcal A \cos(2 \pi \theta_e(\theta^\ast)) + \mathcal O( (\theta - \theta_e(\theta^\ast))^2 , \eps^{3-\alpha}) \right) ,
        \end{split}
        \]
        where $\cos(2 \pi \theta_e(\theta^\ast)) > 0$ since $\theta_e(\theta^\ast) \in (3/4, 1) \cup [0, 1/4)$. Direct calculations verify that this is a stationary two time-scale problem with a regular fold curve passing through the point $(0,\theta^\ast,0)$. This observation allows for a derivation of the desired result using a slight adaptation of the well-known result in \cite{Szmolyan2004}.
    \end{proof}
	
	Proposition \ref{prop:cm_1} describes regular jump type dynamics for solutions with initial conditions in $\tilde \Delta^{\textup{in}}$, which can be chosen arbitrarily close to $\Delta^{\textup{in}}|_{\theta \in \mathcal I}$ by decreasing $\eps_0$ if necessary. It is worthy to note that the leading order approximation for the `jump angle' $\tilde \pi^{(\alpha)}_\theta(z, \theta, \eps) \sim \theta_e(\theta)$ can be calculated explicitly as a function of the initial angle using \eqref{eq:cm_theta}. It is also worthy to note that for $\alpha \in \{1,2\}$ we have $\tilde \pi^{(\alpha)}_a(z,\theta,\eps) = \mathcal O(\eps^{2\alpha/3})$, in contrast to the $\mathcal O(\eps^2)$ estimates for the parameter drift in Theorem \ref{thm:main}. This is a consequence of the fact that the `parameter drift' in system \eqref{eq:tipping_main} in $a$ occurs on the intermediate, rotational time-scale. This is not the case for systems satisfying Assumptions \ref{ass:1_new}-\ref{ass:2_new}.
	
	\ 
	
	We may also consider the case in which initial conditions are chosen `higher up', in an annular section
		\[
		\Sigma^{\textup{in}} \coloneqq \left\{(z,\theta,\rho^2) : |\SV{z} - R| \leq \beta, \theta \in \R / \mathbb Z \right\}
		\]
		where the constant $\rho > 0$ satisfies\SJJJ{
		\begin{equation}
			\label{eq:rho_in_cond}
			\rho^2 > 2 \mathcal A > R^2,
		\end{equation}
		}c.f.~equation \eqref{eq:Delta_in_cond}. A subset of solutions with initial conditions in $\Sigma^{\textup{in}}$ are described by the following result. We formulate \SV{it} in terms of the transition map $\Pi^{(\alpha)} : \Sigma^{\textup{in}} \to \Delta^{\textup{out}}$.
	
	\begin{proposition}
	\label{prop:cm_2}
	    Consider system \eqref{eq:tipping_main} with $\alpha \in \{1,2\}$, assume the constant $\rho$ which defines $\Sigma^{\textup{in}}$ satisfies \eqref{eq:rho_in_cond}, and let
        \[
        \SV{\widetilde \Sigma^{\textup{in}} \coloneqq \{ (z,\theta,\rho^2) \in \Sigma^{\textup{in}} : \theta \in \mathcal J_2 \}} ,
        \]
        where $\mathcal J_2 := \left\{ \theta \in \R / \mathbb Z : \SJJJ{T(\theta,\eps)} \in \mathcal J_1 \right\}$, where $\mathcal J_1 \subset \mathcal I$ is the interval from Proposition \ref{prop:cm_1}, and $\SJJJ{T(\theta,\eps)}$ is the minimal positive solution of the transcendental equation\SVV{
        \begin{equation}
	        \label{eq:transcendental_eqn}
	        \rho^2 + \mathcal A \sin( 2\pi \theta) - R^2 = \eps^3 \nu T + \mathcal A \sin(2 \pi (\theta + \eps^\alpha T)) .
	    \end{equation}
	    }Then there exists an $\eps_0 > 0$ such that for all $\eps \in (0,\eps_0)$, the restricted transition map  $\SV{\widetilde \Pi^{(\alpha)} \coloneqq \Pi^{(\alpha)}|_{\widetilde \Sigma^\textup{in}}:}$ $\SV{\widetilde \Sigma^\textup{in} \rightarrow \Delta^\textup{out}}$ is given by
	    \begin{equation}
	        \label{eq:Pi_cm_1}
	        \widetilde \Pi^{(\alpha)}(z, \theta, \rho^2) = 
    	    \left( -R, \theta_e(\theta + \eps^\alpha \SJJJ{T(\theta,\eps)} \mod 1) + \mathcal O(\eps^\alpha \ln \eps), \mathcal O(\eps^{2 \alpha / 3}) \right) .
	    \end{equation}
	\end{proposition}
	
	\begin{proof}
	    The idea is to consider $\widetilde \Pi^{(\alpha)}$ as a composition $\widetilde \Pi^{(\alpha)} = \tilde \pi^{(\alpha)} \circ \hat \pi^{(\alpha)}$, where $\tilde \pi^{(\alpha)}$ is the local transition map described in Proposition \ref{prop:cm_1} above and $\hat \pi^{(\alpha)} : \SV{ \widetilde\Sigma^{\textup{in}}} \to \tilde \Delta^{\textup{in}}$.
	    The $(\theta',a')$ system decouples from system~\eqref{eq:tipping_main}, and can be solved directly. We obtain
	    \[
	    \theta(t) = \theta^\ast + \eps^\alpha t \mod 1, \qquad
	    \SVV{a(t) = \rho^2 + \mathcal A \left(\sin( 2 \pi \theta^\ast) \SV{ -} \sin(2 \pi (\theta^\ast + \eps^\alpha t)) \right) - \eps^3 \nu t },
	    \]
	    for all $t \in \R$, where $\theta^\ast = \theta(0)$. The transition time $T = \SJJJ{T(\theta^\ast,\eps)} > 0$ is determined implicitly as the first positive solution of $a(T) = R^2$, i.e.~of the transcendental equation \eqref{eq:transcendental_eqn}. The estimate for $z(T)$ follows from Fenichel theory, which implies that the attracting slow manifold $S_\eps^{\textup{a}}$ which perturbs from $S_0^{\textup{a}}$ can be written as a graph $z = \varphi_\eps(\theta, a) = \sqrt{a} + \mathcal O(\eps^\alpha)$ in this regime (normal hyperbolicity is guaranteed by the fact that $a(t) \geq R^2 > 0$ for all $t \in [0,\SJJJ{T(\theta,\eps)}]$).
	    
	    The preceding arguments imply that
	    \[
	    \hat \pi^{(\alpha)}(z, \theta, \rho^2) = 
	    \left( \SV{\varphi_\eps(\theta(T),R^2)} + \mathcal O(\me^{-\kappa / \eps^\alpha}), \theta(T), R^2 \right) .
	    \]
	    Since the domain of $\hat \pi^{(\alpha)}$ is $\widetilde\Sigma^{\textup{in}}$, we have that $\theta(T) = \theta(\SJJJ{T(\theta^\ast,\eps)}) \in \mathcal J_1$, where $\mathcal J_1 \subset \mathcal I$ is the interval from Proposition \ref{prop:cm_1}. Thus the composition $\widetilde \Pi^{(\alpha)} = \tilde \pi^{(\alpha)} \circ \hat \pi^{(\alpha)}$ is well-defined and, using Proposition \ref{prop:cm_1}, given by \eqref{eq:Pi_cm_1}.
	\end{proof}
	
	\begin{remark}
	    Due to the transcendental nature of equation \eqref{eq:transcendental_eqn}, a closed form expression for the function $T = \SJJJ{T(\theta,\eps)}$ can only be given in terms of special functions. Nevertheless, for any given $\theta$, $\SJJJ{T(\theta,\eps)}$ be calculated to arbitrary precision with standard numerical methods (e.g.~Newton or bisection methods).
	\end{remark}
	
	Propositions \ref{prop:cm_1} and \ref{prop:cm_2} describe solutions of `regular jump type', but they do not explain what happens to solutions with initial conditions in $\Delta^{\textup{in}} \setminus \mathcal J_1$. This corresponds to a set of initial conditions that is `small' in these sense that the diameter of the interval $[0,1) \setminus \overline{\mathcal I}$ tends to zero as $R \to 0$, but for fixed $R > 0$ it is still $\mathcal O(1)$ with respect to $\eps \to 0$. Moreover, this set contains solutions close to the singular canard $\gamma_c$, which are expected to play an important role in determining the qualitative dynamics~\cite{Szmolyan2001}. A detailed description of the dynamics in this case is left for future work.
	
	\begin{remark}
	    \label{rem:cm_other_alphas}
	    For $\alpha \geq 3$, the dynamics can be analysed and understood entirely with established theory for two time-scale systems. The case $\alpha = 3$ is non-trivial, and features a degenerate bifurcation in the desingularized reduced problem as the folded saddle $p_s$ and neutral folded center $p_c$ collide and annihilate at a particular value of the forcing $\mathcal A = \mathcal A_c > 0$. For $\mathcal A < \mathcal A_c$, the folded cycle $S_0^{\textup{c}}$ is regular, i.e.~there are no folded singularities. For $\alpha \geq 4$, $S_0^{\textup{c}}$ is always regular. A direct application of the results in \cite{Szmolyan2004} in this case yields an explicit expression for the transition map:
	    \[
	    \SV{\pi^{(\alpha)}}\left( z, \theta, R^2 \right) = 
	    \left( -R, \theta + \mathcal O(\eps^\alpha \ln \eps), \mathcal O(\eps^{2 \alpha / 3}) \right) , \qquad
	    \alpha \geq 4 .
	    \]
	\end{remark}
	
	\SJJJ{
	\begin{remark}
	    The local geometry and dynamics near the fold lines/circles in the forced van der Pol system considered in e.g.~\cite{Burke2016,Guckenheimer2003,Haiduc2008} are characterised by alternation of folded saddles and folded focus singularities, in contrast to the alternation of folded saddle and folded center singularities in Figure \ref{fig:tipping_reduced}. Nevertheless, there are many similarities, including the fact that the image of $S^{\textup{a}}_\eps$ on $\Delta^{\textup{out}}$ is non-circular as $\eps \to 0$, due to the folded saddles. This feature, which is not present for the systems described by Theorem \ref{thm:main}, is utilised in the construction of horseshoes in the forced van der Pol system \cite{Haiduc2008}.
	\end{remark}
	}

	\section{Summary and Outlook}
	\label{sec:summary}
	
	GSPT is an established and powerful tool for the analysis of stationary fast-slow systems, particularly when combined with the blow-up method. The corresponding theory for oscillatory fast-slow systems which possess a limit cycle manifold instead of (or in addition to) a critical manifold is less developed, despite the fact that oscillatory systems are common in applications. One reason for this appears to be that the theory for stationary fast-slow systems, notably Fenichel-Tikhonov theory and the blow-up method, depend on quasi-equilibrium properties possessed by stationary but not oscillatory fast-slow systems. The main purpose of this article has been to show that both Fenichel-Tikhonov theory and \SJJJ{the} blow-up method can be applied to study the dynamics of a class of multiple time-scale systems with three or more time-scales, as long as the angular/rotational dynamics are sufficiently slow relative to the radial dynamics. The class of systems considered, namely those in the general form \eqref{eq:original_system}, contains a \SJJ{`semi-oscillatory'} class of systems that are oscillatory with respect to the partial singular limit $\eps_1 > 0, \ \eps_2 \to 0$, but stationary with respect to the double singular limit $(\eps_1,\eps_2) \to (0,0)$. \SJJ{Although our approach is not applicable in the purely oscillatory case, o}ur analysis showed that Fenichel-Tikhonov theory and the blow-up method can be applied directly \SJJ{in the semi-oscillatory case}.
	
	More concretely, we focused on the dynamics near a (three time-scale) global singularity \SJJ{corresponding to a kind of} folded cycle bifurcation in the layer problem \SJJ{obtained in the partial singular limit}. In order to do so we derived a \SJJ{prototypical system} for the three time-scale global singularity near the non-hyperbolic cycle $S_0^{\textup{c}}$, recall Proposition \ref{prop:normal_form}, and studied the transition map induced by the flow. Our main result is Theorem \ref{thm:main}, which is stated for \SJJ{system} \eqref{eq:systemcontinuous}, which provides a rigorous characterisation of the asymptotic and strong contraction properties of the flow near $S_0^{\textup{c}}$. The \SJJ{asymptotics $h_y^{(\alpha)}(\theta,\eps) = \mathcal O(\eps^2)$} and the strong contraction property are the same as for the stationary regular fold point studied in \SJJ{\cite{Krupa2001a,Mishchenko1975,Szmolyan2004}, not only in the semi-oscillatory cases $\alpha \in \{1,2\}$ considered herein in detail, but for all $\alpha \in \mathbb N_+$.}
	\SJJ{The semi-oscillatory cases $\alpha \in \{1,2\}$ are distinguished from the cases $\alpha \geq 3$ by the fact that the leading order estimates for $h_y^{(\alpha)}(\theta,\eps)$ are, in general, non-constant functions of $\theta$. We derived an explicit expression for the leading order estimate in terms of the functions $a(\theta)$, $b(\theta)$ and $c(\theta)$ appearing in system \eqref{eq:systemcontinuous} when $\alpha = 2$ in particular. In addition, we provided detailed estimates for the exit angle, recall \eqref{eq:theta_asymptotics}, as well as estimates for t}he total number of rotations about the $y$-axis \SJJ{as solutions traverse the region between $\Delta^{\textup{in}}$ and $\Delta^{\textup{out}}$, recall Corollary \ref{cor:num_rotations}}.
	Theorem \ref{thm:main} is proved using the blow-up method in Section~\ref{sec:blowupfold}, thereby demonstrating the applicability of geometric blow-up techniques to study global singularities of limit cycle type in suitable classes of three time-scale systems. \SJJ{We focused on a detailed presentation of the proof in the semi-oscillatory cases $\alpha \in \{1,2\}$. The proof for $\alpha \geq 3$ can be obtained either via straightforward adaptations of the proof for $\alpha = 2$, or via straightforward adaptations of the proof of the established results for two time-scale systems in \cite{Krupa2001a,Mishchenko1975,Szmolyan2004}. The primary complication in cases $\alpha \in \{1,2\}$ stems from the fact that the dynamics remain `global' in $\theta$ after blow-up near $S_0^{\textup{c}}$. Consequently, one has to consider a blow-up of the entire cycle (as opposed to a local blow-up about a point on $S_0^{\textup{c}}$). 
	We also demonstrated the applicability of our results by using them to derive detailed geometric and asymptotic information about the passage through a folded cycle singularity in a class of periodically forced Li\'enard equations; recall Section \ref{sub:applications_1} and in particular Theorem \ref{thm:lienard}. Finally, in Section \ref{sub:applications_2}, we considered a toy model for the study of tipping phenomena. Due to the periodic forcing in the fast equation, this model violates both of our primary Assumptions \ref{ass:1_new}-\ref{ass:2_new}. The partial geometric analysis presented herein provides reason to argue that cases of this kind can for the most part be analysed using classical two time-scale theory.}
	
	\
	
	\SJJ{In summary, the} current manuscript provides a precedent for the effectiveness of stationary methods including Fenichel-Tikhonov theory and the blow-up method in \SJJ{semi-oscillatory} multiple time-scale systems with at least three time-scales. 
	\SJJ{There are many interesting directions which one might take from here. Our analysis of the toy model for tipping considered in Section \ref{sub:applications_2} is incomplete, and motivates the analysis of a more general class of systems with regular folded cycle manifolds which violate our main Assumptions \ref{ass:1_new}-\ref{ass:2_new}. One might also consider} the applicability of the methods developed herein for other global cycle singularities for which there is no direct and lower dimensional analogue in stationary fast-slow systems, e.g.~near a non-hyperbolic flip/period-doubling type cycle in three time-scale systems of the form \eqref{eq:original_system}. \SJJJ{We also expect the detailed estimates on the local dynamics near regular folded cycles in Theorem \ref{thm:main} to play an important role in the description of multi-scale oscillations which involve regular folded cycles as a key ingredient. In \cite{Szmolyan2004}, the authors derive a detailed characterisation of the local dynamics near a regular fold curve using geometric blow-up, and combine it with Fenichel theory in order to derive the form of the Poincar\'e map associated to prototypical 1-fast 2-slow systems with an S-shaped critical manifold. These results were used to prove the existence of an invariant torus for a certain parameter regime in the forced van der Pol equation, which contains an attracting relaxation oscillation if a particular circle map induced on the angular variable is contracting. We conjecture that similar constructions could be applied to the three time-scale semi-oscillatory systems considered herein, under suitable assumptions on the global geometry. The details of these} and other related problems are left for future work. \\

	\noindent\textbf{Acknowledgements.} SJ, \SV{CK and SVK} acknowledge funding from the SFB/TRR 109 Discretization and Geometry in Dynamics\SJJ{, i.e.~Deutsche Forschungsgemeinschaft (DFG - German
	Research Foundation) - Project-ID 195170736 - TRR109}. CK thanks the VolkswagenStiftung for support via a Lichtenberg Professorship. SVK also acknowledges funding from the DFG priority program SPP 2298 Theoretical Foundations of Deep Learning as well as the Munich Data Science Institute. \SJJ{SJ, \SV{CK and SVK} would also like to thank an anonymous referee for valuable feedback and suggestions, which led to significant improvements in the article. \\}

	\noindent\textbf{Data Availability.} No datasets were generated or analyzed during the current study. \\
	
	\noindent\textbf{Conflict of Interest.} The authors have no competing interests.

	
	\bibliographystyle{siam}
	\bibliography{lc_blowup.bib}

\end{document}